\documentclass[11pt,twoside]{article}
\usepackage[nohyperref,abbrvbib]{jmlr2e}
\usepackage{ifpdf} % \ifpdf already defined somewhere
\usepackage{amssymb,amsmath}
\usepackage{graphicx}
\usepackage{psfrag}
\usepackage{enumitem}
\usepackage{newtxtext}
\usepackage[smallerops]{newtxmath}
\usepackage{algorithm}
\usepackage{algorithmic}

\algsetup{indent=1.5em}

\usepackage{tikz}

\newtheorem{assumption}{Assumption}

\newcommand{\minimize}{\mathop{\mathrm{minimize}{}}}

\newcommand{\argmin}{\mathop{\mathrm{arg\,min}{}}}

\newcommand{\prox}{\mathbf{prox}}

\newcommand{\eqdef}{\stackrel{\triangle}{=}}

\newcommand{\ones}{\mathbf{1}}
\newcommand{\R}{\mathbf{R}}

\newcommand{\E}{\mathbf{E}}
\newcommand{\Prob}{\mathbf{P}}

\newcommand{\Bdiv}{\mathcal{B}}

\newcommand{\condbat}{\kappa_{\mathrm{bat}}}

\newcommand{\condrand}{\kappa_{\mathrm{rand}}}

\newcommand{\Lagr}{{L}}

\newcommand{\Ptld}{\widetilde{P}}
\newcommand{\Dtld}{\widetilde{D}}
\newcommand{\aopt}{\alpha^{\star}}
\newcommand{\aini}{\alpha^{(0)}}
\newcommand{\at}{\alpha^{(t)}}

\newcommand{\aM}{\alpha^{(M)}}
\newcommand{\atp}{\alpha^{(t+1)}}
\newcommand{\ai}{\alpha_i}
\newcommand{\aj}{\alpha_j}
\newcommand{\aiopt}{\alpha_i^{\star}}
\newcommand{\aiini}{\alpha_i^{(0)}}
\newcommand{\ait}{\alpha_i^{(t)}}
\newcommand{\aitp}{\alpha_i^{(t+1)}}
\newcommand{\ajt}{\alpha_j^{(t)}}

\newcommand{\aitld}{\tilde{\alpha}_{i}^{(t+1)}}
\newcommand{\aitldr}{\tilde{\alpha}_{i}^{(r)}}
\newcommand{\atld}{\tilde{\alpha}}
\newcommand{\atldopt}{\tilde{\alpha}^{\star}}
\newcommand{\atldini}{\tilde{\alpha}^{(0)}}
\newcommand{\atldr}{\tilde{\alpha}^{(r)}}
\newcommand{\atldrp}{\tilde{\alpha}^{(r+1)}}
\newcommand{\abs}{\bar{\alpha}^{(s)}}
\newcommand{\abS}{\bar{\alpha}^{(S)}}
\newcommand{\absp}{\bar{\alpha}^{(s+1)}}
\newcommand{\abiini}{\bar{\alpha}_{i}^{(0)}}
\newcommand{\abis}{\bar{\alpha}_{i}^{(s)}}
\newcommand{\abjs}{\bar{\alpha}_{j}^{(s)}}
\newcommand{\abini}{\bar{\alpha}^{(0)}}
\newcommand{\bit}{\beta_i^{(t)}}
\newcommand{\bitp}{\beta_i^{(t+1)}}
\newcommand{\biini}{\beta_i^{(0)}}
\newcommand{\jt}{{j^{(t)}}}

\newcommand{\jtps}{{j^{(t+s)}}}
\newcommand{\lt}{{l^{(t)}}}

\newcommand{\ltps}{{l^{(t+s)}}}
\newcommand{\uit}{{u_i^{(t)}}}
\newcommand{\uitp}{{u_i^{(t+1)}}}
\newcommand{\ujtp}{{u_j^{(t+1)}}}
\newcommand{\ubs}{\bar{u}^{(s)}}
\newcommand{\ubt}{\bar{u}^{(t)}}
\newcommand{\ubini}{\bar{u}^{(0)}}
\newcommand{\ubis}{\bar{u}_i^{(s)}}
\newcommand{\ubjs}{\bar{u}_j^{(s)}}
\newcommand{\ubit}{\bar{u}_i^{(t)}}
\newcommand{\ubitp}{\bar{u}_i^{(t+1)}}
\newcommand{\ubjt}{\bar{u}_j^{(t)}}
\newcommand{\ubjtp}{\bar{u}_j^{(t+1)}}
\newcommand{\vkt}{{v_k^{(t)}}}
\newcommand{\vktp}{{v_k^{(t+1)}}}
\newcommand{\vltp}{{v_l^{(t+1)}}}
\newcommand{\vbs}{\bar{v}^{(s)}}
\newcommand{\vbt}{\bar{v}^{(t)}}
\newcommand{\vbini}{\bar{v}^{(0)}}
\newcommand{\vbks}{\bar{v}_k^{(s)}}
\newcommand{\vbkt}{\bar{v}_k^{(t)}}
\newcommand{\vbktp}{\bar{v}_k^{(t+1)}}
\newcommand{\vbls}{\bar{v}_l^{(s)}}
\newcommand{\vblt}{\bar{v}_l^{(t)}}
\newcommand{\vbltp}{\bar{v}_l^{(t+1)}}
\newcommand{\wk}{{w_k}}
\newcommand{\wl}{{w_l}}
\newcommand{\wopt}{{w^{\star}}}
\newcommand{\wini}{{w^{(0)}}}
\newcommand{\wt}{{w^{(t)}}}

\newcommand{\wM}{{w^{(M)}}}
\newcommand{\wtp}{{w^{(t+1)}}}
\newcommand{\wkopt}{{w_k^{\star}}}
\newcommand{\wkini}{{w_k^{(0)}}}
\newcommand{\wkt}{{w_k^{(t)}}}
\newcommand{\wktp}{{w_k^{(t+1)}}}
\newcommand{\wlt}{{w_l^{(t)}}}
\newcommand{\wltp}{{w_l^{(t+1)}}}
\newcommand{\wktld}{\tilde{w}_{k}^{(t+1)}}
\newcommand{\wktldr}{\tilde{w}_{k}^{(r)}}
\newcommand{\wtld}{\tilde{w}}
\newcommand{\wtldopt}{\tilde{w}^{\star}}
\newcommand{\wtldini}{\tilde{w}^{(0)}}
\newcommand{\wtldr}{\tilde{w}^{(r)}}
\newcommand{\wtldrp}{\tilde{w}^{(r+1)}}
\newcommand{\wbs}{\bar{w}^{(s)}}
\newcommand{\wbS}{\bar{w}^{(S)}}
\newcommand{\wbsp}{\bar{w}^{(s+1)}}
\newcommand{\wbkini}{\bar{w}_k^{(0)}}
\newcommand{\wbks}{\bar{w}_k^{(s)}}
\newcommand{\wbls}{\bar{w}_l^{(s)}}
\newcommand{\wbini}{\bar{w}^{(0)}}
\newcommand{\ztld}{\tilde{z}}
\newcommand{\zopt}{{z}^{\star}}

\newcommand{\zit}{{z_i^{(t)}}}
\newcommand{\zt}{{z}^{(t)}}

\newcommand{\ztldini}{\tilde{z}^{(0)}}
\newcommand{\ztldr}{\tilde{z}^{(r)}}
\newcommand{\ztldrp}{\tilde{z}^{(r+1)}}
\newcommand{\ztldopt}{\tilde{z}^{\star}}
\newcommand{\ztldoptr}{\tilde{z}^{\star(r)}}
\newcommand{\Dt}{\Delta^{(t)}}
\newcommand{\DT}{\Delta^{(T)}}
\newcommand{\DM}{\Delta^{(M)}}
\newcommand{\Dtp}{\Delta^{(t+1)}}
\newcommand{\Dini}{\Delta^{(0)}}
\newcommand{\Dbini}{\bar{\Delta}^{(0)}}
\newcommand{\Dbs}{\bar{\Delta}^{(s)}}
\newcommand{\Dbsp}{\bar{\Delta}^{(s+1)}}
\newcommand{\Xik}{{X_{ik}}}
\newcommand{\Xil}{{X_{il}}}
\newcommand{\Xjl}{{X_{jl}}}
\newcommand{\Xjk}{{X_{jk}}}
\newcommand{\Xri}{{X_{i:}}}
\newcommand{\Xrj}{{X_{j:}}}
\newcommand{\Xck}{{X_{:k}}}
\newcommand{\Xcl}{{X_{:l}}}
\newcommand{\At}{{A^{(t)}}}
\newcommand{\Aini}{{A^{(0)}}}
\newcommand{\Aikini}{{A_{ik}^{(0)}}}
\newcommand{\Aikt}{{A_{ik}^{(t)}}}

\newcommand{\Aiktp}{{A_{ik}^{(t+1)}}}

\newcommand{\Ut}{{U^{(t)}}}
\newcommand{\Uikini}{{U_{ik}^{(0)}}}
\newcommand{\Uikt}{{U_{ik}^{(t)}}}
\newcommand{\Ujkt}{{U_{jk}^{(t)}}}
\newcommand{\Uilt}{{U_{il}^{(t)}}}
\newcommand{\Uiktp}{{U_{ik}^{(t+1)}}}
\newcommand{\Ujlt}{{U_{jl}^{(t)}}}

\newcommand{\Vt}{{V^{(t)}}}

\newcommand{\Vikini}{{V_{ik}^{(0)}}}
\newcommand{\Vikt}{{V_{ik}^{(t)}}}
\newcommand{\Vilt}{{V_{il}^{(t)}}}
\newcommand{\Vjkt}{{V_{jk}^{(t)}}}
\newcommand{\Viktp}{{V_{ik}^{(t+1)}}}
\newcommand{\Vjlt}{{V_{jl}^{(t)}}}
\newcommand{\Wt}{{W^{(t)}}}
\newcommand{\Wini}{{W^{(0)}}}
\newcommand{\Wikini}{{W_{ik}^{(0)}}}
\newcommand{\Wikt}{{W_{ik}^{(t)}}}

\newcommand{\Wiktp}{{W_{ik}^{(t+1)}}}
\newcommand{\Omegaini}{\Omega^{(0)}}

\newcommand{\Ifree}{\mathcal{S}_\mathrm{free}}

\title{DSCOVR: Randomized Primal-Dual Block Coordinate Algorithms\\
for Asynchronous Distributed Optimization}

\author{%
\name Lin Xiao
\email lin.xiao@microsoft.com\\
\addr  Microsoft Research AI\\
       Redmond, WA 98052, USA
\AND
\name Adams Wei Yu
\email weiyu@cs.cmu.edu\\
\addr  Machine Learning Department, Carnegie Mellon University\\
       Pittsburgh, PA 15213, USA
\AND
\name Qihang Lin
\email qihang-lin@uiowa.edu\\
\addr Tippie College of Business, The University of Iowa\\
      Iowa City, IA 52245, USA
\AND
\name Weizhu Chen
\email wzchen@microsoft.com\\
\addr Microsoft AI and Research\\
      Redmond, WA 98052, USA
\vspace{-4ex}
}

\editor{}

\begin{document}
\maketitle

\begin{abstract}
Machine learning with big data often involves large optimization models.
For distributed optimization over a cluster of machines,
frequent communication and synchronization of all model parameters 
(optimization variables) can be very costly.
A promising solution is to use parameter servers to store different subsets 
of the model parameters, and update them asynchronously at different machines 
using local datasets.
In this paper, we focus on distributed optimization of large linear models 
with convex loss functions, and propose a family of randomized primal-dual 
block coordinate algorithms that are especially suitable for asynchronous
distributed implementation with parameter servers.  
In particular, we work with the saddle-point formulation of such problems
which allows simultaneous data and model partitioning, and exploit its 
structure by doubly stochastic coordinate optimization with variance reduction 
(DSCOVR).
Compared with other first-order distributed algorithms, we show that DSCOVR 
may require less amount of overall computation and communication,
and less or no synchronization.
We discuss the implementation details of the DSCOVR algorithms, 
and present numerical experiments on an industrial distributed computing system.
\end{abstract}


\section{Introduction}
\label{sec:intro}

Algorithms and systems for distributed optimization are critical for solving
large-scale machine learning problems, especially when the dataset cannot fit 
into the memory or storage of a single machine.
In this paper, we consider distributed optimization problems of the form
\begin{equation}\label{eqn:erm-m}
\minimize_{w\in\R^d} ~~\frac{1}{m}\sum_{i=1}^m f_i(X_i w) + g(w),
\end{equation}
where $X_i\in\R^{N_i\times d}$ is the local data stored at the $i$th machine,
$f_i:\R^{N_i}\to\R$ is a convex cost function associated with the linear 
mapping $X_i w$, and $g(w)$ is a convex regularization function.
In addition, we assume that $g$ is separable, i.e., for some integer $n>0$, 
we can write
\begin{equation}\label{eqn:g-separable}
  g(w) = \sum_{k=1}^n g_k(w_k)\,,
\end{equation}
where $g_k:\R^{d_k}\to\R$, and $w_k\in\R^{d_k}$ for $k=1,\ldots,n$ 
are non-overlapping subvectors of $w\in\R^d$ with $\sum_{k=1}^n d_k = d$
(they form a partition of~$w$).
Many popular regularization functions in machine learning are separable,
for example, $g(w)=(\lambda/2)\|w\|_2^2$ or $g(w)=\lambda\|w\|_1$ 
for some $\lambda>0$.

An important special case of~\eqref{eqn:erm-m} is distributed empirical 
risk minimization (ERM) of linear predictors.  
Let $(x_1,y_1),\ldots, (x_N, y_N)$ be~$N$ training examples, where
each $x_j\in\R^d$ is a feature vector and $y_j\in\R$ is its label.
The ERM problem is formulated as
\begin{equation}\label{eqn:erm-N}
\minimize_{w\in\R^d}~~ \frac{1}{N}\sum_{j=1}^N \phi_j\bigl(x_j^T w\bigr)+g(w),
\end{equation}
where each $\phi_j:\R\to\R$ is a loss function measuring the mismatch between 
the linear prediction $x_j^T w$ and the label~$y_j$.
Popular loss functions in machine learning include, e.g., 
for regression, the squared loss $\phi_j(t)=(1/2)(t-y_j)^2$, 
and for classification, the logistic loss $\phi_j(t)=\log(1+\exp(-y_j t))$
where $y_j\in\{\pm 1\}$.
In the distributed optimization setting, the~$N$ examples are divided 
into~$m$ subsets, each stored on a different machine.
For $i=1,\ldots,m$, let $\mathcal{I}_i$ denote the subset of
$\bigl\{1,\ldots,N\bigr\}$ stored at machine~$i$ and let $N_i=|\mathcal{I}_i|$
(they satisfy $\sum_{i=1}^m N_i=N$).
Then the ERM problem~\eqref{eqn:erm-N} can be written in the form
of~\eqref{eqn:erm-m} by letting $X_i$ consist of $x_j^T$ 
with $j\in\mathcal{I}_i$ as its rows and defining
$f_i:\R^{N_i}\to\R$ as
\begin{equation}\label{eqn:erm-fi}
  f_i(u_{\mathcal{I}_i}) = \frac{m}{N}\sum_{j\in\mathcal{I}_i}\phi_j(u_j),
\end{equation}
where $u_{\mathcal{I}_i}\in\R^{N_i}$ is a subvector of $u\in\R^N$,
consisting of $u_j$ with $j\in\mathcal{I}_i$.
%To simplify notation, we use $u_i$ to denote $u_{\mathcal{I}_i}$ 
%whenever~$i$ is the same index as in $X_i$.

%\subsection{Communication Model and Performance Measures} 
%\label{sec:model-measure}

The nature of distributed algorithms and their convergence properties
largely depend on the model of the communication network that connects 
the~$m$ computing machines. 
A popular setting in the literature is to model the communication network 
as a graph, and each node can only communicate (in one step) with their
neighbors connected by an edge, either synchronously or asynchronously
\citep[e.g.,][]{BertsekasTsitsiklis89,NedicOzdaglar2009}.
The convergence rates of distributed algorithms in this setting often 
depend on characteristics of the graph, such as its diameter and the 
eigenvalues of the graph Laplacian
\citep[e.g.][]{XiaoBoyd2006jota,DuchiAgarwal2012,NedicShi2016,ScamanBach2017}.
This is often called the \emph{decentralized} setting. 
%for distributed optimization.

Another model for the communication network is \emph{centralized}, 
where all the machines participate synchronous, collective communication,
e.g., broadcasting a vector to all~$m$ machines, 
or computing the sum of~$m$ vectors, each from a different machine
(AllReduce).
These collective communication protocols hide the underlying implementation 
details, which often involve operations on graphs. 
They are adopted by many popular distributed computing standards and packages, 
such as MPI \citep{MPIForum}, MapReduce \citep{Dean08MapReduce} 
and Aparche Spark \citep{ApacheSpark2016},
and are widely used in machine learning practice 
\citep[e.g.,][]{LinTsaiLeeLin14,MLlib2016}.
In particular, collective communications are very useful for addressing 
\emph{data parallelism}, i.e., by allowing different machines to work 
in parallel to improve the same model $w\in\R^d$ using their local dataset.
A disadvantage of collective communications is their synchronization cost:
faster machines or machines with less computing tasks have to become idle
while waiting for other machines to finish their tasks 
in order to participate a collective communication.

One effective approach for reducing synchronization cost is to exploit 
\emph{model parallelism} (here ``model'' refers to $w\in\R^d$, 
including all optimization variables).
The idea is to allow different machines work in parallel with different 
versions of the full model or different parts of a common model, 
with little or no synchronization.
The model partitioning approach can be very effective for solving
problems with large models (large dimension~$d$).
Dedicated \emph{parameter servers} can be set up to store and maintain 
different subsets of the model parameters, 
such as the $w_k$'s in~\eqref{eqn:g-separable}, 
and be responsible for coordinating their updates at different workers
\citep{LiMu2014OSDI,Xing2015Petuum}.
This requires flexible point-to-point communication.
%Model parallelism requires flexible point-to-point communication 
%\wei{This sentence seems disconnected}.

In this paper, we develop a family of randomized algorithms that exploit
\emph{simultaneous} data and model parallelism.
Correspondingly, we adopt a centralized communication model that support
both synchronous collective communication and 
asynchronous point-to-point communication.
In particular, it allows any pair of machines to send/receive a message 
in a single step, and multiple point-to-point communications
may happen in parallel in an event-driven, asynchronous manner. 
Such a communication model is well supported by the MPI standard.
To evaluate the performance of distributed algorithms in this setting, 
we consider the following three measures.
\begin{itemize} \itemsep 0pt
\item \emph{Computation complexity:} 
total amount of computation, measured by the number of passes over all 
datasets $X_i$ for $i=1,\ldots,m$, 
which can happen in parallel on different machines.
%If the data $X_i$'s are stored as sparse matrices, then each pass only
%processes the nonzero elements.
\item \emph{Communication complexity:}
the total amount of communication required, 
measured by the equivalent number of vectors in $\R^d$ sent or received
across all machines.
\item \emph{Synchronous communication:} 
measured by the total number of vectors in $\R^d$ that requires
synchronous collective communication involving all~$m$ machines.
We single it out from the overall communication
complexity as a (partial) measure of the synchronization cost.
\end{itemize}

In Section~\ref{sec:framework}, 
we introduce the framework of our randomized algorithms,
\textbf{D}oubly \textbf{S}tochastic \textbf{C}oordinate \textbf{O}ptimization with \textbf{V}ariance \textbf{R}eduction (DSCOVR),
and summarize our theoretical results on the three measures
achieved by DSCOVR.
Compared with other first-order methods for distributed optimization,
we show that DSCOVR may require less amount of overall computation and 
communication, and less or no synchronization.
Then we present the details of several DSCOVR variants and their convergence 
analysis in Sections~\ref{sec:dscovr-svrg}-\ref{sec:dual-free}.
We discuss the implementation of different DSCOVR algorithms
in Section~\ref{sec:implementation}, and present results of our numerical 
experiments in Section~\ref{sec:experiments}.

\section{The DSCOVR Framework and Main Results}
\label{sec:framework}

\begin{figure}[t]
  \centering
  \ifpdf
  \includegraphics[width=0.475\textwidth]{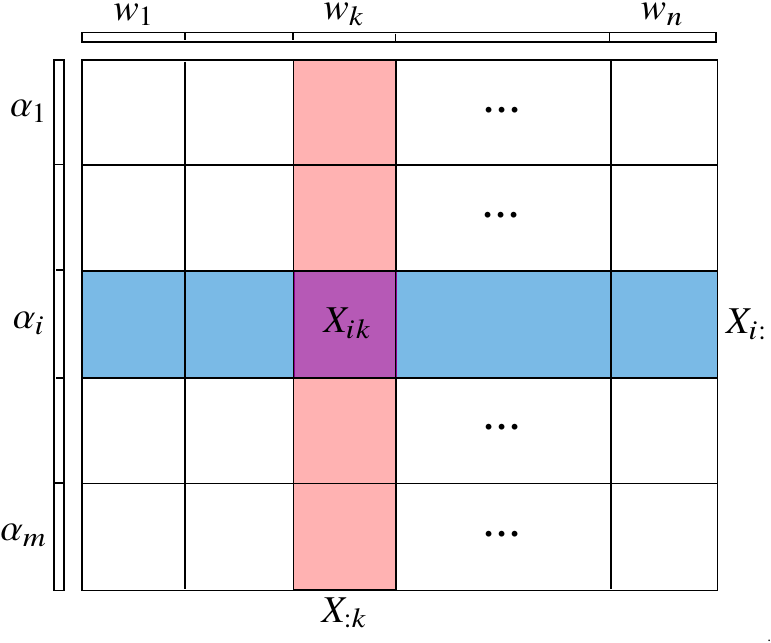}
  \else
  \input{drawings/Xpartition_psfrag}
  \includegraphics[width=0.45\textwidth]{drawings/Xpartition}
  \fi
  \caption{Partition of primal variable~$w$, dual variable~$\alpha$, 
  and the data matrix $X$.}
  \label{fig:partition}
\end{figure}

First, we derive a saddle-point formulation of the convex 
optimization problem~\eqref{eqn:erm-m}.
Let $f_i^*$ be the convex conjugate of $f_i$, i.e., 
$f_i^*(\alpha_i)=\sup_{u_i\in\R^{N_i}}\left\{\alpha_i^T u_i-f_i(u_i)\right\}$,
and define 
\begin{equation}\label{eqn:Lagrangian}
\Lagr(w,\alpha) ~\equiv~ \frac{1}{m}\sum_{i=1}^{m}\alpha_i^T X_i w
-\frac{1}{m}\sum_{i=1}^{m}f_i^*(\alpha_i)+  g(w) \,,
\end{equation}
where $\alpha=[\alpha_1;\ldots;\alpha_m]\in\R^{N}$.
Since both the $f_i$'s and $g$ are convex, $\Lagr(w,\alpha)$ 
is convex in~$w$ and concave in~$\alpha$.
We also define a pair of \emph{primal} and \emph{dual} functions:
\begin{eqnarray}
P(w) &=& \max_{\alpha\in\R^N} \Lagr(w,\alpha)
~=~ \frac{1}{m}\sum_{i=1}^m f_i(X_i w) + g(w)\,,
  \label{eqn:erm-primal} \\
D(\alpha) &=& \min_{w\in\R^d} \Lagr(w,\alpha)
~=~ -\frac{1}{m}\sum_{i=1}^m f_i^*(\alpha_i) 
    - g^*\biggl(-\frac{1}{m}\sum_{i=1}^m(X_i)^T\alpha_i\biggr)\,,
\label{eqn:erm-dual}
\end{eqnarray}
where $P(w)$ is exactly the objective function in~\eqref{eqn:erm-m}\footnote{
More technically, we need to assume that each $f_i$ is convex and lower
semi-continuous so that $f_i^{**}=f_i$ 
\citep[see, e.g.,][Section~12]{Rockafellar70}.
It automatically holds if $f_i$ is convex and differentiable, which we will
assume later.}
and $g^*$ is the convex conjugate of~$g$.
We assume that $L$ has a saddle point $(\wopt,\aopt)$, that is,
\[
\Lagr(\wopt,\alpha)\leq \Lagr(\wopt,\aopt) \leq \Lagr(w,\aopt)\,,
\qquad \forall (w,\alpha)\in\R^d\times\R^N .
\]
In this case, we have
$\wopt=\argmin~P(w)$ and $\aopt=\argmin~D(\alpha)$, and $P(\wopt)=D(\aopt)$.

The DSCOVR framework is based on solving the convex-concave saddle-point 
problem
\begin{equation}\label{eqn:min-max-saddle}
\min_{w\in\R^d}~\max_{\alpha\in\R^N}~  \Lagr(w,\alpha) .
\end{equation}
Since we assume that~$g$ has a separable structure as 
in~\eqref{eqn:g-separable}, we rewrite the saddle-point problem as
\begin{equation}\label{eqn:separable-saddle}
\min_{w\in\R^d}~\max_{\alpha\in\R^N}~ \biggl\{
\frac{1}{m}\sum_{i=1}^{m}\sum_{k=1}^n
\alpha_i^T \Xik w_k
-\frac{1}{m}\sum_{i=1}^{m}f_i^*(\alpha_i)+  \sum_{k=1}^n g_k(w_k) \biggr\}\,,
\end{equation}
where $\Xik\in\R^{N_i\times d_k}$ for $k=1,\ldots,n$ are column partitions 
of~$X_i$.
For convenience, we define the following notations.
First, let $X=[X_1;\ldots;X_m]\in\R^{N\times d}$ be the overall data matrix,
by stacking the $X_i$'s vertically. 
Conforming to the separation of~$g$,
we also partition~$X$ into block columns $\Xck\in\R^{N\times d_k}$
for $k=1,\ldots,n$, where each $\Xck=[X_{1k};\ldots;X_{mk}]$ 
(stacked vertically).
For consistency, we also use $\Xri$ to denote $X_i$ from now on. 
See Figure~\ref{fig:partition} for an illustration.

\begin{algorithm}[t]
\caption{DSCOVR framework}
\label{alg:dscovr}
\begin{algorithmic}[1]
\REQUIRE initial points $\wini,\aini$, and step sizes $\sigma_i$ 
	for $i=1,\ldots,m$ and $\tau_k$ for $k=1,\ldots,n$.
\FOR {$t=0,1,2,\ldots,$}
\STATE pick $j\in\{1,\ldots,m\}$ and $l\in\{1,\ldots,n\}$ 
randomly with distributions $p$ and $q$ respectively.\\
\STATE compute \emph{variance-reduced} stochastic gradients $\ujtp$ and $\vltp$.
\STATE update primal and dual block coordinates:
\vspace{-1ex}
\begin{eqnarray}
\aitp &=& \left\{	\begin{array}{ll}
    \prox_{\sigma_j f_j^*}\bigl(\ajt+\sigma_j\ujtp\bigr) &\textrm{if}~i=j,\\
    \ait,  &\textrm{if}~i\neq j, \end{array} \right.
\label{eqn:ai-update} \\
\wktp &=& \left\{ \begin{array}{ll}
    \prox_{\tau_l g_l}\bigl(\wlt-\tau_l\vltp\bigr) & \text{if}~ k=l,\\
    \wkt, & \textrm{if}~k\neq l. \end{array} \right. 
\label{eqn:wk-update} 
\end{eqnarray}
\vspace{-1em}
\ENDFOR
\end{algorithmic}
\end{algorithm}

We exploit the doubly separable structure in~\eqref{eqn:separable-saddle} 
by a doubly stochastic coordinate update algorithm outlined in
Algorithm~\ref{alg:dscovr}.
Let $p=\{p_1,\ldots,p_m\}$ and $q=\{q_1,\ldots,q_n\}$ be two probability 
distributions.
During each iteration~$t$,
we randomly pick an index $j\in\{1,\ldots,m\}$ with probability~$p_j$, and 
independently pick an index $l\in\{1,\ldots,n\}$ with probability~$q_l$.
Then we compute two vectors $\ujtp\in\R^{N_j}$ and $\vltp\in\R^{d_l}$
(details to be discussed later), and use them to update the block coordinates 
$\alpha_j$ and $w_l$ while leaving other block coordinates unchanged.
The update formulas in~\eqref{eqn:ai-update} and~\eqref{eqn:wk-update}
use the proximal mappings of the (scaled) functions $f_j^*$ and $g_l$
respectively.
We recall that the proximal mapping for any convex function 
$\phi:\R^d\to\R\cup\{\infty\}$ is defined as
\[
\prox_{\phi}(v) \eqdef \argmin_{u\in\R^{d}}
\left\{\phi(u) + \frac{1}{2}\|u-v\|^2\right\}.
\]

There are several different ways to compute the vectors $\ujtp$ and $\vltp$
in Step~3 of Algorithm~\ref{alg:dscovr}.
They should be the partial gradients or stochastic gradients of the bilinear
coupling term in~$L(w,\alpha)$ with respect to $\alpha_j$ and~$w_l$
respectively.
Let
\[
  K(w,\alpha) = \alpha^T X w = \sum_{i=1}^{m}\sum_{k=1}^n \alpha_i^T \Xik w_k ,
\]
which is the bilinear term in $L(w,\alpha)$ without the factor $1/m$.
We can use the following partial gradients in Step~3:
\begin{equation}
\begin{array}{l}
\displaystyle \ubjtp = \frac{\partial K(\wt,\at)}{\partial \alpha_j}
= \sum_{k=1}^n \Xjk\wkt, \\[2ex]
\displaystyle \vbltp  = \frac{1}{m} \frac{\partial K(\wt,\at)}{\partial w_l}
= \frac{1}{m}\sum_{i=1}^m (\Xil)^T\ait . 
\end{array}
\label{eqn:partial-grad-aj-wl}
\end{equation}
We note that the factor $1/m$ does not appear in the first equation
because it multiplies both $K(w,\alpha)$ and $f_j^*(\alpha_j)$ 
in~\eqref{eqn:separable-saddle} 
and hence does not appear in updating~$\alpha_j$.
Another choice is to use 
\begin{equation}
\begin{array}{l}
  \displaystyle \ujtp = \frac{1}{q_l}\Xjl\wlt, \\[2ex]
\displaystyle \vltp = \frac{1}{p_j}\frac{1}{m}(\Xjl)^T\ajt, 
\end{array}
\label{eqn:stoch-grad-aj-wl}
\end{equation}
which are unbiased stochastic partial gradients, because
\begin{eqnarray*}
&& \E_l\bigl[\ujtp\bigr] 
= \sum_{k=1}^n q_k\frac{1}{q_k} \Xjk\wkt = \sum_{k=1}^n \Xjk\wkt
= \ubjtp, \\
%= \frac{\partial K(\wt,\at)}{\partial \alpha_j}, \\
&& \E_j\bigl[\vltp\bigr]  
= \sum_{i=1}^m p_i\frac{1}{p_i}\frac{1}{m}(\Xil)^T\ait 
= \frac{1}{m}\sum_{i=1}^m (\Xil)^T\ait 
%= \frac{1}{m} \frac{\partial K(\wt,\at)}{\partial w_l},
= \vbltp,
\end{eqnarray*}
where $\E_j$ and $\E_l$ are expectations with respect to the random 
indices~$j$ and~$l$ respectively.

It can be shown that, Algorithm~\ref{alg:dscovr} converges to a saddle point
of $L(w,\alpha)$ with either choice~\eqref{eqn:partial-grad-aj-wl}
or~\eqref{eqn:stoch-grad-aj-wl} in Step~3,
and with suitable \emph{step sizes} $\sigma_i$ and $\tau_k$.
It is expected that using the stochastic gradients 
in~\eqref{eqn:stoch-grad-aj-wl} leads to a slower convergence rate
than applying~\eqref{eqn:partial-grad-aj-wl}.
However, using~\eqref{eqn:stoch-grad-aj-wl} has the advantage of much less 
computation during each iteration.
Specifically, it employs only one block matrix-vector multiplication for 
both updates, instead of~$n$ and~$m$ block multiplications done
in~\eqref{eqn:partial-grad-aj-wl}.

More importantly, the choice in~\eqref{eqn:stoch-grad-aj-wl} is suitable
for parallel and distributed computing.
To see this, 
let $(\jt,\lt)$ denote the pair of random indices drawn at iteration~$t$
(we omit the superscript~$(t)$ to simplify notation 
whenever there is no confusion from the context).
Suppose for a sequence of consecutive iterations $t,\ldots,t+s$,
there is no common index among $\jt,\ldots,\jtps$,
nor among $\lt,\ldots,\ltps$,
then these $s+1$ iterations can be done in parallel and they produce 
the same updates as being done sequentially.
Suppose there are $s+1$ processors or machines, then each can
carry out one iteration, which includes the updates 
in~\eqref{eqn:stoch-grad-aj-wl} as well as~\eqref{eqn:ai-update}
and~\eqref{eqn:wk-update}.
These $s+1$ iterations are independent of each other, and in fact can be done
in any order, because each only involve one primal block $w_{\lt}$ and 
one dual block $\alpha_{\jt}$, for both input and output
(variables on the right and left sides of the assignments respectively).
In contrast, the input for the updates in~\eqref{eqn:partial-grad-aj-wl}
depend on all primal and dual blocks at the previous iteration,
thus cannot be done in parallel.

\begin{figure}[t]
  \centering
  \ifpdf
  \includegraphics[width=0.95\textwidth]{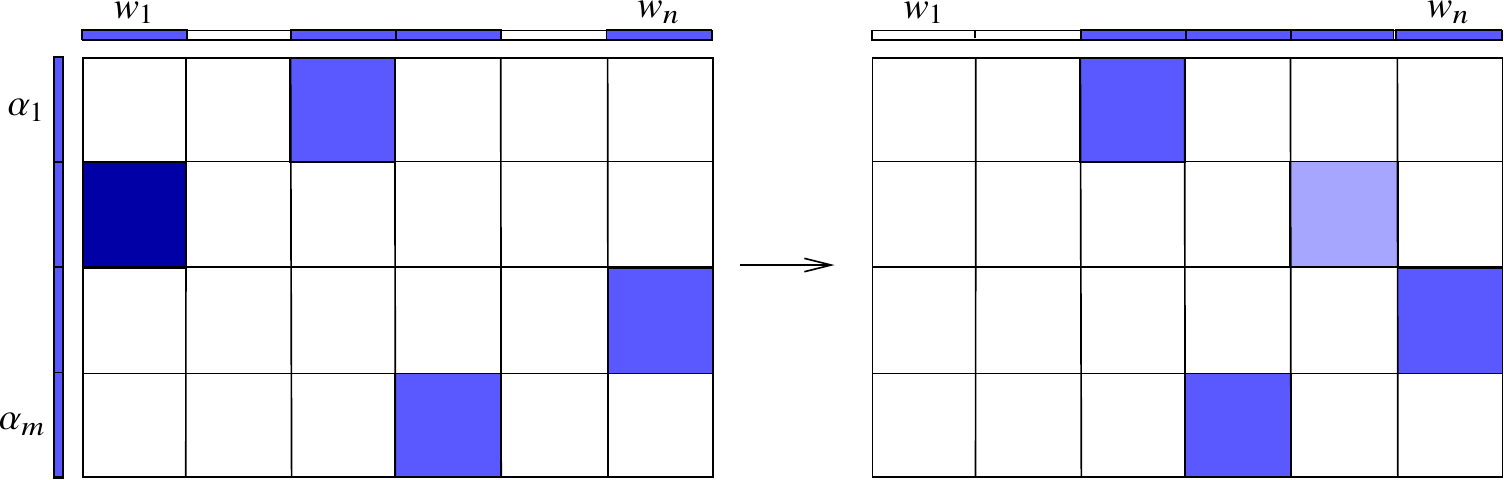}
  \else
  \input{drawings/Xpartition_psfrag}
  \includegraphics[width=0.9\textwidth]{drawings/separation1step}
  \fi
  \caption{Simultaneous data and model parallelism.
  At any given time, each machine is busy updating one parameter block
  and its own dual variable. Whenever some machine is done,
  it is assigned to work on a random block that is not being updated.}
  \label{fig:simul-parallel}
\end{figure}

\label{page:impl-sketch}
In practice, suppose we have $m$ machines for solving 
problem~\eqref{eqn:separable-saddle}, and each holds the data matrix
$\Xri$ in memory and maintains the dual block $\alpha_i$, for $i=1,\ldots,m$. 
We assume that the number of model partitions~$n$ is larger than~$m$, 
and the~$n$ model blocks $\{w_1,\ldots,w_n\}$ are stored at one or more
parameter servers.
In the beginning, we can randomly pick $m$ model blocks 
(sampling without replacement) from $\{w_1,\ldots,w_n\}$, and assign
each machine to update one of them.
If machine~$i$ is assigned to update block $k$, then both $\alpha_i$
and $w_k$ are updated, using only the matrix~$\Xik$;
moreover, it needs to communicate only the block~$w_k$ with the parameter
server that are responsible to maintain it.
Whenever one machine finishes its update, a scheduler can randomly pick
another parameter block that is not currently updated by other machines, 
and assign it to the free machine.
Therefore all machines can work in parallel, in an asynchronous, event-driven
manner. 
Here an event is the completion of a block update at any machine,
as illustrated in Figure~\ref{fig:simul-parallel}.
We will discuss the implementation details in
Section~\ref{sec:implementation}.

The idea of using doubly stochastic updates for distributed optimization 
in not new. It has been studied by \citet{NOMAD2014} for solving the 
matrix completion problem, and by \citet{DSO2014} for solving the saddle-point
formulation of the ERM problem.
Despite their nice features for parallelization, these algorithms inherit 
the $O(1/\sqrt{t})$ (or $O(1/t)$ with strong convexity) sublinear
convergence rate of the classical stochastic gradient method.
They translate into high communication and computation cost for
distributed optimization.
In this paper, we propose new variants of doubly stochastic update algorithms
by using \emph{variance-reduced} stochastic gradients 
(Step~3 of Algorithm~\ref{alg:dscovr}).
More specifically, we borrow the variance-reduction techniques from
SVRG \citep{JohnsonZhang13} and SAGA \citep{NIPS2014SAGA}
to develop the DSCOVR algorithms, which enjoy fast linear rates of convergence.
In the rest of this section, we summarize our theoretical results 
characterizing the three measures for DSCOVR:
computation complexity, communication complexity, and synchronization cost.
We compare them with distributed implementation of 
batch first-order algorithms.

\subsection{Summary of Main Results}
\label{sec:main-results}

Throughout this paper, we use $\|\cdot\|$ to denote the standard Euclidean
norm for vectors. 
For matrices, $\|\cdot\|$ denotes the operator (spectral) norm 
and $\|\cdot\|_F$ denotes the Frobenius norm.
We make the following assumption regarding the optimization 
problem~\eqref{eqn:erm-m}.
\begin{assumption}\label{asmp:convexity}
Each $f_i$ is convex and differentiable, and its gradient is 
$(1/\gamma_i)$-Lipschitz continuous, i.e., 
\begin{equation}\label{eqn:fi-smooth}
  \|\nabla f_i(u) -\nabla f_i(v)\| \leq \frac{1}{\gamma_i}\|u-v\|,
    \quad \forall\, u,v \in\R^{N_i}, \quad i=1,\ldots,m.
  \end{equation}
In addition, the regularization function~$g$ is $\lambda$-strongly convex, 
i.e., 
\[
    g(w') \geq g(w) + \xi^T(w'-w) + \frac{\lambda}{2}\|w'-w\|^2,
    \quad \forall\,\xi\in\partial g(w), \quad w', w\in\R^d.
\]
\end{assumption}
Under Assumption~\ref{asmp:convexity}, each $f_i^*$
is $\gamma_i$-strongly convex \citep[see, e.g.,][Theorem~4.2.2]{HUL01book}, 
and $L(w,\alpha)$~defined in~\eqref{eqn:Lagrangian} 
has a unique saddle point $(\wopt,\aopt)$.

The condition~\eqref{eqn:fi-smooth} is often referred to as 
$f_i$ being $1/\gamma_i$-smooth.
To simplify discussion, here we assume $\gamma_i=\gamma$ for $i=1,\ldots,m$.
Under these assumptions, each composite function $f_i(X_i w)$ has a 
smoothness parameter $\|X_i\|^2/\gamma$
(upper bound on the largest eigenvalue of its Hessian).
Their average $(1/m)\sum_{i=1}^m f_i(X_i w)$ has a smooth parameter 
$\|X\|^2/(m\gamma)$, which no larger than the average of the 
individual smooth parameters $(1/m)\sum_{i=1}^m \|X_i\|^2/\gamma$.
We define a condition number for problem~\eqref{eqn:erm-m} as the ratio
between this smooth parameter and the convexity parameter~$\lambda$ of~$g$:
\begin{equation}\label{eqn:cond-bat}
\condbat = \frac{\|X\|^2}{m\lambda\gamma} 
\leq \frac{1}{m}\sum_{i=1}^m \frac{\|\Xri\|^2}{\lambda\gamma} 
\leq \frac{\|X\|_{\max}^2}{\lambda\gamma},
\end{equation}
where $\|X\|_{\max} = \max_i\{\|\Xri\|\}$.
This condition number is a key factor to characterize the iteration complexity 
of \emph{batch} first-order methods for solving problem~\eqref{eqn:erm-m},
i.e., minimizing $P(w)$.
Specifically, to find a~$w$ such that $P(w)-P(\wopt)\leq\epsilon$,
the proximal gradient method requires
$O\left((1+\condbat)\log(1/\epsilon) \right)$ iterations, 
and their accelerated variants require
$O\left(\bigl(1+\sqrt{\condbat}\bigr)\log(1/\epsilon) \right)$ iterations
\citep[e.g.,][]{Nesterov04book,BeckTeboulle09,Nesterov13composite}.
Primal-dual first order methods for solving the saddle-point 
problem~\eqref{eqn:min-max-saddle} share the same complexity
\citep{ChambollePock2011,ChambollePock2015ergodic}.

A fundamental baseline for evaluating any distributed optimization algorithms
is the distributed implementation of batch first-order methods.
Let's consider solving problem~\eqref{eqn:erm-m} using the proximal
gradient method.
During every iteration~$t$, each machine receives a copy of $\wt\in\R^d$ 
from a master machine (through Broadcast), and 
computes the local gradient $\zit=X_i^T \nabla f_i(X_i\wt)\in\R^d$.
Then a collective communication is invoked to compute 
the batch gradient $\zt =(1/m)\sum_{i=1}^m \zit$ at the master (Reduce).
The master then takes a proximal gradient step,
using $\zt$ and the proximal mapping of~$g$,
to compute the next iterate~$\wtp$ and broadcast it to every machine for the
next iteration.
We can also use the AllReduce operation in MPI to obtain $\zt$ at each
machine without a master.
In either case, the total number of passes over the data is twice the
number of iterations 
(due to matrix-vector multiplications using both $X_i$ and $X_i^T$), 
and the number of vectors in $\R^d$ sent/received across all machines 
is~$2m$ times the number of iterations (see Table~\ref{tab:comm-comp}).
Moreover, all communications are collective and synchronous.

\begin{table}[t]
\renewcommand{\arraystretch}{1.5}
\setlength{\tabcolsep}{3pt}
  \centering
  \begin{tabular}{|c|c|c|}
   \hline
   Algorithms & Computation complexity & Communication complexity \\[-1ex]
   & (number of passes over data) & (number of vectors in $\R^d$) \\
   \hline 
   batch first-order methods 
   & $(1+\condbat)\log(1/\epsilon)$ 
   & $m(1+\condbat)\log(1/\epsilon)$  \\
   DSCOVR 
   & $\left(1+\condrand/m\right)\log(1/\epsilon)$ 
   & $(m+\condrand)\log(1/\epsilon)$ \\ 
   \hline
   accelerated batch first-order methods
   & $\left(1+\sqrt{\condbat}\right)\log(1/\epsilon)$ 
   & $m\left(1+\sqrt{\condbat}\right)\log(1/\epsilon)$ \\
   accelerated DSCOVR 
   & $\bigl(1+\sqrt{\condrand/m}\bigr)\log(1/\epsilon)$ 
   & $\left(m+\sqrt{m\!\cdot\!\condrand}\right)\log(1/\epsilon)$ \\
   \hline
  \end{tabular}
  \caption{Computation and communication complexities of
  batch first-order methods and DSCOVR (for both SVRG and SAGA variants).
  We omit the $O(\cdot)$ notation in all entries and an extra 
  $\log(1+\condrand/m)$ factor for accelerated DSCOVR algorithms.} 
  \label{tab:comm-comp}
\vspace{-1ex}
\end{table}

Since DSCOVR is a family of randomized algorithms for solving
the saddle-point problem~\eqref{eqn:min-max-saddle}, 
we would like to find $(w,\alpha)$ such that 
$\|\wt-\wopt\|^2+(1/m)\|\at-\aopt\|^2\leq\epsilon$ holds in expectation
and with high probability.
We list the communication and computation complexities of DSCOVR 
in Table~\ref{tab:comm-comp}, comparing them with 
batch first-order methods.
Similar guarantees also hold for reducing the duality gap
$P(\wt)-D(\at)$, where~$P$ and~$D$ are defined 
in~\eqref{eqn:erm-primal} and~\eqref{eqn:erm-dual} respectively.

The key quantity characterizing the complexities of DSCOVR is the condition 
number $\condrand$, which can be defined in several different ways.
If we pick the data block~$i$ and model block~$k$ with uniform distribution,
i.e., $p_i=1/m$ for $i=1,\ldots,m$ and $q_k=1/n$ for $k=1,\ldots,n$, then 
\begin{equation}\label{eqn:cond-rand-mn}
\condrand = \frac{n\|X\|^2_{m\times n}}{\lambda\gamma},
\qquad \mbox{where}\qquad \|X\|_{m\times n}=\max_{i,k}\|\Xik\|. 
\end{equation}
Comparing the definition of $\condbat$ in~\eqref{eqn:cond-bat}, we have
$\condbat\leq\condrand$ because
\[
\frac{1}{m}\|X\|^2 \leq \frac{1}{m}\sum_{i=1}^m \|X_i\|^2
\leq \frac{1}{m}\sum_{i=1}^m \sum_{k=1}^n \|\Xik\|^2 
\leq n\|X\|_{m\times n}^2 .
\]
With $\Xri= [X_{i1} \cdots X_{im}] \in\R^{N_i\times d}$ and
$\Xck= [X_{1k}; \ldots; X_{mk}] \in\R^{N\times d_k}$, we can also define
\begin{equation}\label{eqn:cond-rand-maxF}
\condrand' = \frac{\|X\|_{\max,F}^2}{\lambda\gamma}, 
\qquad \mbox{where}\qquad
\|X\|_{\max,F} = \max_{i,k}\bigl\{\|\Xri\|_F,\, \|\Xck\|_F\bigr\}.
\end{equation}
In this case, we also have $\condbat\leq\condrand'$ because 
$\|X\|_{\max}\leq\|X\|_{\max,F}$.
Finally, if we pick the pair $(i,k)$ with non-uniform distribution
$p_i=\|\Xri\|_F^2/\|X\|_F^2$ and $q_k=\|\Xck\|_F^2/\|X\|_F^2$, 
then we can define
\begin{equation}\label{eqn:cond-rand-Fro}
\condrand'' = \frac{\|X\|_F^2}{m\lambda\gamma}.
\end{equation}
Again we have $\condbat\leq\condrand''$ because $\|X\|\leq\|X\|_F$.
We may replace $\condrand$ in Tables~\ref{tab:comm-comp} 
and~\ref{tab:sync-async} by
either $\condrand'$ or $\condrand''$, depending on the 
probability distributions~$p$ and~$q$ and different proof techniques.

From Table~\ref{tab:comm-comp}, we observe similar type of speed-ups in
computation complexity, as obtained by variance reduction techniques 
over the batch first-order algorithms for convex optimization 
\citep[e.g.,][]{LeRouxSchmidtBach12,JohnsonZhang13, NIPS2014SAGA, 
XiaoZhang14ProxSVRG,LanZhou2015RPDG,Katyusha2017},
as well as for convex-concave saddle-point problems
\citep{ZhangXiao2017SPDC,BalamuruganBach2016}.
Basically, DSCOVR algorithms have potential improvement over batch first-order
methods by a factor of~$m$ (for non-accelerated algorithms)
or~$\sqrt{m}$ (for accelerated algorithms), but with a worse condition number.
In the worst case, the ratio between $\condrand$ and $\condbat$ 
may be of order~$m$ or larger, thus canceling the potential improvements.

\begin{table}[t]
\renewcommand{\arraystretch}{1.3}
\setlength{\tabcolsep}{3pt}
  \centering
  \begin{tabular}{|c|c|c|}
   \hline
   Algorithms & Synchronous Communication & Asynchronous Communication \\[-1ex]
   & (number of vectors in $\R^d$) & (equiv.\ number of vectors in $\R^d$)\\
   \hline 
   DSCOVR-SVRG
   & $m\log(1/\epsilon)$ 
   & $\condrand\log(1/\epsilon)$ \\
   DSCOVR-SAGA
   & $m$ 
   & $\left(m+\condrand\right)\log(1/\epsilon)$ \\
   \hline
   accelerated DSCOVR-SVRG
   & $m\log(1/\epsilon)$ 
   & $\left(1+\sqrt{m\!\cdot\!\condrand}\right)\log(1/\epsilon)$ \\
   accelerated DSCOVR-SAGA 
   & $m$
   & $\left(1+\sqrt{m\!\cdot\!\condrand}\right)\log(1/\epsilon)$ \\
   \hline
  \end{tabular}
  \caption{Breakdown of communication complexities into synchronous and 
    asynchronous communications for two 
    different types of DSCOVR algorithms.
  We omit the $O(\cdot)$ notation and an extra $\log(1+\condrand/m)$ factor
for accelerated DSCOVR algorithms.}
  \label{tab:sync-async}
\end{table}

More interestingly, DSCOVR also has similar improvements 
in terms of communication complexity over batch first-order methods.
In Table~\ref{tab:sync-async}, we decompose the communication complexity
of DSCOVR into synchronous and asynchronous communication.
The decomposition turns out to be different depending on the variance reduction
techniques employed: SVRG \citep{JohnsonZhang13} 
versus SAGA \citep{NIPS2014SAGA}.
We note that DSCOVR-SAGA essentially requires only
asynchronous communication, because the synchronous communication of~$m$
vectors are only necessary for initialization with non-zero starting point.

The comparisons in Table~\ref{tab:comm-comp} and~\ref{tab:sync-async}  
give us good understanding of the complexities of different algorithms.
However, these complexities are not accurate measures of their performance
in practice.
For example, collective communication of $m$ vectors in $\R^d$ 
can often be done in parallel over a spanning tree of the underlying 
communication network, thus only cost $\log(m)$ times  (insted of $m$ times)
compared with sending only one vector.
Also, for point-to-point communication, sending one vector in $\R^d$ altogether
can be much faster than sending~$n$ smaller vectors of total length~$d$ 
separately. 
A fair comparison in term of wall-clock time on a real-world distributed 
computing system
requires customized, efficient implementation of different algorithms.
We will shed some light on timing comparisons with 
numerical experiments in Section~\ref{sec:experiments}.

\subsection{Related Work}

There is an extensive literature on distributed optimization.
Many algorithms developed for machine learning adopt the centralized 
communication setting, due to the wide availability of supporting standards 
and platforms such as MPI, MapReduce and Spark 
(as discussed in the introduction).
They include parallel implementations of the batch first-order and 
second-order methods 
\citep[e.g.,][]{LinTsaiLeeLin14,ChenWangZhou2014,LeeWangChenLin2017},
ADMM \citep{Boyd2011ADMM}, 
and distributed dual coordinate ascent 
\citep{Yang2013DistributedSDCA,CoCoA2014NIPS,CoCoA2015ICML}.

For minimizing the average function $(1/m)\sum_{i=1}^m f_i(w)$, 
in the centralized setting and with only first-order oracles 
(i.e., gradients of $f_i$'s or their conjugates),
it has been shown that distributed implementation of accelerated gradient 
methods achieves the optimal convergence rate and communication complexity 
\citep{ArjevaniShamir2015NIPS,ScamanBach2017}.
The problem~\eqref{eqn:erm-m} we consider has the extra
structure of composition with a linear transformation by the local data, 
which allows us to exploit simultaneous data and model parallelism using
randomized algorithms and obtain improved communication and computation 
complexity.

Most work on asynchronous distributed algorithms exploit model parallelism
in order to reduce the synchronization cost, especially in the setting with 
parameter servers 
\citep[e.g.,][]{LiMu2014OSDI,Xing2015Petuum,AytekinJohansson2016}.
Besides, delay caused by the asynchrony can be incorporated to the step size to gain practical improvement on convergence
\citep[e.g.,][]{agDuc11,McMahanS14,SraYLS16}, 
though the theoretical sublinear rates remain.  
There are also many recent work on asynchronous parallel stochastic gradient
and coordinate-descent algorithms for convex optimization 
\citep[e.g.,][]{recht2011hogwild,LiuWrightRe14icml,EXTRA2015,
Reddi2015AsyncVR,RichtarikTakac12bigdata,ARock2016}.
When the workloads or computing power of different machines or processors
are nonuniform, they may significantly increase 
\emph{iteration efficiency} (number of iterations done in unit time),
but often at the cost of requiring more iterations than their 
synchronous counterparts (due to delays and stale updates). 
So there is a subtle balance between iteration efficiency
and iteration complexity \citep[e.g.,][]{HannahYin2017}.
Our discussions in Section~\ref{sec:main-results} show that DSCOVR is 
capable of improving both aspects.

For solving bilinear saddle-point problems with a finite-sum structure, 
\cite{ZhangXiao2017SPDC} proposed a randomized algorithm that 
works with dual coordinate update but full primal update.
\citet{YuLinYang15} proposed a doubly stochastic algorithm that works with
both primal and dual coordinate updates based on 
equation~\eqref{eqn:partial-grad-aj-wl}. 
Both of them achieved accelerated linear convergence rates, 
but neither can be readily applied to distributed computing.
In addition,
\citet{BalamuruganBach2016} proposed stochastic variance-reduction methods
(also based on SVRG and SAGA)
for solving more general convex-concave saddle point problems.
For the special case with bilinear coupling, they obtained 
similar
%essentially the same 
computation complexity as DSCOVR.
However, their methods require full model updates at each iteration 
(even though working with only one sub-block of data),
thus are not suitable for distributed computing.

With additional assumptions and structure, such as similarity between the local
cost functions at different machines or using second-order information, 
it is possible to obtain better communication complexity for distributed 
optimization; see, e.g., 
\citet{ShamirSrebroZhang14DANE,ZhangXiao2015DiSCO,AIDE2016}.
However, these algorithms rely on much more computation at each machine
for solving a local sub-problem at each iteration.
With additional memory and preprocessing at each machine, 
\citet{LeeLinMaYang2015} showed that SVRG can be adapted
for distributed optimization to obtain low communication complexity.

\section{The DSCOVR-SVRG Algorithm}
\label{sec:dscovr-svrg}

From this section to Section~\ref{sec:dual-free},
we present several realizations of DSCOVR using different variance
reduction techniques and acceleration schemes, and analyze their
convergence properties.
These algorithms are presented and analyzed as sequential randomized algorithms.
We will discuss how to implement them for asynchronous distributed 
computing in Section~\ref{sec:implementation}. 

\begin{algorithm}[t]
\caption{DSCOVR-SVRG}
\label{alg:dscovr-svrg}
\begin{algorithmic}[1]
\REQUIRE initial points $\wbini$, $\abini$, number of stages $S$ 
    and number of iterations per stage $M$.
\FOR {$s=0,1,2,\ldots,S-1$}
\vspace{0.3em}
\STATE $\ubs=X\wbs$ and $\vbs=\frac{1}{m}X^T\abs$\\[0.3em]
\STATE $\wini=\wbs$ and $\aini=\abs$
\FOR {$t=0,1,2,\ldots,M-1$}
\vspace{0.3em}
\STATE pick $j\in\{1,\ldots,m\}$ and $l\in\{1,\ldots,n\}$ 
randomly with distributions $p$ and $q$ respectively.\\[0.5em]
\STATE compute variance-reduced stochastic gradients:
\vspace{-0.5ex}
\begin{eqnarray}
\ujtp &=& \ubjs+\frac{1}{q_l}\Xjl\bigl(\wlt-\wbls\bigr), 
	\label{eqn:vr-sg-u}\\
\vltp &=& \vbls+\frac{1}{p_j}\frac{1}{m}(\Xjl)^T\bigl(\ajt-\abjs\bigr).
	\label{eqn:vr-sg-v}
\end{eqnarray}
\vspace{-1em}
\STATE update primal and dual block coordinates:
\vspace{-1ex}
\begin{eqnarray*}
\aitp &=& \left\{	\begin{array}{ll}
    \prox_{\sigma_j f_j^*}\bigl(\ajt+\sigma_j\ujtp\bigr) &\textrm{if}~i=j,\\
    \ait,  &\textrm{if}~i\neq j, \end{array} \right. \\
\wktp &=& \left\{ \begin{array}{ll}
    \prox_{\tau_l g_l}\bigl(\wlt-\tau_l\vltp\bigr) & \text{if}~ k=l,\\
    \wkt, & \textrm{if}~k\neq l. \end{array} \right. 
\end{eqnarray*}
\vspace{-1em}
\ENDFOR \\[0.5em]
\STATE $\wbsp=\wM$ and $\absp=\aM$.
\ENDFOR \\[0.5em]
\ENSURE $\wbS$ and $\abS$.
\end{algorithmic}
\end{algorithm}

Algorithm~\ref{alg:dscovr-svrg} is a DSCOVR algorithm that uses the 
technique of SVRG \citep{JohnsonZhang13} for variance reduction.
The iterations are divided into stages and each stage has a inner loop.
Each stage is initialized by a pair of vectors $\wbs\in\R^d$ and $\abs\in\R^N$, 
which come from either initialization (if $s=0$) or the last iterate of the
previous stage (if $s>0$).
At the beginning of each stage, we compute the batch gradients
\[
\ubs = \frac{\partial}{\partial \abs} \left((\abs)^T X \wbs\right)
	= X\wbs, \qquad
\vbs = \frac{\partial}{\partial \wbs} \left(\frac{1}{m}(\abs)^T X \wbs \right)
= \frac{1}{m} X^T \abs .
\]
The vectors $\ubs$ and $\vbs$ share the same partitions as $\at$ and $\wt$,
respectively.
Inside each stage~$s$, the variance-reduced stochastic gradients are computed
in~\eqref{eqn:vr-sg-u} and~\eqref{eqn:vr-sg-v}.
It is easy to check that they are unbiased.
More specifically, 
taking expectation of $\ujtp$ with respect to the random index~$l$ gives
\[
\E_l\bigl[\ujtp\bigr] 
= \ubjs + \sum_{k=1}^n q_k\frac{1}{q_k} \Xjk\bigl(\wkt -\wbks\bigr)
= \ubjs + \Xrj\wt -\Xrj\wbs = \Xrj\wt,
\]
and taking expectation of $\vltp$ with respect to the random index~$j$ gives
\[
\E_j\bigl[\vltp\bigr]  
= \vbls +\sum_{i=1}^m p_i\frac{1}{p_i}\frac{1}{m}(\Xil)^T\bigl(\ait-\abis\bigr)
= \vbls + \frac{1}{m}(\Xcl)^T\left(\at  - \abs \right)
= \frac{1}{m}(\Xcl)^T\at.
\]

In order to measure the distance of any pair of primal and dual variables to
the saddle point, we define a weighted squared Euclidean norm on $\R^{d+N}$.
Specifically, 
for any pair $(w,\alpha)$ where $w\in\R^d$ and 
$\alpha=[\alpha_1,\ldots,\alpha_m]\in\R^N$ with $\alpha_i\in\R^{N_i}$,
we define 
\begin{equation}\label{eqn:Omega}
\Omega(w,\alpha)=\lambda\|w\|^2+\frac{1}{m}\sum_{i=1}^m \gamma_i\|\alpha_i\|^2.
\end{equation}
If $\gamma_i=\gamma$ for all $i=1,\ldots,m$, then 
$\Omega(w,\alpha)=\lambda\|w\|^2+\frac{\gamma}{m}\|\alpha\|^2$.
We have the following theorem concerning the convergence rate of
Algorithm~\ref{alg:dscovr-svrg}.

\begin{theorem} \label{thm:dscovr-svrg}
Suppose Assumption~\ref{asmp:convexity} holds, and let $(\wopt,\aopt)$
be the unique saddle point of $L(w,\alpha)$.
Let $\Gamma$ be a constant that satisfies
\begin{equation}\label{eqn:Gamma-svrg}
\Gamma ~\geq~ \max_{i,k}\left\{ 
\frac{1}{p_i}\left(1+\frac{9\|\Xik\|^2}{2 q_k \lambda\gamma_i}\right) ,\;
\frac{1}{q_k}\left(1+\frac{9n\|\Xik\|^2}{2m p_i\lambda\gamma_i}\right)\right\}.
\end{equation}
In Algorithm~\ref{alg:dscovr-svrg}, if we choose the step sizes as
\begin{eqnarray}
  \sigma_i &=& \frac{1}{2\gamma_i(p_i\Gamma-1)},  \qquad i=1,\ldots,m, 
	\label{eqn:sigma-i-svrg} \\
  \tau_k &=& \frac{1}{2\lambda(q_k\Gamma-1)},  \qquad k=1,\ldots,n,
	\label{eqn:tau-k-svrg}
\end{eqnarray}
and the number of iterations during each stage satisfies
$M\geq\log(3)\Gamma$,
then for any $s>0$,
\begin{equation}\label{eqn:svrg-stage-converge}
  \E\left[\Omega\bigl(\wbs-\wopt,\abs-\aopt\bigr)\right]
  \leq \left(\frac{2}{3}\right)^s 
  \Omega\bigl(\wbini-\wopt,\abini-\aopt\bigr).
\end{equation}
\end{theorem}

The proof of Theorem~\ref{thm:dscovr-svrg} is given in 
Appendix~\ref{sec:proof-svrg}.
Here we discuss how to choose the parameter~$\Gamma$ to 
satisfy~\eqref{eqn:Gamma-svrg}.
For simplicity, we assume $\gamma_i=\gamma$ for all $i=1,\ldots,m$. 

\begin{itemize}
\item 
If we let $\|X\|_{m\times n}=\max_{i,k}\{\|\Xik\|\}$
and sample with the uniform distribution across both rows and columns, 
i.e., $p_i=1/m$ for $i=1,\ldots,m$ and $q_k=1/n$ for $k=1,\ldots,n$, 
then we can set
\[
\Gamma=\max\{m,n\}\left(1+\frac{9n\|X\|^2_{m\times n}}{2\lambda\gamma}\right)
=\max\{m,n\}\left(1 + \frac{9}{2}\condrand\right),
\]
where $\condrand = n\|X\|^2_{m\times n}\big/(\lambda\gamma)$ as defined
in~\eqref{eqn:cond-rand-mn}.
\item 
An alternative condition for $\Gamma$ to satisfy is
(shown in Section~\ref{sec:proof-Fro-step-sizes} in the Appendix)
\begin{equation}\label{eqn:Gamma-Fro-bound}
\Gamma \geq \max_{i,k}\left\{ 
\frac{1}{p_i}\left(1+\frac{9\|\Xck\|_F^2}{2 q_k m\lambda\gamma_i}\right),\;
\frac{1}{q_k}\left(1+\frac{9\|\Xri\|_F^2}{2 p_i m\lambda\gamma_i}\right)
\right\}.
\end{equation}
Again using uniform sampling, we can set
\[
\Gamma=\max\{m,n\}\left(1+\frac{9\|X\|^2_{\max, F}}{2\lambda\gamma}\right)
=\max\{m,n\}\left(1 + \frac{9}{2}\condrand'\right),
\]
where $\|X\|_{\max, F}=\max_{i,k}\{\|\Xri\|_F,\|\Xck\|_F\}$
and $\condrand'=\|X\|_{\max,F}^2\big/(\lambda\gamma)$
as defined in~\eqref{eqn:cond-rand-maxF}.
\item 
Using the condition~\eqref{eqn:Gamma-Fro-bound}, 
if we choose the probabilities to be proportional to the squared Frobenius 
norms of the data partitions, i.e., 
\begin{equation}\label{eqn:nonunif-sampling}
  p_i = \frac{\|\Xri\|_F^2}{\|X\|_F^2}, \qquad
  q_k = \frac{\|\Xck\|_F^2}{\|X\|_F^2}, 
\end{equation}
then we can choose
\[
  \Gamma = \frac{1}{\min_{i,k}\{p_i,q_k\}}
  \left(1+\frac{9\|X\|_F^2}{2m\lambda\gamma}\right)
  = \frac{1}{\min_{i,k}\{p_i,q_k\}}
  \left(1+\frac{9}{2}\condrand''\right),
\]
where $\condrand'' = \|X\|_F^2\big/(m\lambda\gamma)$.
Moreover, we can set the step sizes as
(see Appendix~\ref{sec:proof-Fro-step-sizes})
\begin{eqnarray*}
\sigma_i = \frac{m\lambda}{9\|X\|_F^2}, \qquad
\tau_k = \frac{m\gamma_i}{9\|X\|_F^2}.
\end{eqnarray*}
\item 
For the ERM problem~\eqref{eqn:erm-N}, we assume that each loss function
$\phi_j$, for $j=1,\ldots,N$, is $1/\nu$-smooth.
According to~\eqref{eqn:erm-fi}, the smooth parameter for each~$f_i$ is
$\gamma_i = \gamma = (N/m)\nu$.
Let $R$ be the largest Euclidean norm among all rows of~$X$
(or we can normalize each row to have the same norm~$R$), 
then we have $\|X\|_F^2\leq N R^2$ and 
\begin{equation}\label{eqn:erm-condrand}
\condrand'' = \frac{\|X\|_F^2}{m\lambda\gamma}
\leq \frac{N R^2}{m\lambda\gamma}
= \frac{R^2}{\lambda\nu}.
\end{equation}
The upper bound $R^2/(\lambda\nu)$ is a condition number used for 
characterizing the iteration complexity of many randomized algorithms for ERM
\citep[e.g.,][]{SSZhang13SDCA,LeRouxSchmidtBach12,JohnsonZhang13,NIPS2014SAGA,
ZhangXiao2017SPDC}.
In this case, using the non-uniform sampling in~\eqref{eqn:nonunif-sampling},
we can set the step sizes to be 
\begin{equation}\label{eqn:simple-step-size-R}
\sigma_i = \frac{\lambda}{9R^2}\frac{m}{N}, \qquad
\tau_k = \frac{\gamma}{9R^2}\frac{m}{N} = \frac{\nu}{9R^2}.
\end{equation}
\end{itemize}

Next we estimate the overall computation complexity of DSCOVR-SVRG in order to
achieve $\E\bigl[\Omega(\wbs-\wopt,\abs-\aopt)\bigr]\leq \epsilon$.
From~\eqref{eqn:svrg-stage-converge}, the number of stages required is
$\log\bigl(\Omegaini/\epsilon\bigr)\big/\log(3/2)$, where
$\Omegaini=\Omega(\wbini-\wopt,\abini-\aopt)$.
The number of inner iterations within each stage is
$M=\log(3)\Gamma$.
At the beginning of of each stage, computing the batch gradients~$\ubs$ 
and~$\vbs$ requires going through the whole data set $X$, 
whose computational cost is equivalent to $m\times n$ inner iterations.
Therefore, the overall complexity of Algorithm~\ref{alg:dscovr-svrg},
measured by total number of inner iterations, is
\[
O\left(\bigl(mn+\Gamma\bigr)\log\left(\frac{\Omegaini}{\epsilon}\right)\right).
\]
To simplify discussion, we further assume $m\leq n$, which is always the case
for distributed implementation 
(see Figure~\ref{fig:simul-parallel} and Section~\ref{sec:implementation}).
In this case, we can let $\Gamma=n(1+(9/2)\condrand)$.
Thus the above iteration complexity becomes
\begin{equation}\label{eqn:iter-complexity-svrg}
  O\bigl(n(1+m+\condrand)\log(1/\epsilon)\bigr) .
\end{equation}
Since the iteration complexity in~\eqref{eqn:iter-complexity-svrg} counts 
the number of blocks $\Xik$ being processed, the number of passes over
the whole dataset $X$ can be obtained by dividing it by $mn$, i.e., 
\begin{equation}\label{eqn:pass-complexity-svrg}
O\left( \left(1+\frac{\condrand}{m}\right)\log(1/\epsilon) \right).
\end{equation}
This is the computation complexity of DSCOVR listed in 
Table~\ref{tab:comm-comp}.
We can replace $\condrand$ by $\condrand'$ or $\condrand''$ 
depending on different proof techniques and sampling probabilities 
as discussed above.
We will address the communication complexity for DSCOVR-SVRG,
including its decomposition into synchronous and asynchronous ones,
after describing its implementation details in Section~\ref{sec:implementation}.

In addition to convergence to the saddle point, 
our next result shows that the primal-dual optimality gap 
also enjoys the same convergence rate, under slightly different conditions.

\begin{theorem}\label{thm:dscovr-svrg-obj}
Suppose Assumption~\ref{asmp:convexity} holds,
and let $P(w)$ and $D(\alpha)$ be the primal and dual functions defined 
in~\eqref{eqn:erm-primal} and~\eqref{eqn:erm-dual}, respectively.
Let $\Lambda$ and $\Gamma$ be two constants that satisfy
\[
  \Lambda ~\geq~  \|\Xik\|_F^2 \,,  
  \qquad i=1,\ldots,m,\quad k=1,\ldots,n,
\]
and 
\[
\Gamma ~\geq~ \max_{i,k}\left\{ 
\frac{1}{p_i}\left(1 + \frac{18\Lambda}{q_k \lambda\gamma_i}\right) ,\;
\frac{1}{q_k}\left(1 + \frac{18n\Lambda}{p_i m\lambda\gamma_i}\right)\right\}.
\]
In Algorithm~\ref{alg:dscovr-svrg}, if we choose the step sizes as
\begin{eqnarray}
  \sigma_i &=& \frac{1}{\gamma_i(p_i\Gamma-1)},  \qquad i=1,\ldots,m, 
	\label{eqn:sigma-i-obj} \\
  \tau_k &=& \frac{1}{\lambda(q_k\Gamma-1)},  \qquad k=1,\ldots,n,
	\label{eqn:tau-k-obj}
\end{eqnarray}
and the number of iterations during each stage satisfies
$M\geq\log(3)\Gamma$,
then
\begin{equation}\label{eqn:duality-gap-converge}
  \E\left[ P(\wbs)-D(\abs) \right] ~\leq~ 
  \left(\frac{2}{3}\right)^s 
  2\Gamma \left(P(\wbini)-D(\abini)\right) .
\end{equation}
\end{theorem}

The proof of Theorem~\ref{thm:dscovr-svrg-obj} is given in
Appendix~\ref{sec:proof-svrg-obj}.
In terms of iteration complexity or total number of passes to reach 
$\E\bigl[P(\wbs) - D(\abs)\bigr]\leq\epsilon$, 
we need to add an extra factor of $\log(1+\condrand)$ 
to~\eqref{eqn:iter-complexity-svrg} or~\eqref{eqn:pass-complexity-svrg},
due to the factor $\Gamma$ on the right-hand side 
of~\eqref{eqn:duality-gap-converge}.

\section{The DSCOVR-SAGA Algorithm}
\label{sec:dscovr-saga}

\begin{algorithm}[!t]
\caption{DSCOVR-SAGA}
\label{alg:dscovr-saga}
\begin{algorithmic}[1]
\REQUIRE initial points $\wini,\aini$, and number of iterations $M$.
\vspace{0.3em}
\STATE $\ubini=X\wini$ and $\vbini=\frac{1}{m}X^T\aini$ \\[0.3em]
\STATE $\Uikini = \Xik\wkini$, $\Vikini=\frac{1}{m}(\aiini)^T \Xik$, 
       for all $i=1,\ldots,m$ and $k=1,\ldots,K$.\\[0.5em]
\FOR {$t=0,1,2,\ldots,M-1$}
\vspace{0.3em}
\STATE pick $j\in\{1,\ldots,m\}$ and $l\in\{1,\ldots,n\}$ 
randomly with distributions $p$ and $q$ respectively.\\[0.3em]
\STATE compute variance-reduced stochastic gradients:
\begin{eqnarray}
\ujtp &=& \ubjt-\frac{1}{q_l}\Ujlt+\frac{1}{q_l}\Xjl\wlt, 
\label{eqn:saga-sg-u} \\
\vltp &=& \vblt-\frac{1}{p_j}(\Vjlt)^T +\frac{1}{p_j}\frac{1}{m}(\Xjl)^T\ajt.
\label{eqn:saga-sg-v} 
\end{eqnarray}
\vspace{-1em}
\STATE update primal and dual block coordinates: 
\begin{eqnarray*}
\aitp &=& \Biggl\{	\begin{array}{ll}
    \prox_{\sigma_j\Phi_j^*}\bigl(\ajt+\sigma_j\ujtp\bigr) &\textrm{if}~i=j.\\
    \ait,  &\textrm{if}~i\neq j, \end{array} \\
\wktp &=& \Biggl\{ \begin{array}{ll}
    \prox_{\tau_l g_l}\bigl(\wlt-\tau_l\vltp\bigr) & \text{if}~ k=l,\\
    \wkt, & \textrm{if}~i\neq j. \end{array} 
\end{eqnarray*}
\vspace{-1em}
\STATE update averaged stochastic gradients:
\begin{eqnarray*}
    \ubitp &=& \biggl\{\begin{array}{ll}\ubjt-\Ujlt+\Xjl\wlt
     &\textrm{if}~i=j, \\ \ubit & \textrm{if}~i\neq j, \end{array} \\
    \vbktp &=& \biggl\{\begin{array}{ll}\vblt-(\Vjlt)^T+\frac{1}{m}(\Xjl)^T\ajt
     &\textrm{if}~k=l, \\ \vbkt & \textrm{if}~k\neq l, \end{array} 
\end{eqnarray*}
\vspace{-1em}
\STATE update the table of historical stochastic gradients:
\begin{eqnarray*}
    \Uiktp &=& \biggl\{\begin{array}{ll}\Xjl\wlt 
     & \textrm{if}~i=j~\textrm{and}~k=l,\\
     \Uikt & \textrm{otherwise}. \end{array} \\
    \Viktp &=& \biggl\{\begin{array}{ll}\frac{1}{m}\bigl((\Xjl)^T\ajt\bigr)^T
     & \textrm{if}~i=j~\textrm{and}~k=l,\\
     \Vikt & \textrm{otherwise}. \end{array}
\end{eqnarray*}
\vspace{-1em}
\ENDFOR \\[0.3em]
\ENSURE $\wM$ and $\aM$.
\end{algorithmic}
\end{algorithm}

\noindent
Algorithm~\ref{alg:dscovr-saga} is a DSCOVR algorithm that uses the techniques
of SAGA \citep{NIPS2014SAGA} for variance reduction.
This is a single stage algorithm with iterations indexed by~$t$.
In order to compute the variance-reduced stochastic gradients 
$\ujtp$ and $\vltp$ at each iteration, we also need to maintain and update
two vectors $\ubt\in\R^N$ and $\vbt\in\R^d$, and 
two matrices $\Ut\in\R^{N\times n}$ and $\Vt\in\R^{m\times d}$.
The vector $\ubt$ shares the same partition as $\at$ into~$m$ blocks, 
and $\vbt$ share the same partitions as $\wt$ into~$n$ blocks. 
The matrix $\Ut$ is partitioned into $m\times n$ blocks, with each block
$\Uikt\in\R^{N_i \times 1}$.
The matrix $\Vt$ is also partitioned into $m\times n$ blocks, with each block
$\Vikt\in\R^{1\times d_k}$.
According to the updates in Steps~7 and~8 of Algorithm~\ref{alg:dscovr-saga},
we have
\begin{eqnarray}
  \ubit &=& \sum_{k=1}^n \Uikt, \qquad i=1,\ldots,m, 
  \label{eqn:ubit=sumUikt} \\
  \vbkt &=& \sum_{i=1}^m \bigl(\Vikt\bigr)^T, \qquad k=1,\ldots,n.
  \label{eqn:vbkt=sumVikt}
\end{eqnarray}

Based on the above constructions, we can show that $\ujtp$ is an unbiased 
stochastic gradient of $(\at)^T X\wt$ with respect to $\aj$, 
and $\vltp$ is an unbiased stochastic gradient of 
$(1/m)\bigl((\at)^T X\wt\bigr)$ with respect to~$\wl$.
More specifically, according to~\eqref{eqn:saga-sg-u}, we have
\begin{eqnarray}
\E_l\bigl[\ujtp\bigr] 
&=& \ubjt-\sum_{k=1}^n q_k\left(\frac{1}{q_k}\Ujkt\right) 
 +\sum_{k=1}^n q_k \left(\frac{1}{q_k}\Xjk\wkt\right) \nonumber \\
&=& \ubjt -\sum_{k=1}^n \Ujkt + \sum_{k=1}^n \Xjk\wkt \nonumber \\
&=& \ubjt -\ubjt + \Xrj\wt \nonumber \\
&=& \Xrj\wt 
~=~ \frac{\partial}{\partial \aj}\left(\bigl(\at\bigr)^T X\wt\right),
\label{eqn:ujtp-expect}
\end{eqnarray}
where the third equality is due to~\eqref{eqn:ubit=sumUikt}.
Similarly, according to~\eqref{eqn:saga-sg-v}, we have
\begin{eqnarray}
\E_j\bigl[\vltp\bigr] 
&=& \vblt-\sum_{i=1}^m p_i\left(\frac{1}{p_i}(\Vilt)^T\right)
+\sum_{i=1}^m p_i \left(\frac{1}{p_i m}(\Xil)^T\ait\right) \nonumber\\
&=& \vblt -\sum_{i=1}^m \Vilt + \frac{1}{m}\sum_{i=1}^m(\Xil)^T\ait \nonumber\\
&=& \vblt -\vblt + \frac{1}{m} (\Xcl)^T\at \nonumber\\
&=& \frac{1}{m}(\Xcl)^T\at
~=~ \frac{\partial}{\partial\wl}\left(\frac{1}{m}\bigl(\at\bigr)^T X\wt\right),
\label{eqn:vltp-expect}
\end{eqnarray}
where the third equality is due to~\eqref{eqn:vbkt=sumVikt}.

Regarding the convergence of DSCOVR-SAGA, we have the following theorem,
which is proved in Appendix~\ref{sec:proof-saga}.

\begin{theorem}\label{thm:dscovr-saga}
Suppose Assumption~\ref{asmp:convexity} holds, and let $(\wopt,\aopt)$
be the unique saddle point of $L(w,\alpha)$.
Let $\Gamma$ be a constant that satisfies
\begin{equation}\label{eqn:Gamma-saga}
\Gamma ~\geq~ \max_{i,k}\left\{ 
\frac{1}{p_i}\left(1+\frac{9\|\Xik\|^2}{2 q_k \lambda\gamma_i}\right),\;
\frac{1}{q_k}\left(1+\frac{9n\|\Xik\|^2}{2 p_i m\lambda\gamma_i}\right),\;
\frac{1}{p_i q_k} \right\}.
\end{equation}
If we choose the step sizes as
\begin{eqnarray}
  \sigma_i &=& \frac{1}{2\gamma_i(p_i\Gamma-1)},  \qquad i=1,\ldots,m, 
	\label{eqn:sigma-i-saga} \\
  \tau_k &=& \frac{1}{2\lambda(q_k\Gamma-1)},  \qquad k=1,\ldots,n,
	\label{eqn:tau-k-saga}
\end{eqnarray}
Then the iterations of Algorithm~\ref{alg:dscovr-saga} satisfy,
for $t=1,2,\ldots$,
\begin{equation}\label{eqn:saga-convergence}
\E\left[\Omega\bigl(\wt-\wopt, \at-\aopt\bigr)\right] 
~\leq~ \left(1-\frac{1}{3\Gamma}\right)^{t} 
\frac{4}{3}\Omega\bigl(\wini-\wopt,\aini-\aopt\bigr).
\end{equation}
\end{theorem}

The condition on~$\Gamma$ in~\eqref{eqn:Gamma-saga} is very similar to
the one in~\eqref{eqn:Gamma-svrg}, except that here we have an additional term 
$1/(p_i q_k)$ when taking the maximum over~$i$ and~$k$. 
This results in an extra $mn$ term in estimating~$\Gamma$ under uniform
sampling.
Assuming $m\leq n$ (true for distributed implementation), we can let 
\[
  \Gamma = n\left(1+\frac{9}{2}\condrand\right) + mn.
\]
According to~\eqref{eqn:saga-convergence}, 
in order to achieve $\E\bigl[\Omega(\wt-\wopt,\at-\aopt)\bigr]\leq \epsilon$,
DSCOVR-SAGA needs $O\left(\Gamma\log(1/\epsilon)\right)$ iterations.
Using the above expression for~$\Gamma$, the iteration complexity is
\begin{equation}\label{eqn:iter-complexity-saga}
  O\bigl(n(1+m+\condrand)\log(1/\epsilon)\bigr) ,
\end{equation}
which is the same as~\eqref{eqn:iter-complexity-svrg} for DSCOVR-SVRG.
This also leads to the same computational complexity measured by the number
of passes over the whole dataset, which is given 
in~\eqref{eqn:pass-complexity-svrg}.
Again we can replace $\condrand$ by $\condrand'$ or $\condrand''$ 
as discussed in Section~\ref{sec:dscovr-svrg}.
We will discuss the communication complexity of DSCOVR-SAGA in
Section~\ref{sec:implementation}, after describing its
implementation details.

\iffalse
The following discussions parallel those in Section~\ref{sec:dscovr-svrg}.
\begin{itemize}\itemsep 0pt
\item With the uniform distribution $p_i=1/m$ for all~$i$
    and $q_k=1/n$ for all~$k$, we can set
\[
\Gamma
%~=~\max\{m,n\}\left(1+\frac{9n\|X\|^2_{m\times n}}{2\lambda\gamma}\right)+mn 
~=~\max\{m,n\}\left(1 + \frac{9}{2}\condrand\right) +mn, 
\qquad \mbox{where}\qquad
\condrand = \frac{n\|X\|^2_{m\times n}}{\lambda\gamma}.
\]
\item Using an alternative condition on~$\Gamma$ similar 
  to~\eqref{eqn:Gamma-Fro-bound} and the uniform distribution, we can set
\[
\Gamma
%~=~\max\{m,n\}\left(1+\frac{9\|X\|^2_{\max, F}}{2\lambda\gamma}\right) + mn
~=~\max\{m,n\}\left(1 + \frac{9}{2}\condrand'\right) + mn ,
\qquad \mbox{where}\qquad
\condrand = \frac{\|X\|^2_{\max,F}}{\lambda\gamma}.
\]
\item Using the non-uniform distribution defined 
  in~\eqref{eqn:nonunif-sampling}, we can choose
\[
\Gamma 
%~=~ \frac{1}{\min_{i,k}\{p_i,q_k\}}
%\left(1+\frac{9\|X\|_F^2}{2m\lambda\gamma}\right)
~=~ \frac{1}{\min_{i,k}\{p_i,q_k\}} \left(1+\frac{9}{2}\condrand''\right) + mn,
\qquad \mbox{where}\qquad
\condrand'' = \|X\|_F^2\big/(m\lambda\gamma).
\]
\end{itemize}
\fi

\section{Accelerated DSCOVR Algorithms}
\label{sec:dscovr-accl}

In this section, we develop an accelerated DSCOVR algorithm by following
the ``catalyst'' framework 
\citep{LinMairalHarchaoui2015,FrostigGeKakadeSidford2015}. 
More specifically, we adopt the same procedure by \citet{BalamuruganBach2016} 
for solving convex-concave saddle-point problems.

\begin{algorithm}[t]
\caption{Accelerated DSCOVR}
\label{alg:accl-dscovr}
\begin{algorithmic}[1]
\REQUIRE initial points $\wtldini,\atldini$, and parameter $\delta>0$.
\FOR {$r=0,1,2,\ldots,$}
\STATE find an approximate saddle point of~\eqref{eqn:Lagrangian-delta}
using one of the following two options:
\begin{itemize}
\item \hspace{-1.5em} \emph{option 1:} 
run Algorithm~\ref{alg:dscovr-svrg} 
with $S=\frac{2\log(2(1+\delta))}{\log(3/2)}$ and $M=\log(3)\Gamma_\delta$ 
to obtain
\[
  (\wtldrp,\atldrp) = \textrm{DSCOVR-SVRG}(\wtldr,\atldr,S,M).
\]
\item \hspace{-1.5em} \emph{option 2:} 
run Algorithm~\ref{alg:dscovr-saga} with
$M=6\log\left(\frac{8(1+\delta)}{3}\right)\Gamma_\delta$ to obtain
\[
(\wtldrp,\atldrp) = \textrm{DSCOVR-SAGA}(\wtldr,\atldr,M).
\]
\end{itemize}
\ENDFOR
\end{algorithmic}
\end{algorithm}

Algorithm~\ref{alg:accl-dscovr} proceeds in rounds indexed by $r=0,1,2,\ldots$.
Given the initial points $\wtldini\in\R^d$ and $\atldini\in\R^N$, 
each round~$r$ computes two new vectors $\wtldrp$ and $\atldrp$ using either
the DSCOVR-SVRG or DSCOVR-SAGA algorithm for solving a regulated saddle-point
problem, similar to the classical proximal point algorithm
\citep{Rockafellar76PPA}.

Let $\delta>0$ be a parameter which we will determine later.
Consider the following perturbed saddle-point function for round~$r$:
\begin{equation}\label{eqn:Lagrangian-delta}
  \Lagr^{(r)}_\delta(w,a) = \Lagr(w,\alpha) 
  + \frac{\delta\lambda}{2}\|w-\wtldr\|^2 
  - \frac{\delta}{2m}\sum_{i=1}^m\gamma_i\|\ai-\aitldr\|^2 .
\end{equation}
Under Assumption~\ref{asmp:convexity}, the function $\Lagr^{(r)}_\delta(w,a)$
is $(1+\delta)\lambda$-strongly convex in~$w$ and
$(1+\delta)\gamma_i/m$-strongly concave in~$\alpha_i$.
Let $\Gamma_\delta$ be a constant that satisfies
\[
\Gamma_\delta ~\geq~ \max_{i,k}\left\{ 
\frac{1}{p_i}\left(1 + \frac{9\|\Xik\|^2}{2 q_k \lambda\gamma_i(1+\delta)^2}
\right),~
\frac{1}{q_k}\left(1 + \frac{9n\|\Xik\|^2}{2 p_i m\lambda\gamma_i(1+\delta)^2}
\right),~ \frac{1}{p_i q_k} \right\},
\]
where the right-hand side is obtained from~\eqref{eqn:Gamma-saga}
by replacing $\lambda$ and $\gamma_i$ with $(1+\delta)\lambda$ and  
$(1+\delta)\gamma_i$ respectively.
The constant $\Gamma_\delta$ is used in Algorithm~\ref{alg:accl-dscovr}
to determine the number of inner iterations to run with each round,
as well as for setting the step sizes.
The following theorem is proved in Appendix~\ref{sec:proof-accl}.

\begin{theorem}\label{thm:dscovr-accl}
Suppose Assumption~\ref{asmp:convexity} holds, and let $(\wopt,\aopt)$ 
be the saddle-point of~$\Lagr(w,\alpha)$.
With either options in Algorithm~\ref{alg:accl-dscovr}, 
if we choose the step sizes 
(inside Algorithm~\ref{alg:dscovr-svrg} or Algorithm~\ref{alg:dscovr-saga}) as
\begin{eqnarray}
\sigma_i &=& \frac{1}{2(1+\delta)\gamma_i(p_i\Gamma_\delta-1)},  
\qquad i=1,\ldots,m, \label{eqn:sigma-i-delta} \\
\tau_k &=& \frac{1}{2(1+\delta)\lambda(q_k\Gamma_\delta-1)},  
\qquad k=1,\ldots,n. \label{eqn:tau-k-delta}
\end{eqnarray}
Then for all $r\geq 1$,
\[
\E\left[\Omega\bigl(\wtldr-\wopt, \atldr-\aopt\bigr)\right] 
~\leq~ \left(1 - \frac{1}{2(1+\delta)}\right)^{2r}
\Omega\bigl(\wtldini-\wopt, \atldini-\aopt\bigr).
\]
\end{theorem}

According to Theorem~\ref{thm:dscovr-accl},
in order to have 
$\E\bigl[\Omega\bigl(\wtldr-\wopt, \atldr-\aopt\bigr)\bigr] 
\leq\epsilon$, 
we need the number of rounds $r$ to satisfy
\[
r\geq (1+\delta)\log\biggl(\frac{\Omega\bigl(\wtldini-\wopt,
\atldini-\aopt\bigr)}{\epsilon}\biggr).
\]
Following the discussions in 
Sections~\ref{sec:dscovr-svrg} and~\ref{sec:dscovr-saga},
when using uniform sampling and assuming $m\leq n$, we can have
\begin{equation}\label{eqn:Gamma-delta}
\Gamma_{\delta} 
= n\left(1+\frac{9\condrand}{2(1+\delta)^2} \right) + mn.
\end{equation}
Then the total number of block coordinate updates in 
Algorithm~\ref{alg:accl-dscovr} is
\[
O\bigl( (1+\delta)\Gamma_\delta \log(1+\delta) \log(1/\epsilon) \bigr),
\]
where the $\log(1+\delta)$ factor comes from the number of 
stages~$S$ in option~1 and number of steps~$M$ in option~2.
We hide the $\log(1+\delta)$ factor with the $\widetilde{O}$ notation
and plug~\eqref{eqn:Gamma-delta} into the expression above to obtain
\[
\widetilde{O}\left(n\left((1+\delta)(1+m)+\frac{\condrand}{(1+\delta)} \right)
  \log\left(\frac{1}{\epsilon}\right)\right).
\]
Now we can choose~$\delta$ depending on the relative size of $\condrand$ 
and~$m$:
\begin{itemize}
\item
If $\condrand>1+m$, we can minimizing the above expression by choosing
$\delta = \sqrt{\frac{\condrand}{1+m}} - 1$,
so that the overall iteration complexity becomes
$\widetilde{O}\left(n\sqrt{m\condrand}\log(1/\epsilon)\right)$.
\item 
If $\condrand \leq m + 1$, then no acceleration is necessary and we can choose
$\delta=0$ to proceed with a single round.
In this case, the iteration complexity is $O(mn)$ as seen 
from~\eqref{eqn:Gamma-delta}.
\end{itemize}
Therefore, in either case, the total number of block iterations by
Algorithm~\ref{alg:accl-dscovr} can be written as
\begin{equation}\label{eqn:accl-iter-complexity}
\widetilde{O}\left(mn + n\sqrt{m\condrand}\log(1/\epsilon)\right).
\end{equation}
As discussed before, the total number of passes over the whole dataset
is obtained by dividing by~$mn$:
\[
\widetilde{O}\left(1 + \sqrt{\condrand/m}\log(1/\epsilon)\right).
\]
This is the computational complexity of accelerated DSCOVR listed in
Table~\ref{tab:comm-comp}.

\subsection{Proximal Mapping for Accelerated DSCOVR}
\label{sec:accl-prox}

When applying Algorithm~\ref{alg:dscovr-svrg} or~\ref{alg:dscovr-saga} to
approximate the saddle-point of~\eqref{eqn:Lagrangian-delta},
we need to replace the proximal mappings of $g_k(\cdot)$ and $f_i^*(\cdot)$
by those of 
$g_k(\cdot)+(\delta\lambda/2)\|\cdot-\wktldr\|^2$ and 
$f_i^*(\cdot)+(\delta\gamma_i/2)\|\cdot-\aitldr\|^2$, respectively.
More precisely, we replace
$
\wktp = \prox_{\tau_k g_k}\bigl(\wkt-\tau_k \vktp\bigr)
$
by
\begin{align}
\wktp 
&=\argmin_{\wk\in\R^{d_k}}\left\{g_k(\wk) 
+ \frac{\delta\lambda}{2} \bigl\|\wk-\wktldr\bigr\|^2 
+ \frac{1}{2\tau_k} \left\|\wk-\left(\wkt-\tau_k \vktp\right)\right\|^2 
\right\} \nonumber \\
&= \prox_{\frac{\tau_k}{1+\tau_k\delta\lambda} g_k}\left(
\frac{1}{1+\tau_k\delta\lambda}\left(\wkt-\tau_k \vktp\right)
+\frac{\tau_k\delta\lambda}{1+\tau_k\delta\lambda}\wktldr \right) ,
\label{eqn:accl-primal-prox-mapping}
\end{align}
and replace 
$
\aitp = \prox_{\sigma_i f_i^*}\bigl(\ait+\sigma_i\uitp\bigr) 
$
by
\begin{align}
\aitp 
&=\argmin_{\ai\in\R^{N_i}}\left\{f_i^*(\ai) 
+\frac{\delta\gamma_i}{2}\bigl\|\ai-\aitldr\bigr\|^2
+ \frac{1}{2\sigma_i} \left\|\ai-\left(\ait+\sigma_i\uitp\right)\right\|^2
\right\} \nonumber \\
&= \prox_{\frac{\sigma_i}{1+\sigma_i\delta\gamma_i} f_i^\ast}\left(
\frac{1}{1+\sigma_i\delta\gamma_i}\left(\ait+\sigma_i \uitp\right)
+\frac{\sigma_i\delta\gamma_i}{1+\sigma_i\delta\gamma_i}\aitldr \right) .
\label{eqn:accl-dual-prox-mapping}
\end{align}

We also examine the number of inner iterations determined by $\Gamma_\delta$
and how to set the step sizes.
If we choose $\delta = \sqrt{\frac{\condrand}{1+m}} - 1$,
then $\Gamma_\delta$ in~\eqref{eqn:Gamma-delta} becomes
\[
  \Gamma_{\delta} = n\left(1+\frac{9\condrand}{2(1+\delta)^2}\right) + mn
  = n\left(1+\frac{9\condrand}{2\condrand/(m+1)}\right) + mn
  = 5.5(m+1)n.
\]
Therefore a small constant number of passes is sufficient within each round.
Using the uniform sampling, the step sizes can be estimated as follows:
\begin{align}
\sigma_i &= \frac{1}{2(1+\delta)\gamma_i(p_i\Gamma_{\delta}-1)}
\approx \frac{1}{2\sqrt{\condrand/m}\gamma_i (5.5 n-1)} 
\approx \frac{1}{11\gamma_i n\sqrt{\condrand/m}}, \label{eqn:accl-sigma-i}\\
\tau_k &= \frac{1}{2(1+\delta)\lambda(q_k\Gamma_{\delta}-1)}
\approx \frac{1}{2\sqrt{\condrand/m}\lambda(5.5m-1)}
\approx \frac{1}{11\lambda \sqrt{m\cdot \condrand}}. \label{eqn:accl-tau-k}
\end{align}
As shown by our numerical experiments in Section~\ref{sec:experiments}, 
the step sizes can be set much larger in practice.

\section{Conjugate-Free DSCOVR Algorithms}
\label{sec:dual-free}

A major disadvantage of primal-dual algorithms for solving
problem~\eqref{eqn:erm-m} is the requirement of computing the proximal mapping
of the conjugate function $f_i^*$, 
which may not admit closed-formed solution or efficient computation.
This is especially the case for logistic regression, 
one of the most popular loss functions used in classification.

\citet{LanZhou2015RPDG} developed ``conjugate-free'' variants of primal-dual 
algorithms that avoid computing the proximal mapping of the conjugate functions.
The main idea is to replace the Euclidean distance in the dual proximal
mapping with a Bregman divergence defined over the conjugate function itself.
This technique has been used by \citet{WangXiao2017ICML} 
to solve structured ERM problems with primal-dual first order methods.
Here we use this approach to derive conjugate-free DSCOVR algorithms.
In particular, we replace the proximal mapping for the dual update
\[
\aitp = \prox_{\sigma_i f_i^*}\bigl(\ait+\sigma_i\uitp\bigr) 
= \argmin_{\ai\in\R^{n_i}}\Bigl\{f_i^*(\ai)
  -\bigl\langle \ai,\,\uitp \bigr\rangle
  +\frac{1}{2\sigma_i}\bigl\|\ai-\ait\bigr\|^2 \Bigr\},
\]
by
\begin{equation}\label{eqn:conj-prox-Bregman}
\aitp = \argmin_{\ai\in\R^{n_i}}\left\{f_i^*(\ai)
  -\bigl\langle \ai,\,\uitp \bigr\rangle
  +\frac{1}{\sigma_i}\Bdiv_i\bigl(\ai,\ait\bigr) \right\},
\end{equation}
where 
$
  \Bdiv_i(\ai,\ait) = f_i^*(\ai) - \bigl\langle \nabla f_i^*(\ait),\,
  \ai-\ait\bigr\rangle .
$
The solution to~\eqref{eqn:conj-prox-Bregman} is given by
\[
  \aitp = \nabla f_i\bigl(\bitp\bigr), 
\]
where $\bitp$ can be computed recursively by
\[
  \bitp = \frac{\bit + \sigma_i \uitp}{1+\sigma_i}, \qquad t\geq 0,
\]
with initial condition $\biini=\nabla f_i^*(\aiini)$
\citep[see][Lemma~1]{LanZhou2015RPDG}.
Therefore, in order to update the dual variables $\alpha_i$, we do not need 
to compute the proximal mapping for the conjugate function $f_i^*$; 
instead, taking the gradient of~$f_i$ at some easy-to-compute points 
is sufficient.
This conjugate-free update can be applied in Algorithms~\ref{alg:dscovr}, 
\ref{alg:dscovr-svrg} and~\ref{alg:dscovr-saga}.

For the accelerated DSCOVR algorithms, 
we repalce~\eqref{eqn:accl-dual-prox-mapping} by
\begin{align*}
\aitp &= \argmin_{\ai\in\R^{n_i}}\left\{f_i^*(\ai)
  -\bigl\langle \ai,\,\uitp \bigr\rangle
  +\frac{1}{\sigma_i}\Bdiv_i\bigl(\ai,\ait\bigr)
  +\delta\gamma\Bdiv_i\bigl(\ai,\aitld\bigr)\right\}.
\end{align*}
The solution to the above minimization problem can also be written as
\[
  \aitp = \nabla f_i\bigl(\beta_i^{(t+1)}\bigr), 
\]
where $\beta_i^{(t+1)}$ can be computed recursively as 
\[
  \beta_i^{(t+1)} = \frac{\beta_t^{(t)} + \sigma_i \uitp 
  + \sigma_i\delta\gamma \tilde{\beta_i}}{1+\sigma_i+\sigma_i\delta\gamma}, 
  \qquad t\geq 0,
\]
with the initialization
$\beta_i^{(0)} = \nabla f_i^*\bigl(\aiini\bigr)$ and
$\tilde{\beta_i} = \nabla f_i^*\bigl(\aitldr\bigr)$.

The convergence rates and computational complexities of the conjugate-free
DSCOVR algorithms are very similar to the ones given in 
Sections~\ref{sec:dscovr-svrg}--\ref{sec:dscovr-accl}.
We omit details here, but refer the readers to
\citet{LanZhou2015RPDG} and \citet{WangXiao2017ICML}
for related results.

\section{Asynchronous Distributed Implementation}
\label{sec:implementation}

In this section, we show how to implement the DSCOVR algorithms presented
in Sections~\ref{sec:dscovr-svrg}--\ref{sec:dual-free} in a distributed 
computing system. 
We assume that the system provide both synchronous collective communication
and asynchronous point-to-point communication, which are all supported by the
MPI standard \citep{MPIForum}.
Throughout this section, we assume $m<n$ (see Figure~\ref{fig:simul-parallel}).

\begin{figure}[t]
  \centering
  \ifpdf
  \includegraphics[width=0.78\textwidth]{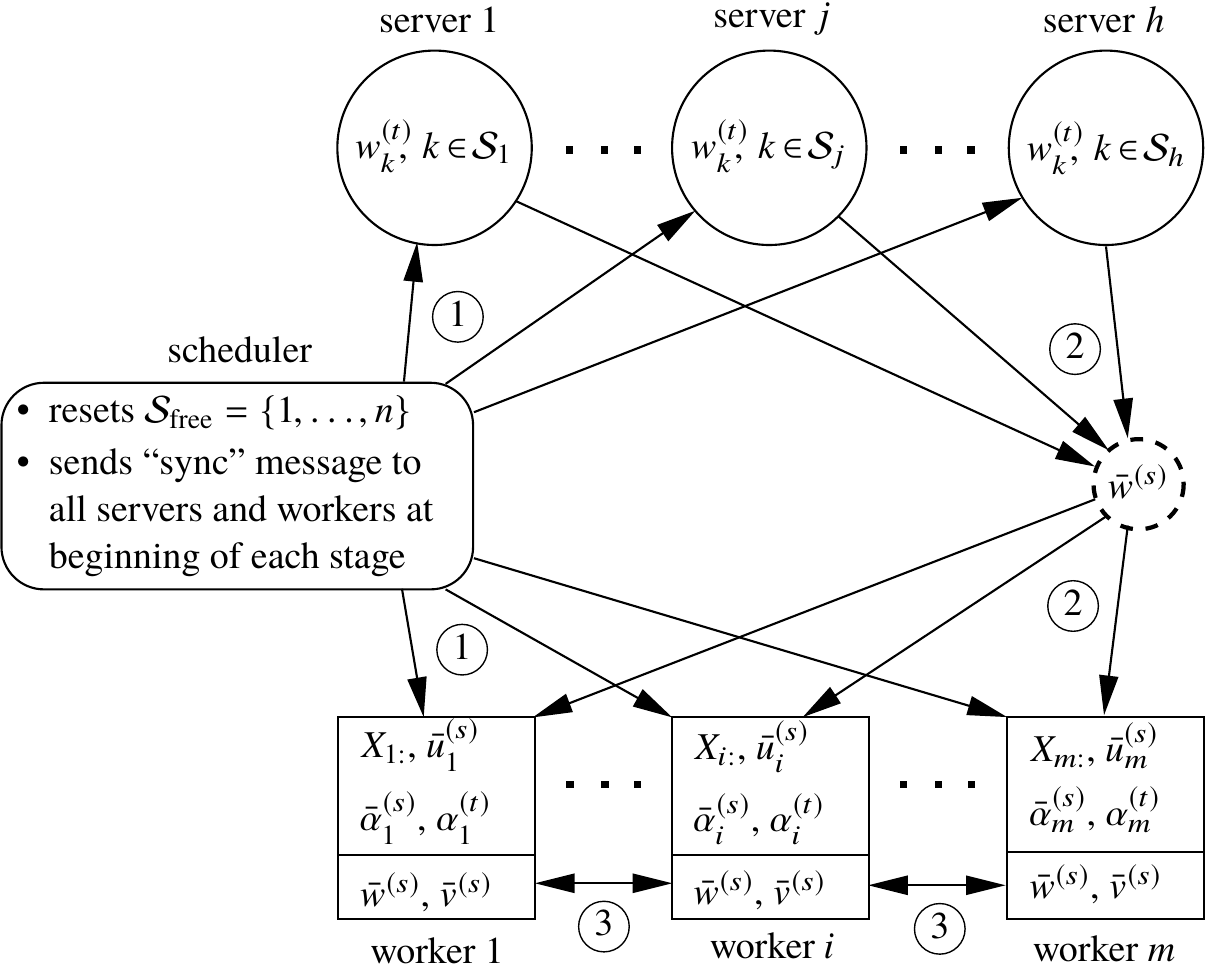}
  \else
  \input{drawings/dsvrg_sync_psfrag}
  \includegraphics[width=0.8\textwidth]{drawings/dsvrg_sync}
  \fi
  \vspace{2ex}
  \caption{A distributed system for implementing DSCOVR consists of
    $m$ workers, $h$ parameter servers, and one scheduler.
    The arrows labeled with the numbers~1, 2 and~3 represent three
    collective communications at the beginning of each stage
    in DSCOVR-SVRG.}
  \label{fig:dsvrg-sync}
\end{figure}

\subsection{Implementation of DSCOVR-SVRG}

In order to implement Algorithm~\ref{alg:dscovr-svrg}, 
the distributed system need to have the following components 
(see Figure~\ref{fig:dsvrg-sync}):
\begin{itemize} %\itemsep 0pt
  \item \emph{$m$ workers}. 
    Each worker~$i$, for $i=1,\ldots,m$, stores the 
    following local data and variables :
    \begin{itemize}
      \item data matrix $\Xri \in \R^{N_i\times d}$.
      \item vectors in $\R^{N_i}$:  $\ubis$, $\ait$, $\abis$.
      \item vectors in $\R^d$: $\wbs$, $\vbs$.
      \item extra buffers for computation and communication: 
            $\ujtp$, $\vltp$, $\wlt$ and $\wltp$.
    \end{itemize}
  \item \emph{$h$ parameter servers}. 
    Each server~$j$ stores a subset of the blocks 
    $\bigl\{\wkt\in\R^{d_k}:k\in\mathcal{S}_j\bigr\}$,
    where $\mathcal{S}_1,\ldots,\mathcal{S}_h$
    form a partition of the set $\{1,\ldots,n\}$.
  \item \emph{one scheduler}. It maintains a set of block indices
    $\Ifree\subseteq\{1,\ldots,n\}$. 
    At any given time, $\Ifree$ contains indices of parameter blocks 
    that are not currently updated by any worker.
\end{itemize}
The reason for having $h>1$ servers is not about insufficient storage
for parameters, but rather to avoid the communication overload between
only one server and all~$m$ workers ($m$ can be in hundreds).

At the beginning of each stage~$s$, 
the following three collective communications take place across the system 
(illustrated in Figure~\ref{fig:dsvrg-sync} by arrows with circled
labels~1, 2 and~3):
\begin{enumerate} %\itemsep 0pt
  \item[(1)] 
    The scheduler sends a ``sync'' message to all servers and workers, 
    and resets $\Ifree=\{1,\ldots,n\}$. 
  \item[(2)] 
    Upon receiving the ``sync'' message, the servers aggregate their blocks
    of parameters together to form $\wbs$ and send it to all workers
    (e.g., through the AllReduce operation in MPI).
  \item[(3)] 
    Upon receiving $\wbs$, each worker compute $\ubis=\Xri\wbs$ and 
    $(\Xri)^T\abis$, then invoke a collective communication (AllReduce)
    to compute $\vbs=(1/m)\sum_{i=1}^m(\Xri)^T\abis$.
\end{enumerate}
The number of vectors in $\R^d$ sent and received during the above process
is~$2m$, counting the communications to form $\wbs$ and $\vbs$ 
at~$m$ workers (ignoring the short ``sync'' messages). 

After the collective communications at the beginning of each stage, 
all workers start working on the inner iterations of 
Algorithm~\ref{alg:dscovr-svrg}
in parallel in an asynchronous, event-driven manner.
Each worker interacts with the scheduler and the servers in a four-step loop
shown in Figure~\ref{fig:dsvrg-update}.
There are always~$m$ iterations taking place concurrently
(see also Figure~\ref{fig:simul-parallel}), each may at a
different phase of the four-step loop:
\begin{enumerate}%\itemsep 0pt
\item[(1)] 
  Whenever worker~$i$ finishes updating a block~$k'$, it sends the pair 
  $(i,k')$ to the scheduler to request for another block to update. 
  At the beginning of each stage, $k'$ is not needed.
\item[(2)] When the scheduler receives the pair $(i,k')$, it randomly choose 
  a block~$k$ from the list of free blocks $\Ifree$ (which are not currently 
  updated by any worker), looks up for the server~$j$ which stores the 
  parameter block $\wkt$ (i.e., $\mathcal{S}_j\owns k$), 
  and then send the pair $(i,k)$ to server~$j$. In addition, the 
  scheduler updates the list $\Ifree$ by adding~$k'$ and deleting~$k$.
\item[(3)] When server~$j$ receives the pair $(i, k)$, it sends the vector 
  $\wkt$ to worker~$i$, and waits for receiving the updated version 
  $\wktp$ from worker~$i$.
\item[(4)] After worker~$i$ receives $\wkt$, it computes the updates $\ait$ 
  and $\wkt$ following steps~6-7 in Algorithm~\ref{alg:dscovr-svrg}, and then 
  send $\wktp$ back to server~$j$. At last, it assigns the value of~$k$
  to~$k'$ and send the pair $(i,k')$ to the scheduler, requesting the next
  block to work on.
\end{enumerate}
The amount of point-to-point communication required during the above process
is $2d_k$ float numbers, for sending and receiving $\wkt$ and $\wktp$
(we ignore the small messages for sending and receiving $(i,k')$ and $(i,k)$).
Since the blocks are picked randomly, the average amount of communication 
per iteration is $2d/n$, or equivalent to $2/n$ vectors in $\R^d$.
According to Theorem~\ref{thm:dscovr-svrg}, each stage of 
Algorithm~\ref{alg:dscovr-svrg} requires $\log(3)\Gamma$ inner iterations;
In addition, the discussions above~\eqref{eqn:iter-complexity-svrg} show that
we can take $\Gamma=n(1+(9/2)\condrand)$.
Therefore, the average amount of point-to-point communication within each stage
is $O(\condrand)$ vectors in $\R^d$.

\begin{figure}[t]
  \centering
  \ifpdf
  \includegraphics[width=0.99\textwidth]{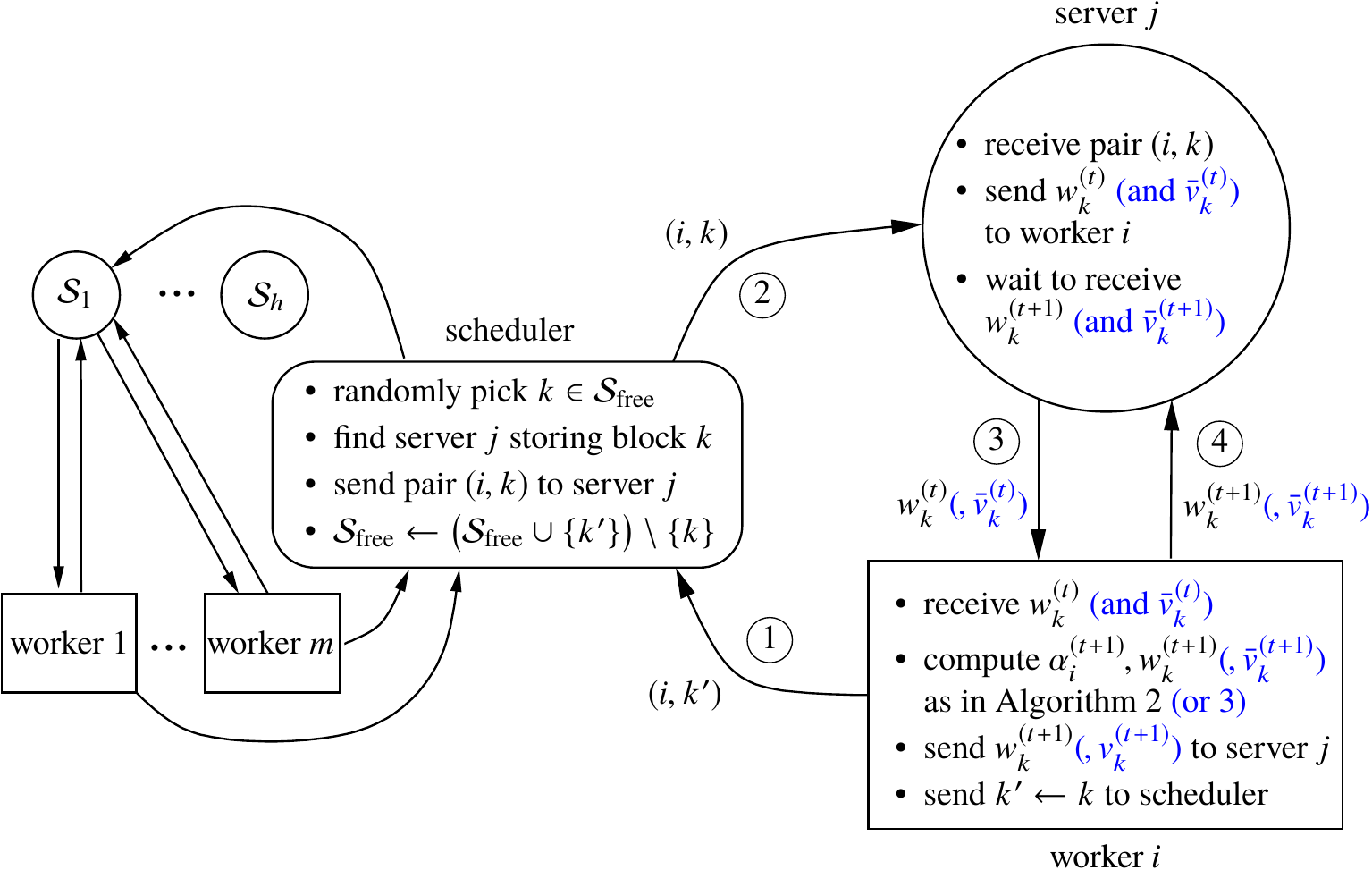}
  \else
  \input{drawings/dsaga_onestep_psfrag}
  \includegraphics[width=0.99\textwidth]{drawings/dscovr_onestep}
  \fi
  \vspace{2ex}
  \caption{Communication and computation processes for one
	inner iteration of DSCOVR-SVRG (Algorithm~\ref{alg:dscovr-svrg}). 
    The blue texts in the parentheses are the additional vectors required
    by DSCOVR-SAGA (Algorithm~\ref{alg:dscovr-saga}).
    There are always $m$ iterations taking place in parallel 
    asynchronously, each evolving around one worker.
    A server may support multiple (or zero) iterations if more than one 
    (or none) of its stored parameter blocks are being updated.
  }
  \label{fig:dsvrg-update}
\end{figure}

Now we are ready to quantify the communication complexity of
DSCOVR-SVRG to find an $\epsilon$-optimal solution.
Our discussions above show that each stage requires collective communication 
of $2m$ vectors in $\R^d$ and asynchronous point-to-point communication
of equivalently $\condrand$ such vectors.
Since there are total $O(\log(1/\epsilon))$ stages, the total
communication complexity is
\[
O\left((m+\condrand)\log(1/\epsilon)\right).
\]
This gives the communication complexity shown in Table~\ref{tab:comm-comp},
as well as its decomposition in Table~\ref{tab:sync-async}.

\subsection{Implementation of DSCOVR-SAGA}
\label{sec:implement-saga}

We can implement Algorithm~\ref{alg:dscovr-saga} using the same distributed
system shown in Figure~\ref{fig:dsvrg-sync}, but with some modifications
described below. First, the storage at different components are different:
\begin{itemize}\itemsep 0pt
    \item \emph{$m$ workers.} Each worker~$i$, for $i=1,\ldots,m$, 
      stores the following data and variables:
  \begin{itemize}\itemsep 0pt
      \item data matrix $\Xri \in \R^{N_i\times d}$
      \item vectors in $\R^{N_i}$: $\ait$, $\uit$, $\ubit$, 
            and $\Uikt$ for $k=1,\ldots,n$.
      \item vector in $\R^d$: 
            $V_{i:}^{(t)}=\bigl[V_{i1}^{(t)} \cdots V_{in}^{(t)}\bigr]^T$
            (which is the $i$th row of $\Vt$, with $\Vikt\in\R^{1\times d_k}$).
      \item buffers for communication and update of
            $\wkt$ and $\vbkt$ (both stored at some server).
    \end{itemize}
  \item \emph{$h$ servers}. Each server~$j$ stores a subset of blocks 
        $\bigl\{\wkt\!\!,\,\vbkt\in\R^{d_k}:k\in\mathcal{S}_j\bigr\}$,
        for $j=1,\ldots,n$.
  \item \emph{one scheduler}. It maintains the set of indices
    $\Ifree\subseteq\{1,\ldots,n\}$, same as in DSCOVR-SVRG.
\end{itemize}
Unlike DSCOVR-SVRG, there is no stage-wise ``sync'' messages.
All workers and servers work in parallel asynchronously all the time, 
following the four-step loops illustrated in Figure~\ref{fig:dsvrg-update}
(including blue colored texts in the parentheses).
Within each iteration, the main difference from DSCOVR-SVRG is that, 
the server and worker need to exchange two vectors of length~$d_k$:
$\wkt$ and $\vkt$ and their updates. This doubles the amount of 
point-to-point communication, and the average amount of communication
per iteration is $4/n$ vectors of length~$d$.
Using the iteration complexity in~\eqref{eqn:iter-complexity-saga},
the total amount of communication required (measured by number of
vectors of length~$d$) is
\[
  O\left((m+\condrand)\log(1/\epsilon)\right),
\]
which is the same as for DSCOVR-SVRG.
However, its decomposition into synchronous and asynchronous communication
is different, as shown in Table~\ref{tab:sync-async}.
If the initial vectors $\wini\neq 0$ or $\aini\neq 0$, then one round
of collective communication is required to propagate the initial conditions
to all servers and workers, which reflect the $O(m)$ synchronous 
communication in Table~\ref{tab:sync-async}.

\subsection{Implementation of Accelerated DSCOVR}

Implementation of the accelerated DSCOVR algorithm is very similar to 
the non-accelerated ones.
The main differences lie in the two proximal mappings 
presented in Section~\ref{sec:accl-prox}.
In particular, the primal update in~\eqref{eqn:accl-primal-prox-mapping} 
needs the extra variable $\wktldr$, which should be stored at a parameter
server together with $\wkt$.
We modify the four-step loops shown in Figures~\ref{fig:dsvrg-update}
as follows:
\begin{itemize} \itemsep 0pt
\item Each parameter server~$j$ stores the extra block parameters 
$\bigl\{\wktldr, k\in\mathcal{S}_j\bigr\}$. 
During step~(3), $\wktldr$ is send together with $\wkt$ (for SVRG)
or $(\wkt,\vkt)$ (for SAGA) to a worker.
\item In step (4), no update of $\wktldr$ is sent back to the server.
Instead, whenever switching rounds, the scheduler will inform
each server to update their $\wktldr$ to the most recent $\wkt$.
\end{itemize}
For the dual proximal mapping in~\eqref{eqn:accl-dual-prox-mapping}, 
each worker~$i$ needs to store an extra vector $\aitldr$, and reset it to
the most recent $\ait$ when moving to the next round.
There is no need for additional synchronization or collective communication 
when switching rounds in Algorithm~\ref{alg:accl-dscovr}.
The communication complexity 
(measured by the number of vectors of length~$d$ sent or received)
can be obtained by dividing the iteration complexity 
in~\eqref{eqn:accl-iter-complexity} by~$n$, i.e., 
$O\bigl((m+\sqrt{m\condrand})\log(1/\epsilon)\bigr)$,
as shown in Table~\ref{tab:comm-comp}.

Finally, in order to implement the conjugate-free DSCOVR algorithms described
in Section~\ref{sec:dual-free}, each worker~$i$ simply need to maintain and 
update an extra vector $\bit$ locally.

\section{Experiments}
\label{sec:experiments}

In this section, we present numerical experiments on an industrial distributed
computing system.
This system has hundreds of computers connected by high speed Ethernet in a
data center. 
The hardware and software configurations for each machine are listed in
Table~\ref{tab:hardware}.
We implemented all DSCOVR algorithms presented in this paper, including the 
SVRG and SAGA versions, their accelerated variants, as well as the 
conjugate-free algorithms.
All implementations are written in C++, using MPI for both collective and 
point-to-point communications 
(see Figures~\ref{fig:dsvrg-sync} and~\ref{fig:dsvrg-update} respectively).
On each worker machine, we also use OpenMP \citep{OpenMP}
to exploit the multi-core architecture for parallel computing,
including sparse matrix-vector multiplications
and vectorized function evaluations.

\begin{table}[t]
  \centering
  \begin{tabular}{c|c|c|c|c}
   \hline
   CPU & \#cores & RAM & network & operating system \\
   \hline
   dual Intel\textsuperscript{\textregistered} 
   Xeon\textsuperscript{\textregistered} processors & 
   16 & 128 GB & 10 Gbps & 
   Windows\textsuperscript{\textregistered} Server \\
   E5-2650 (v2), ~2.6 GHz & & ~~~1.8 GHz & Ethernet adapter & (version 2012)\\
   \hline
  \end{tabular}
  \caption{Configuration of each machine in the distributed computing system.}
  \label{tab:hardware}
\end{table}

Implementing the DSCOVR algorithms requires $m+h+1$ machines, 
among them~$m$ are workers with local datasets, 
$h$ are parameter servers, and one is a scheduler
(see Figure~\ref{fig:dsvrg-sync}).
We focus on solving the ERM problem~\eqref{eqn:erm-N},
where the total of~$N$ training examples are evenly partitioned and 
stored at~$m$ workers. 
We partition the $d$-dimensional parameters into $n$ subsets of 
roughly the same size (differ at most by one), where each subset consists of
randomly chosen coordinates (without replacement). 
Then we store the~$n$ subsets of parameters on~$h$ servers, each getting 
either $\lfloor n/h\rfloor $ or $\lceil n/h \rceil$ subsets. 
As described in Section~\ref{sec:implementation},
we make the configurations to satisfy $n>m>h\geq 1$.

For DSCOVR-SVRG and DSCOVR-SAGA, the step sizes 
in~\eqref{eqn:simple-step-size-R} are very conservative.
In the experiments, we replace the coefficient $1/9$ 
by two tuning parameter~$\eta_\mathrm{d}$ and $\eta_\mathrm{p}$
for the dual and primal step sizes respectively, i.e.,
\begin{equation}\label{eqn:dscovr-step-sizes}
  \sigma_i = \eta_\mathrm{d} \frac{\lambda}{R^2}\cdot \frac{m}{N}, \qquad
  \tau_k = \eta_\mathrm{p} \frac{\nu}{R^2}.
\end{equation}
For the accelerated DSCOVR algorithms, we use $\condrand=R^2/(\lambda\nu)$ 
as shown in~\eqref{eqn:erm-condrand} for ERM.
Then the step sizes in~\eqref{eqn:accl-sigma-i} and~\eqref{eqn:accl-tau-k}, 
with $\gamma_i=(m/N)\nu$ and a generic constant coefficient~$\eta$, become
\begin{equation}\label{eqn:accl-dscovr-step-sizes}
\sigma_i = \frac{\eta_\mathrm{d}}{nR}\sqrt{\frac{m\lambda}{\nu}}
\cdot\frac{m}{N},  \qquad
\tau_k = \frac{\eta_\mathrm{p}}{R}\sqrt{\frac{\nu}{m\lambda}}.
\end{equation}
For comparison, we also implemented the following first-order methods for
solving problem~\ref{eqn:erm-m}:
\begin{itemize} \itemsep 0pt
  \item PGD: parallel implementation of the Proximal Gradient Descent method
    (using synchronous collective communication over $m$ machines).
    We use the adaptive line search procedure proposed in 
    \citet{Nesterov13composite}, and the exact form used is Algorithm~2 
    in \citet{LinXiao2015homotopy}.
  \item APG: parallel implementation of the Accelerated Proximal Gradient 
    method \citep{Nesterov04book,Nesterov13composite}.
    We use a similar adaptive line search scheme to the one for PGD,
    and the exact form used (with strong convexity) is
    Algorithm~4 in \citet{LinXiao2015homotopy}.
  \item ADMM: the Alternating Direction Method of Multipliers.
    We use the regularized consensus version 
    in \citet[][Section~7.1.1]{Boyd2011ADMM}. 
    For solving the local optimization problems at each node, we use 
    the SDCA method \citep{SSZhang13SDCA}.
  \item CoCoA+: the adding version of CoCoA in \citet{CoCoA2015ICML}.
    Following the suggestion in~\citet{CoCoA2017arbitrary},
    we use a randomized coordinate descent algorithm 
    \citep{Nesterov12rcdm,RichtarikTakac12} for solving the local
    optimization problems. 
\end{itemize}
These four algorithms all require~$m$ workers only. 
Specifically, we use the AllReduce call in MPI for the collective 
communications so that a separate master machine is not necessary.

\begin{table}[t]
  \centering
  \begin{tabular}{l|r|r|r}
   \hline
   Dataset & \#instances ($N$) & \#features ($d$) & \#nonzeros \\
   \hline
   \texttt{rcv1-train}   &    677,399 &     47,236 &      49,556,258 \\
   \texttt{webspam}     &    350,000 & 16,609,143 &   1,304,697,446 \\
   \texttt{splice-site} & 50,000,000 & 11,725,480 & 166,167,381,622 \\
   \hline
  \end{tabular}
  \caption{Statistics of three datasets. Each feature vector is normalized
to have unit norm.}
  \label{tab:data-stat}
\end{table}

We conducted experiments on three binary classification datasets
obtained from the collection maintained by \citet{LIBSVMdata}.
Table~\ref{tab:data-stat} lists their sizes and dimensions.
In our experiments, we used two configurations: one with $m=20$ and $h=10$
for two relatively small datasets, \texttt{rcv1-train} and \texttt{webspam},
and the other with $m=100$ and $h=20$ for the large dataset 
\texttt{splice-site}.

For \texttt{rcv1-train}, we solve the ERM problem~\eqref{eqn:erm-N}
with a smoothed hinge loss defined as
\[
\phi_j(t) = \left\{\begin{array}{ll} 
0                      & \mbox{if}~y_j t\geq 1,\\
\frac{1}{2}-y_j t      & \mbox{if}~ y_j t\leq 0,\\
\frac{1}{2}(1-y_j t)^2 & \mbox{otherwise},
\end{array}\right.
\qquad\mbox{and}\qquad
\phi_j^*(\beta) = \left\{\begin{array}{ll}
y_j\beta + \frac{1}{2}\beta^2 & \mbox{if}~ -1\leq y_j\beta \leq 0, \\
+\infty & \mbox{otherwise}.
\end{array}\right.
\]
for $j=1,\ldots,N$. This loss function is 1-smooth, therefore $\nu=1$;
see discussion above~\eqref{eqn:erm-condrand}.
We use the $\ell_2$ regularization $g(w)=(\lambda/2)\|w\|^2$.
Figures~\ref{fig:rcv1train4} and~\ref{fig:rcv1train6} show the reduction
of the primal objective gap $P(\wt)-P(\wopt)$ by different algorithms, 
with regularization parameter $\lambda=10^{-4}$ and $\lambda=10^{-6}$ 
respectively. 
All started from the zero initial point.
Here the~$N$ examples are randomly shuffled and then divided into~$m$ subsets.
The labels SVRG and SAGA mean DSCOVR-SVRG and DSCOVR-SAGA, respectively,
and A-SVRG and A-SAGA are their accelerated versions. 

Since PGD and APG both use adaptive line search, there is no parameter to tune.
For ADMM, we manually tuned the penalty parameter~$\rho$ 
\citep[see][Section~7.1.1]{Boyd2011ADMM} to obtain good performance:
$\rho=10^{-5}$ in Figure~\ref{fig:rcv1train4} 
and $\rho=10^{-6}$ in Figure~\ref{fig:rcv1train6}.
For CoCoA+, two passes over the local datasets using a randomized coordinate
descent method are sufficient for solving the local optimization problem
(more passes do not give meaningful improvement).
For DSCOVR-SVRG and SAGA, we used $\eta_\mathrm{p}=\eta_\mathrm{d}=20$ 
to set the step sizes in~\eqref{eqn:dscovr-step-sizes}.
For DSCOVR-SVRG, each stage goes through the whole dataset~$10$ times, i.e., 
the number of inner iterations in Algorithm~\ref{alg:dscovr-svrg} is $M=10mn$.
For the accelerated DSCOVR algorithms, better performance are obtained with
small periods to update the proximal points and we set it to be every 
$0.2$ passes over the dataset, i.e., $0.2mn$ inner iterations.
For accelerated DSCOVR-SVRG, we set the stage period (for variance reduction)
to be $M=mn$, which is actually longer than the period for updating the 
proximal points.

From Figures~\ref{fig:rcv1train4} and~\ref{fig:rcv1train6}, we observe that
the two distributed algorithms based on model averaging, ADMM and CoCoA+, 
converges relatively fast in the beginning but becomes very slow in the 
later stage.
Other algorithms demonstrate more consistent linear convergence rates.
For $\lambda=10^{-4}$, the DSCOVR algorithms are very competitive compared
with other algorithms.
For $\lambda=10^{-6}$, the non-accelerated DSCOVR algorithms become very slow,
even after tuning the step sizes.
But the accelerated DSCOVR algorithms are superior 
in terms of both number of passes over data and wall-clock time
(with adjusted step size coefficient $\eta_\mathrm{p}=10$ 
and $\eta_\mathrm{d}=40$).

For ADMM and CoCoA+, each marker represents the finishing of one iteration.
It can be seen that they are mostly evenly spaced in terms of number of passes
over data, but have large variations in terms of wall-clock time.
The variations in time per iteration are due to resource sharing with other 
jobs running simultaneously on the distributed computing cluster.
Even if we have exclusive use of each machine, 
sharing communications with other jobs over the Ethernet is unavoidable.
This reflects the more realistic environment in cloud computing.

\begin{figure}[t]
  \centering
  \includegraphics[width=0.49\textwidth]{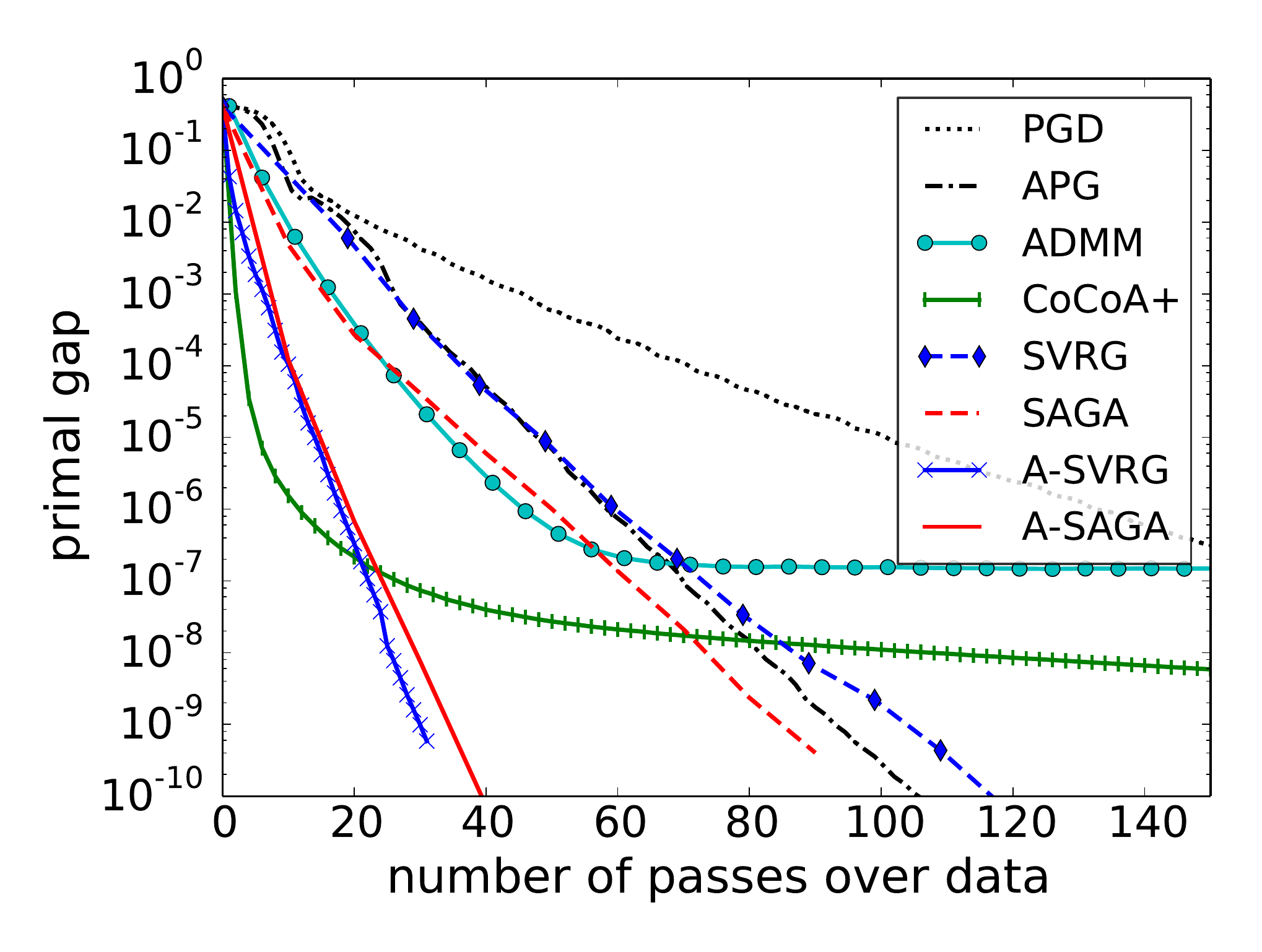}
  \includegraphics[width=0.49\textwidth]{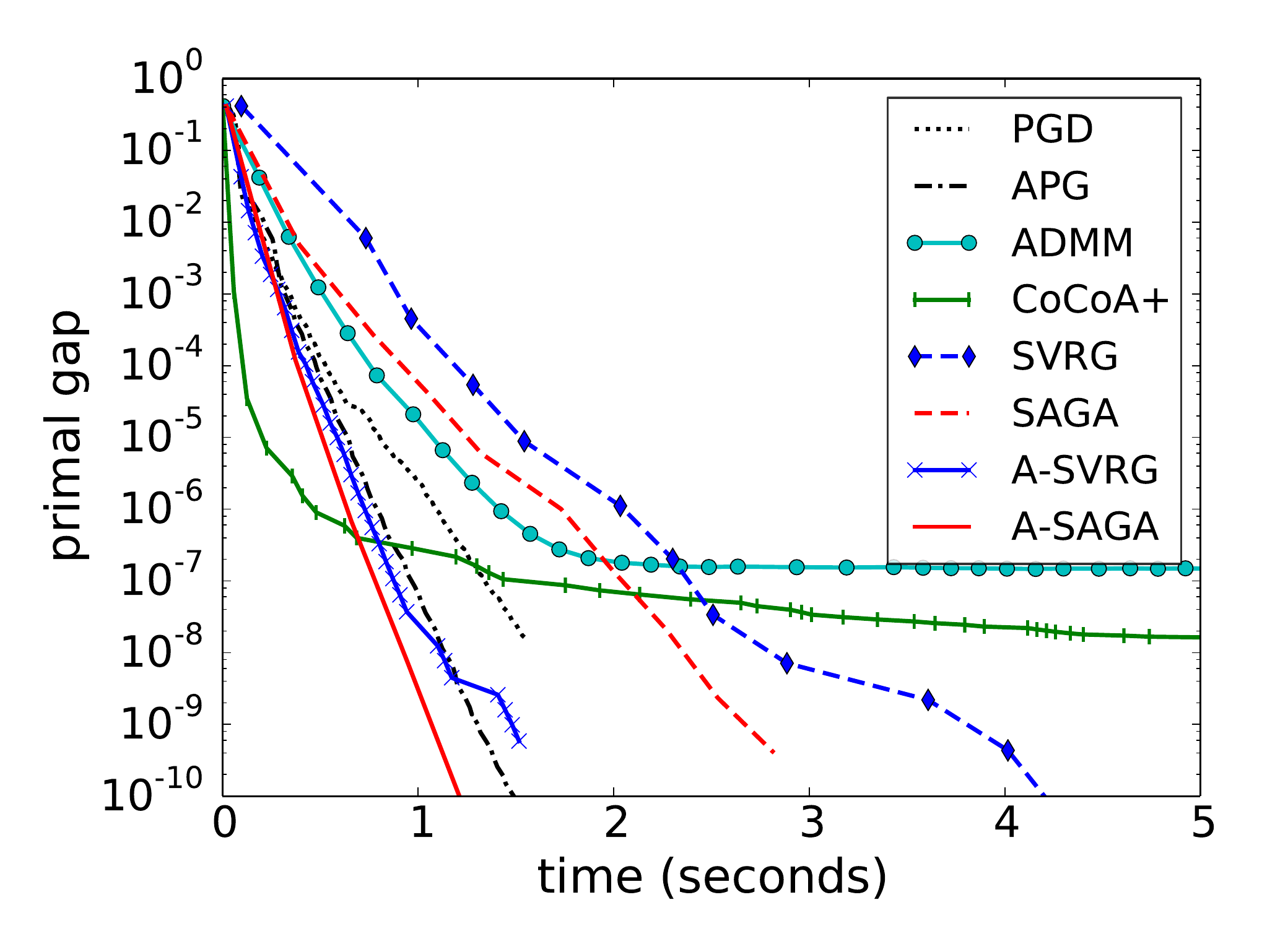}
  \caption{\texttt{rcv1-train}: smoothed-hinge loss, $\lambda\!=\!10^{-4}$,
            randomly shuffled, $m\!=\!20$, $n\!=\!37$, $h\!=\!10$.}
  \label{fig:rcv1train4}
\end{figure}

\begin{figure}[t]
  \centering
  \includegraphics[width=0.49\textwidth]{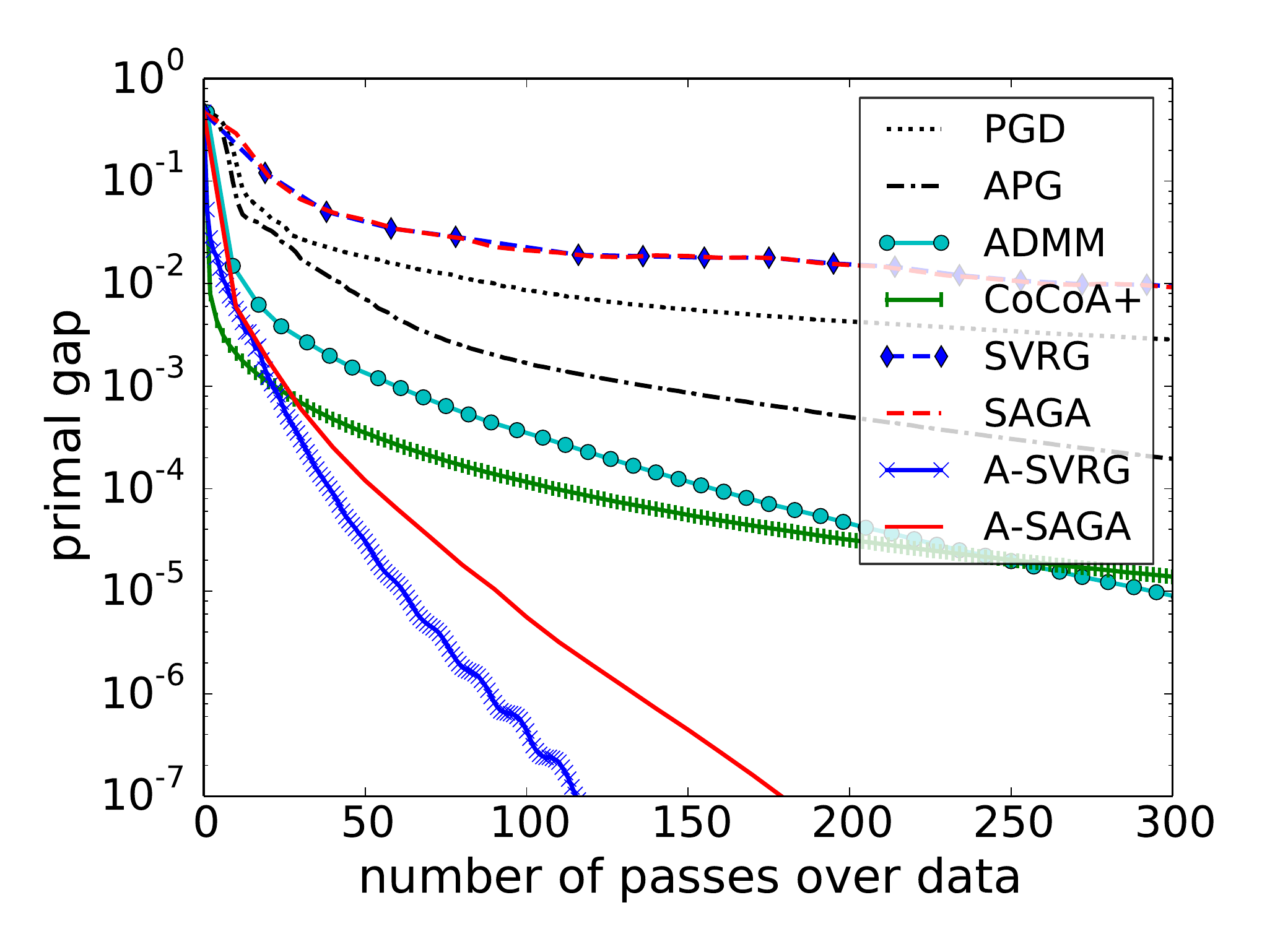}
  \includegraphics[width=0.49\textwidth]{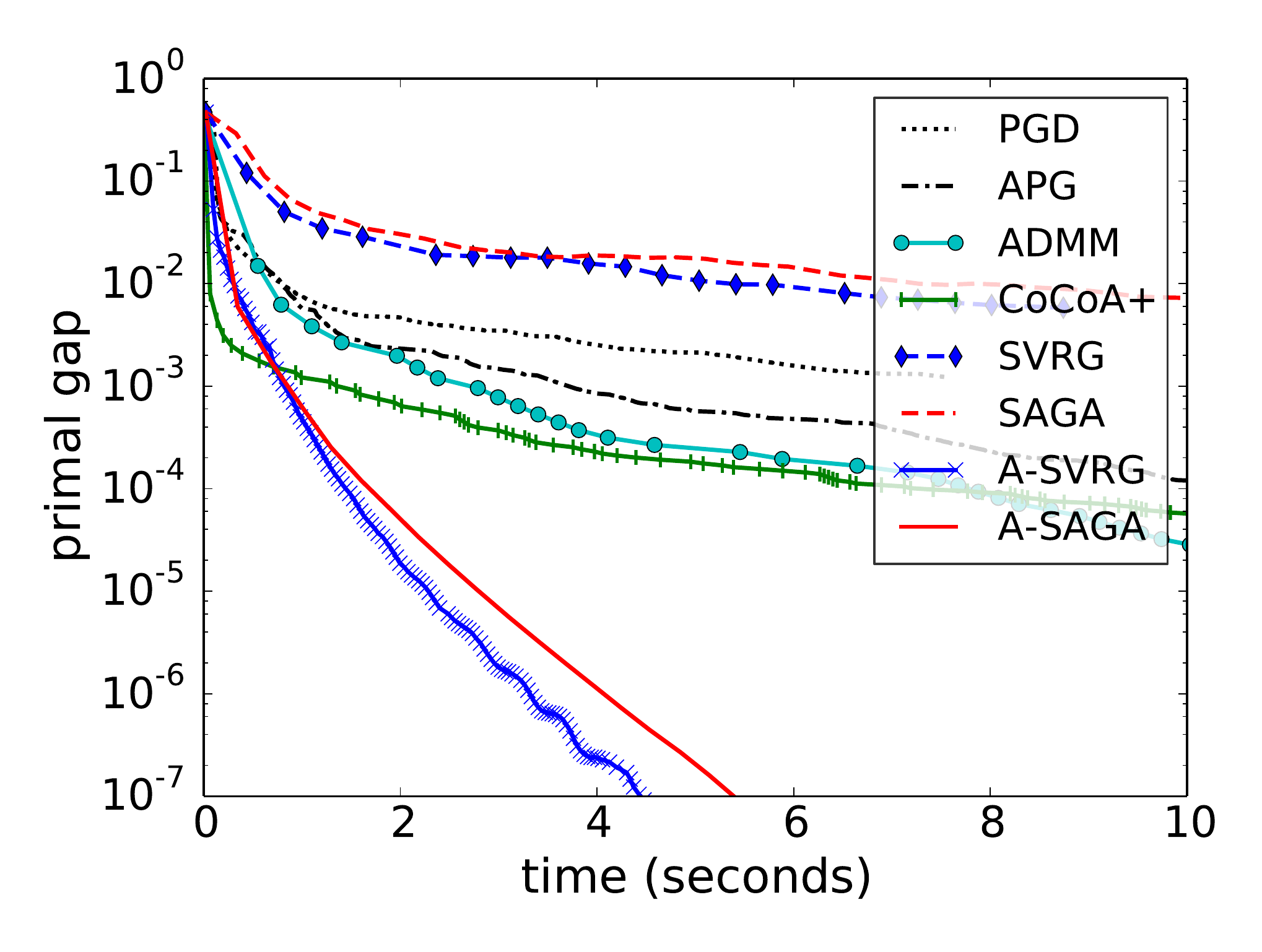}
  \caption{\texttt{rcv1-train}: smoothed-hinge loss, $\lambda\!=\!10^{-6}$,
            randomly shuffled, $m\!=\!20$, $n\!=\!37$, $h\!=\!10$.}
  \label{fig:rcv1train6}
\end{figure}

For the \texttt{webspam} dataset, we solve the ERM problem with logistic
loss $\phi_j(t)=\log(1+\exp(-y_j t))$ where $y_j\in\{\pm 1\}$.
The logistic loss is $1/4$-smooth, so we have $\nu=4$.
Since the proximal mapping of its conjugate $\phi_j^*$ does not have a 
closed-form solution, we used the conjugate-free DSCOVR algorithms
described in Section~\ref{sec:dual-free}.
Figures~\ref{fig:webspam4ini} and~\ref{fig:webspam6ini} shows the 
reduction of primal objective gap by different algorithms, for
$\lambda=10^{-4}$ and $\lambda=10^{-6}$ respectively.
Here the starting point is no longer the all-zero vectors.
Instead, each machine~$i$ first computes a local solution by minimizing
$f_i(X_i w)+g(w)$, and then compute their average using an AllReduce 
operation. Each algorithm starts from this average point.
This averaging scheme has been proven to be very effective to warm
start distributed algorithms for ERM \citep{ZhangDuchi2013averaging}.
In addition, it can be shown that when starting from the zero initial point, 
the first step of CoCoA+ computes exactly such an averaged point.

\begin{figure}[t]
  \centering
  \includegraphics[width=0.49\textwidth]{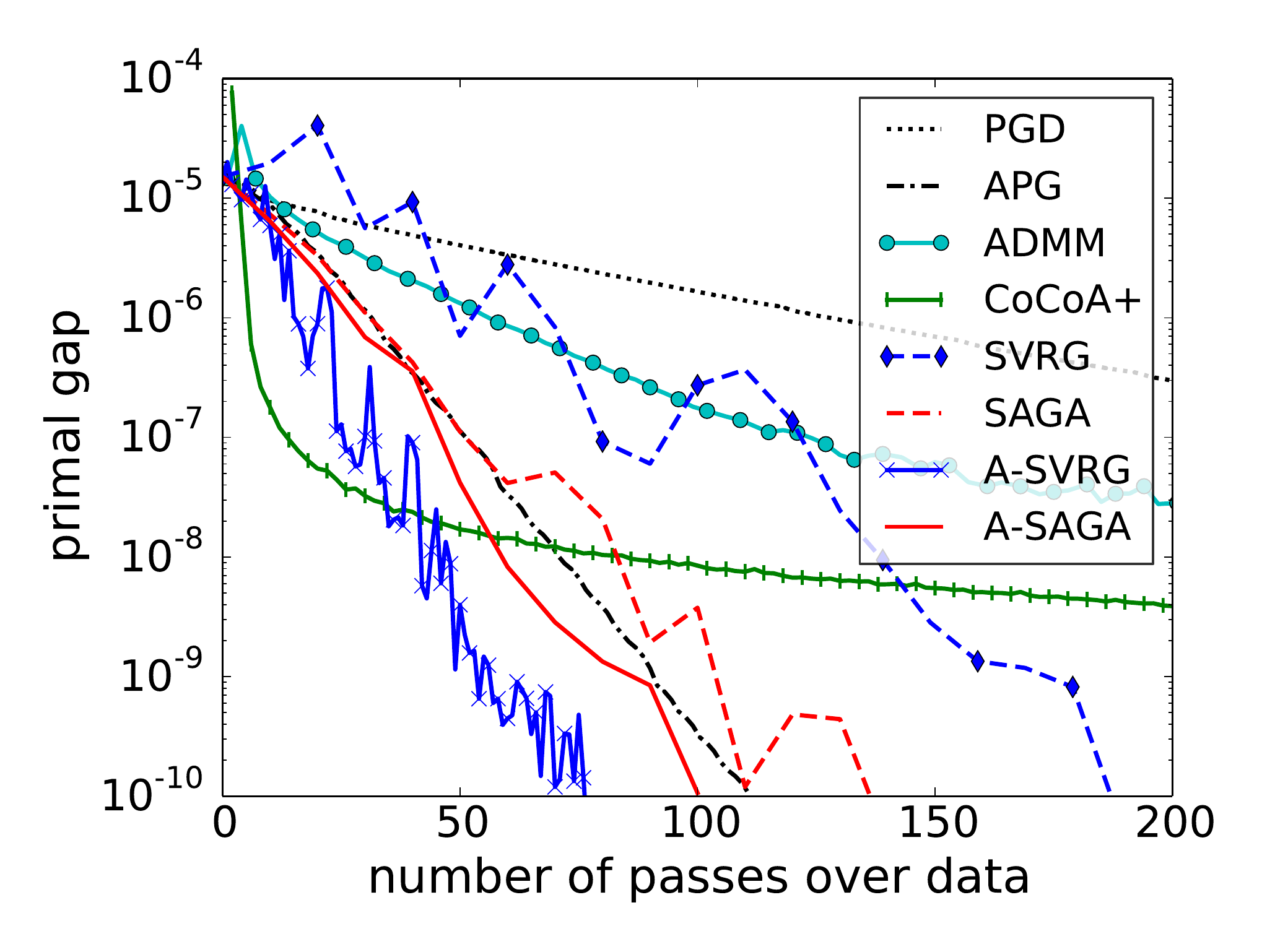}
  \includegraphics[width=0.49\textwidth]{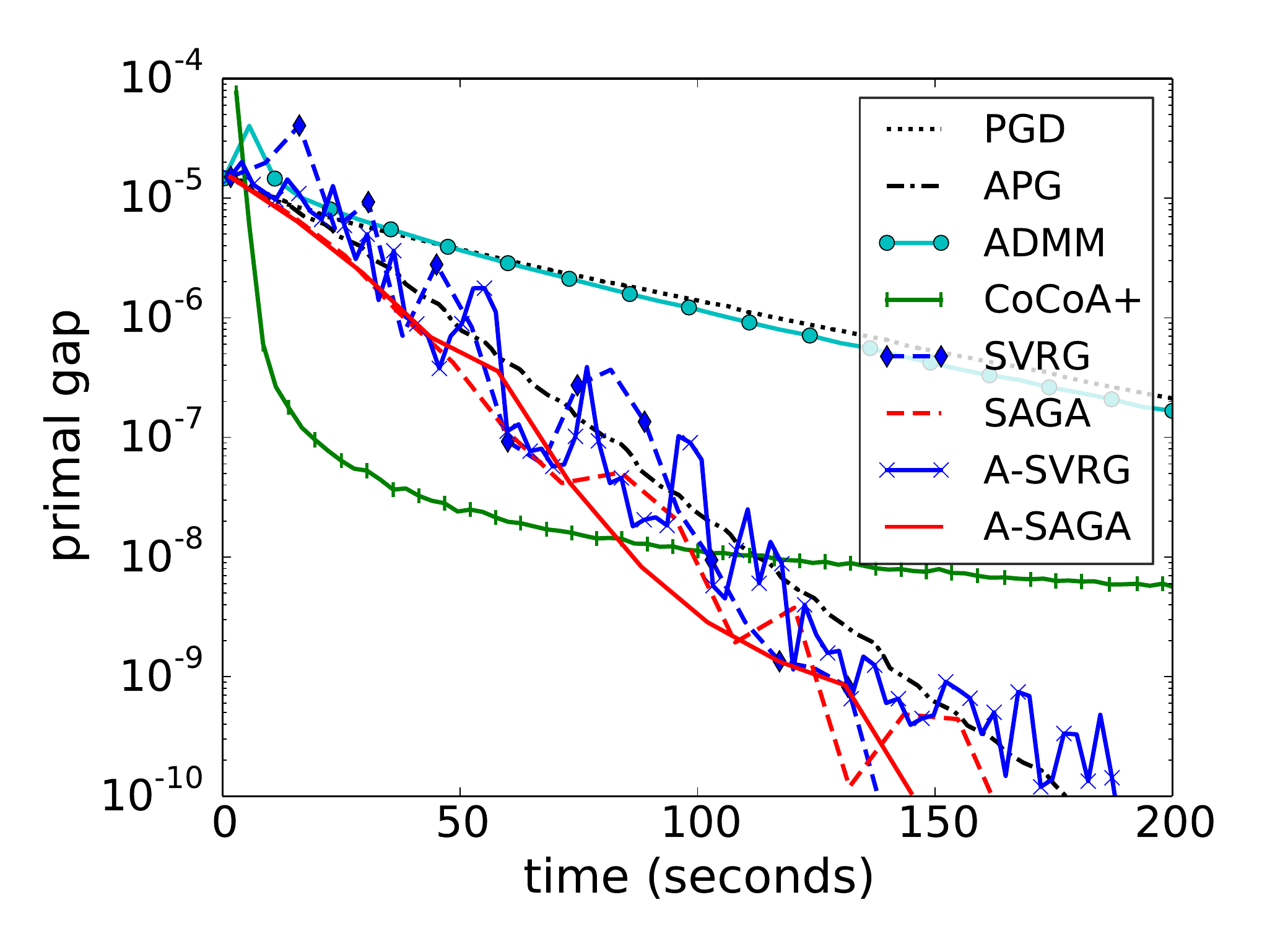}
  \vspace{-3ex}
  \caption{\texttt{webspam}: logistic regression, $\lambda=10^{-4}$, 
            randomly shuffled, $m=20$, $n=50$, $h=10$.}
  \label{fig:webspam4ini}
\end{figure}

\begin{figure}[t]
  \centering
  \includegraphics[width=0.49\textwidth]{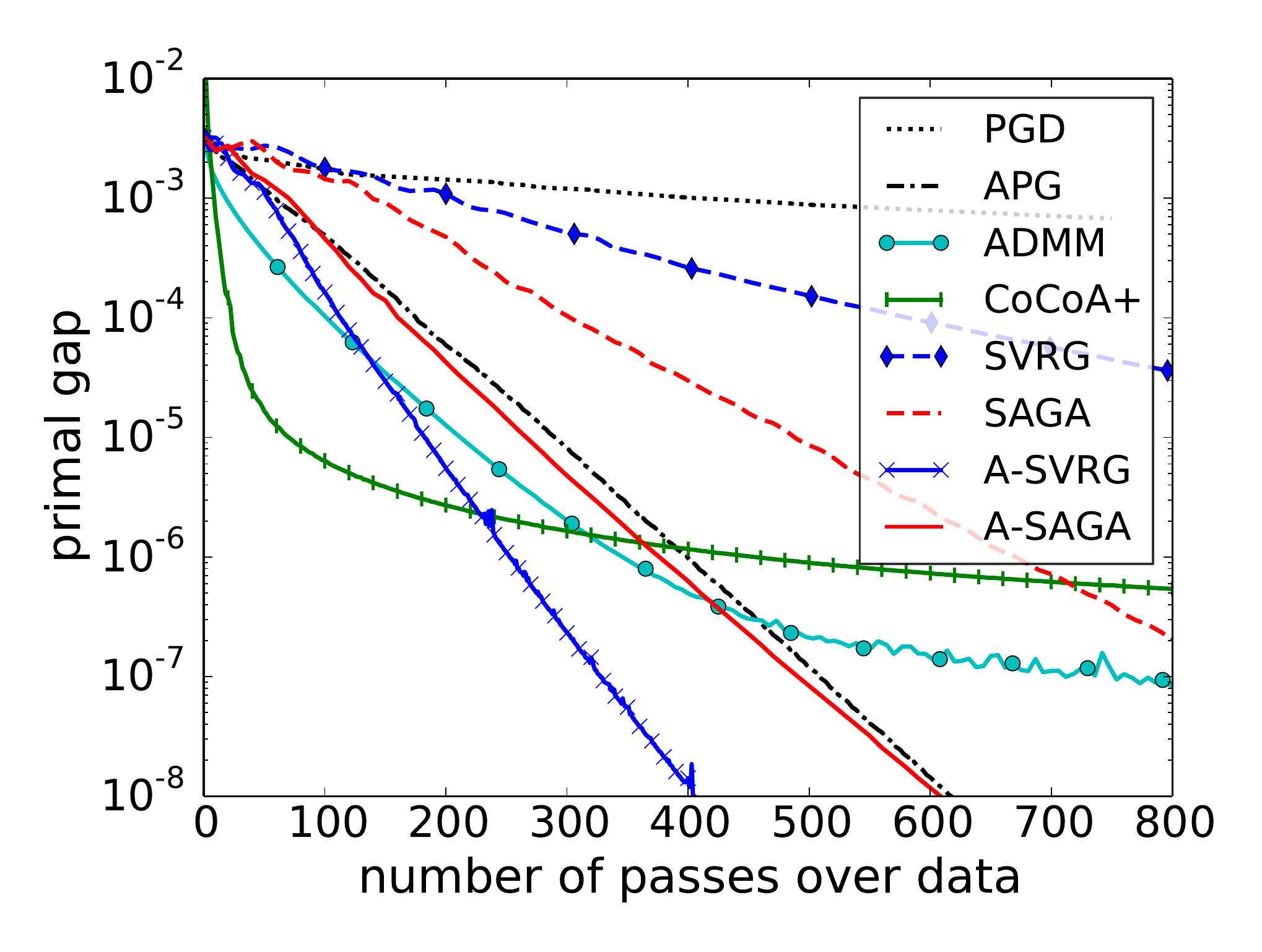}
  \includegraphics[width=0.49\textwidth]{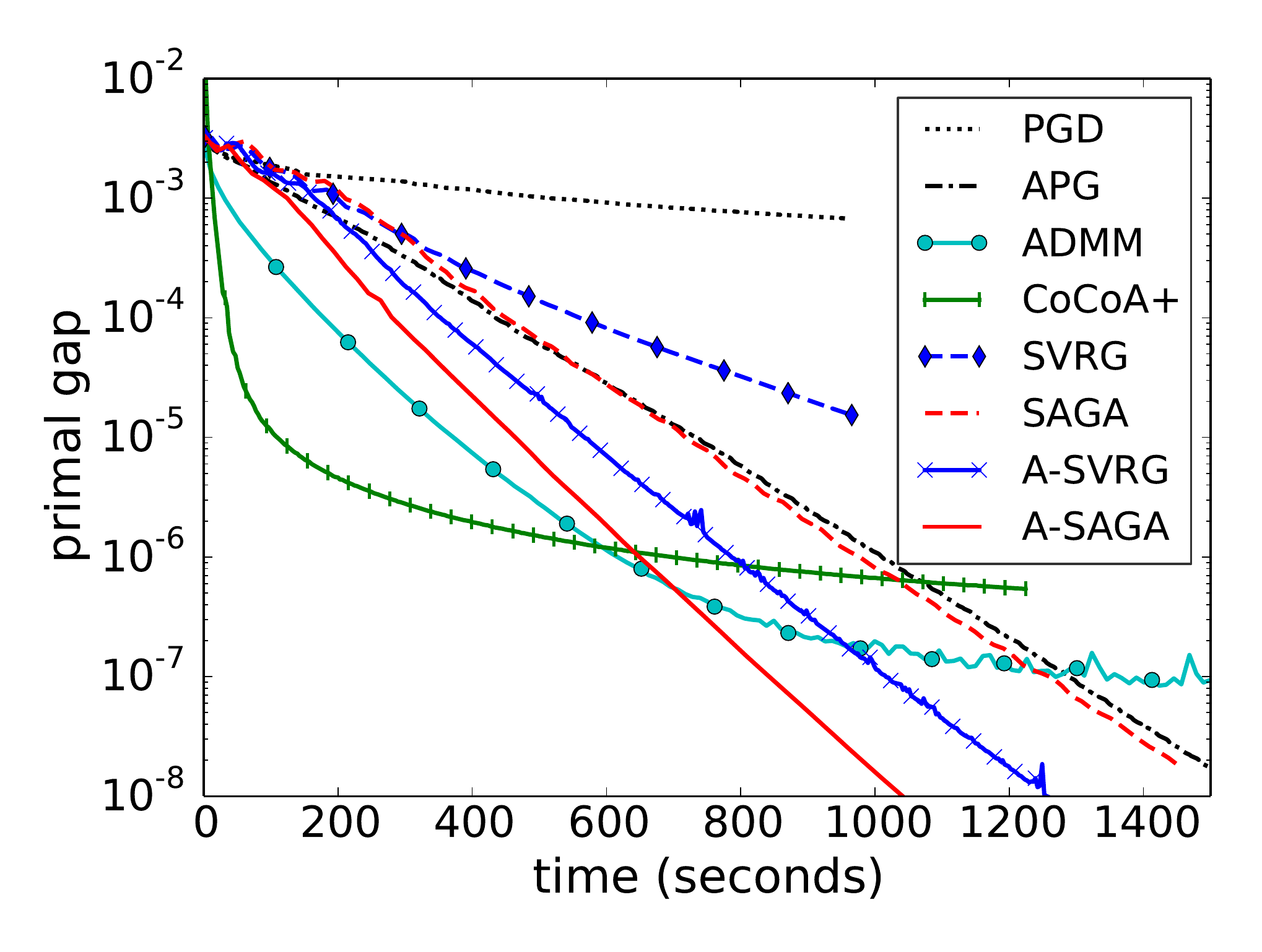}
  \caption{\texttt{webspam}: logistic regression, $\lambda=10^{-6}$, 
            randomly shuffled, $m=20$, $n=50$, $h=10$.}
  \label{fig:webspam6ini}
\end{figure}

From Figures~\ref{fig:webspam4ini} and~\ref{fig:webspam6ini}, we again
observe that CoCoA+ has very fast convergence in the beginning but converges
very slowly towards higher precision.
The DSCOVR algorithms, especially the accelerated variants, 
are very competitive in terms of both number of iterations and wall-clock time.

In order to investigate the fast initial convergence of CoCoA+ and ADMM, 
we repeated the experiments on \texttt{webspam} without random shuffling.
More specifically, we sorted the~$N$ examples by their labels,
and then partitioned them into~$m$ subsets sequentially. 
That is, most of the machines have data with only $+1$ or $-1$ labels, and only
one machine has mixed $\pm 1$ examples.
The results are shown in Figures~\ref{fig:webspamSort4ini}
and~\ref{fig:webspamSort6ini}.
Now the fast initial convergence of CoCoA+ and ADMM disappeared.
In particular, CoCoA+ converges with very slow linear rate.
This shows that statistical properties of random shuffling of the dataset 
is the main reason for the fast initial convergence
of model-averaging based algorithms such as CoCoA+ and ADMM
\citep[see, e.g.,][]{ZhangDuchi2013averaging}.

On the other hand,  this should not have
any impact on PGD and APG, because their iterations are computed over the
whole dataset, which is the same regardless of random shuffling or sorting.
The differences between the plots for PGD and APG in 
Figures~\ref{fig:webspam4ini} and~\ref{fig:webspamSort4ini}
(also for Figures~\ref{fig:webspam6ini} and~\ref{fig:webspamSort6ini})
are due to different initial points computed through averaging local solutions,
which does depends on the distribution of data at different machines.

Different ways for splitting the data over the~$m$ workers also affect
the DSCOVR algorithms. 
In particular, the non-accelerated DSCOVR algorithms become
very slow, as shown in Figures~\ref{fig:webspamSort4ini}
and~\ref{fig:webspamSort6ini}.
However, the accelerated DSCOVR algorithms are still very competitive
against the adaptive APG.
The accelerated DSCOVR-SAGA algorithm performs best.
In fact, the time spent by accelerated DSCOVR-SAGA should be even less than
shown in Figures~\ref{fig:webspamSort4ini} and~\ref{fig:webspamSort6ini}.
Recall that other than the initialization with non-zero starting point,
DSCOVR-SAGA is completely asynchronous and does not need any collective
communication (see Section~\ref{sec:implement-saga}).
However, in order to record the objective function for the purpose of
plotting its progress, we added collective communication and computation
to evaluate the objective value for every 10 passes over the data.
For example, in Figure~\ref{fig:webspamSort6ini}, such extra collective 
communications take about 160 seconds (about 15\% of total time)
for accelerated DSCOVR-SAGA,
which can be further deducted from the horizontal time axis.

\begin{figure}[t]
  \centering
  \includegraphics[width=0.49\textwidth]{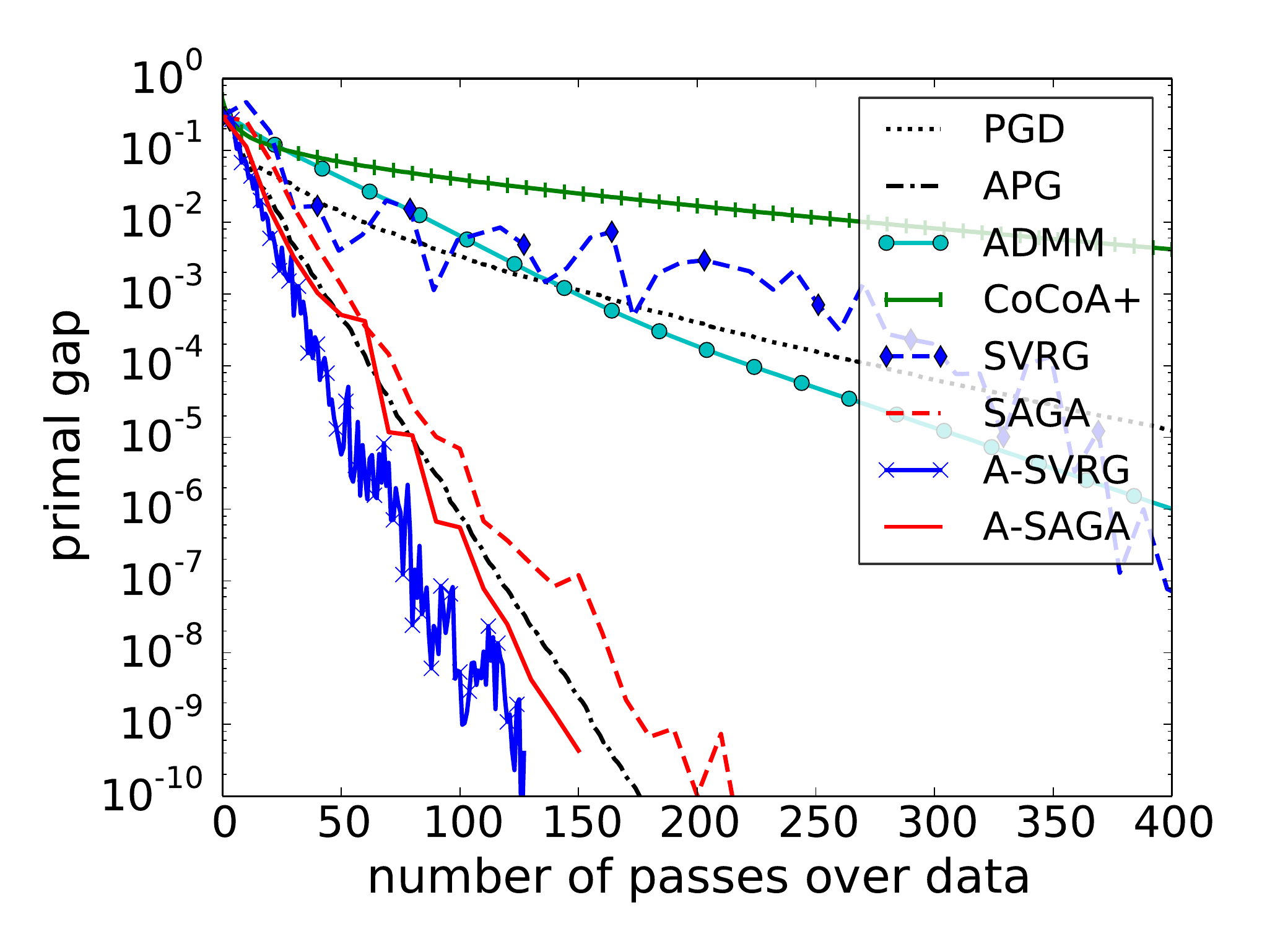}
  \includegraphics[width=0.49\textwidth]{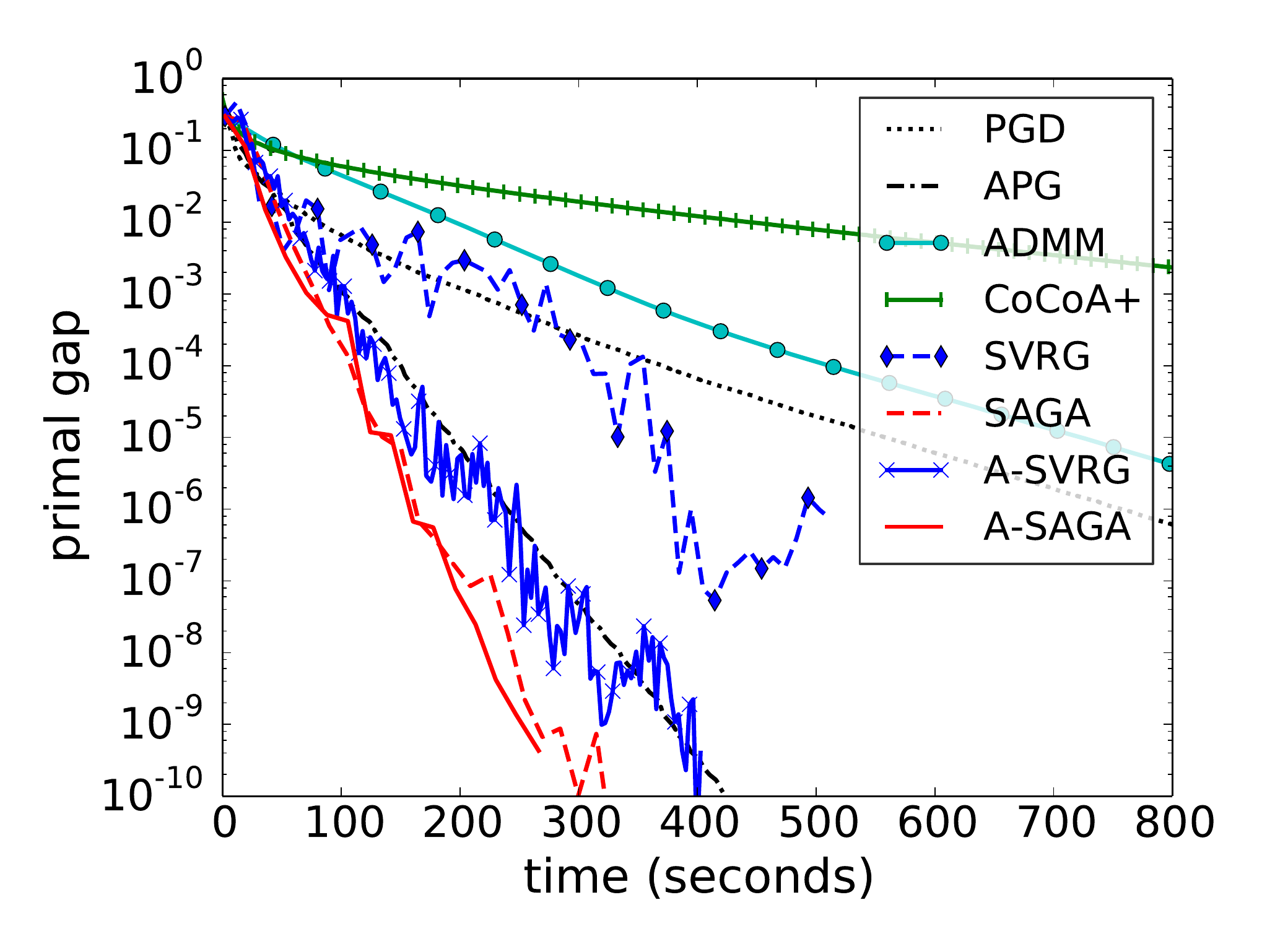}
  \vspace{-3ex}
  \caption{\texttt{webspam}: logistic regression, $\lambda=10^{-4}$, 
            sorted labels, $m=20$, $n=50$, $h=10$.}
  \label{fig:webspamSort4ini}
\end{figure}

\begin{figure}[t]
  \centering
  \includegraphics[width=0.49\textwidth]{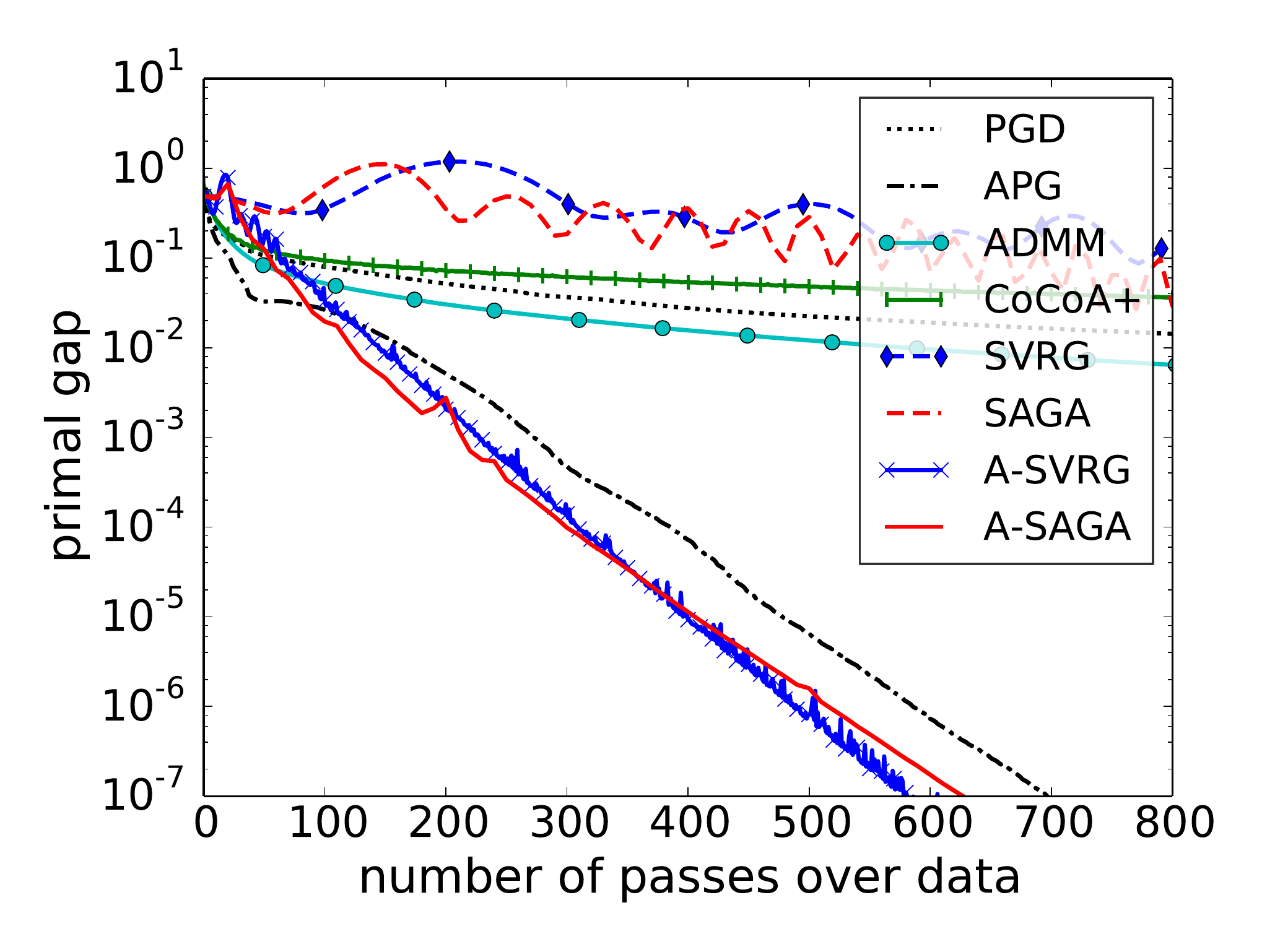}
  \includegraphics[width=0.49\textwidth]{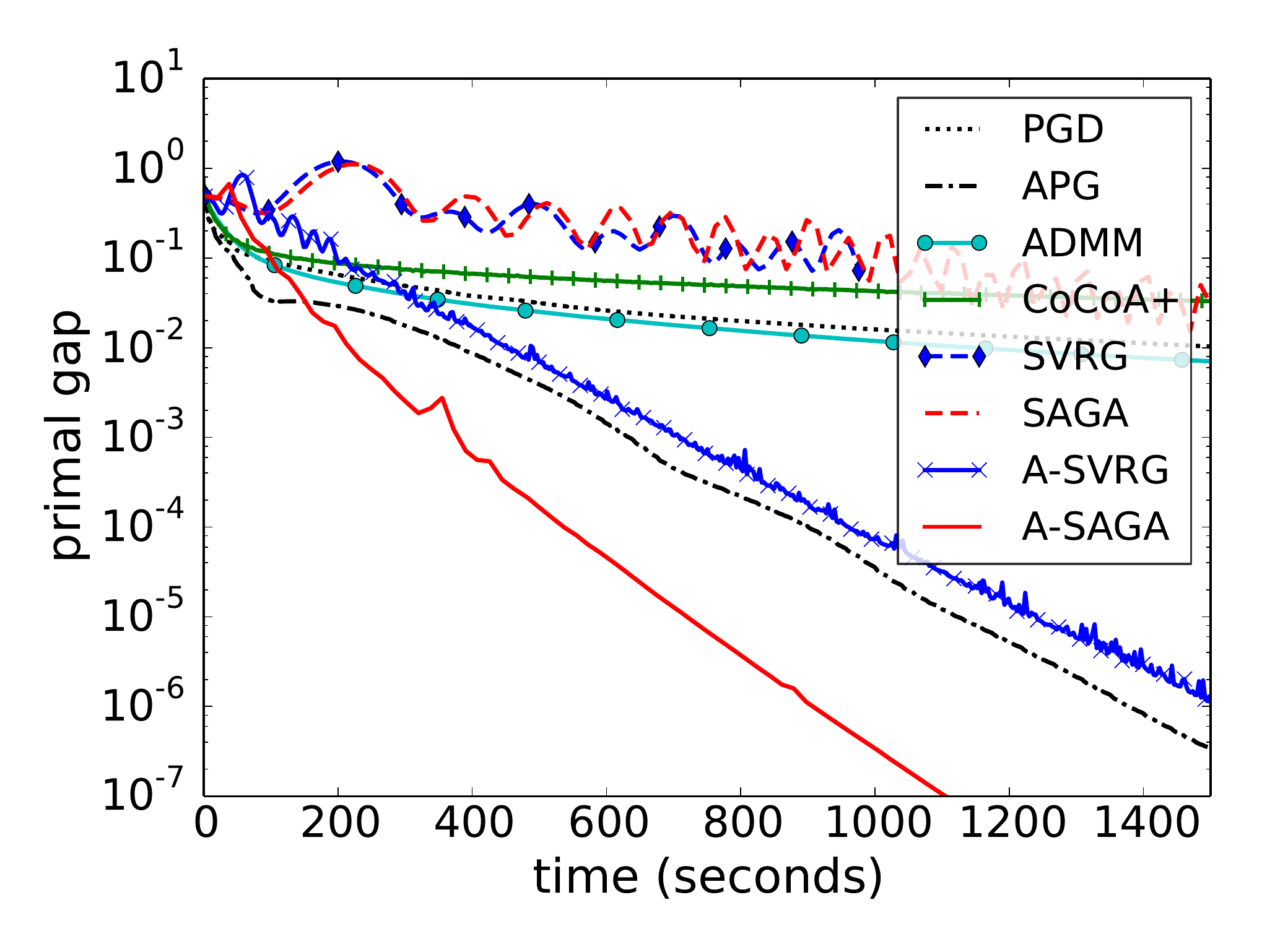}
  \caption{\texttt{webspam}: logistic regression, $\lambda=10^{-6}$, 
            sorted labels, $m=20$, $n=50$, $h=10$.}
  \label{fig:webspamSort6ini}
\end{figure}

Finally, we conducted experiments on the \texttt{splice-site} dataset
with 100 workers and 20 parameter servers.
The results are shown in Figure~\ref{fig:splice6ini}.
Here the dataset is again randomly shuffled and evenly distributed to
the workers.
The relative performance of different algorithms are similar to those
for the other datasets.

\begin{figure}[t]
  \centering
  \includegraphics[width=0.49\textwidth]{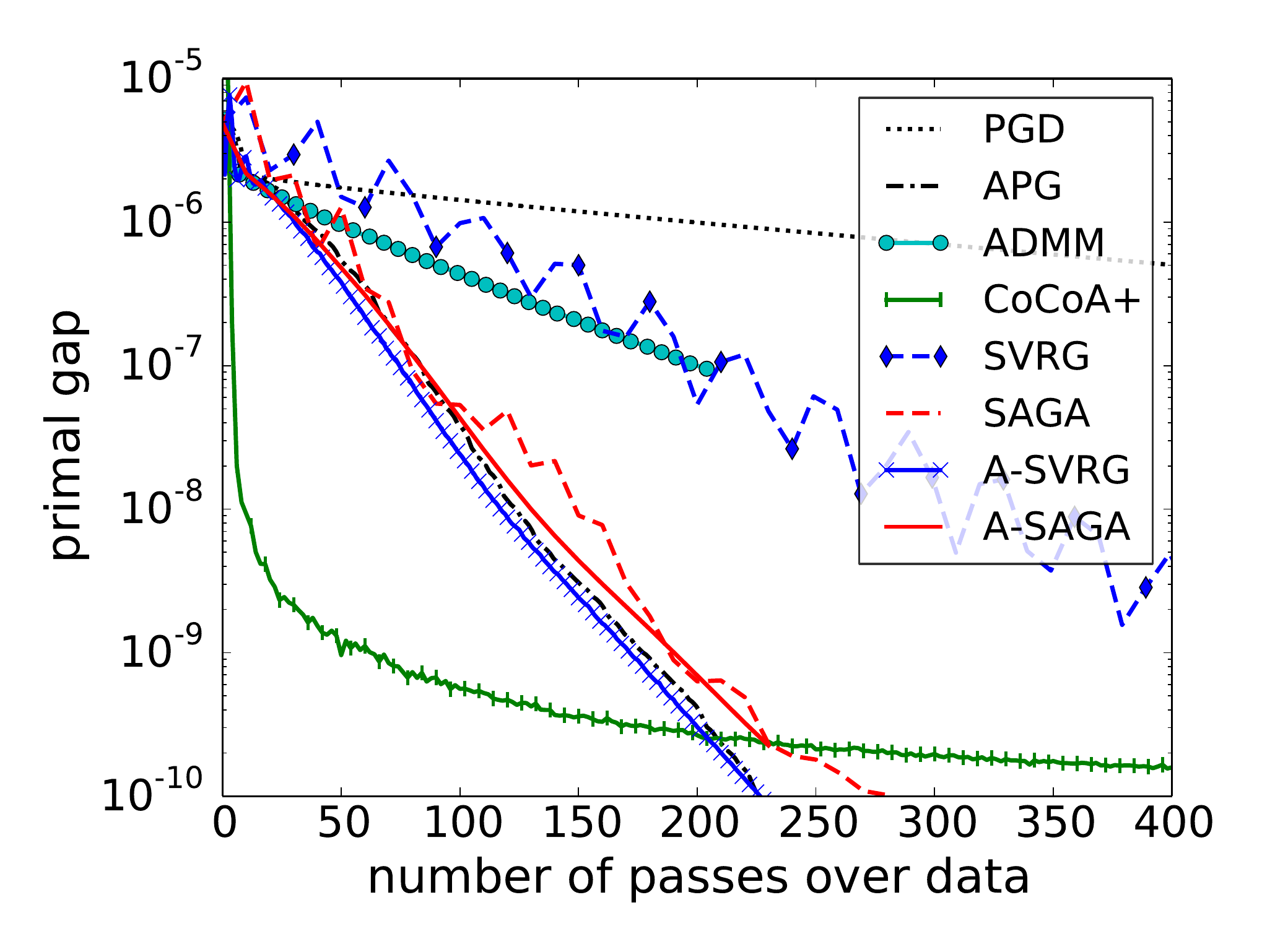}
  \includegraphics[width=0.49\textwidth]{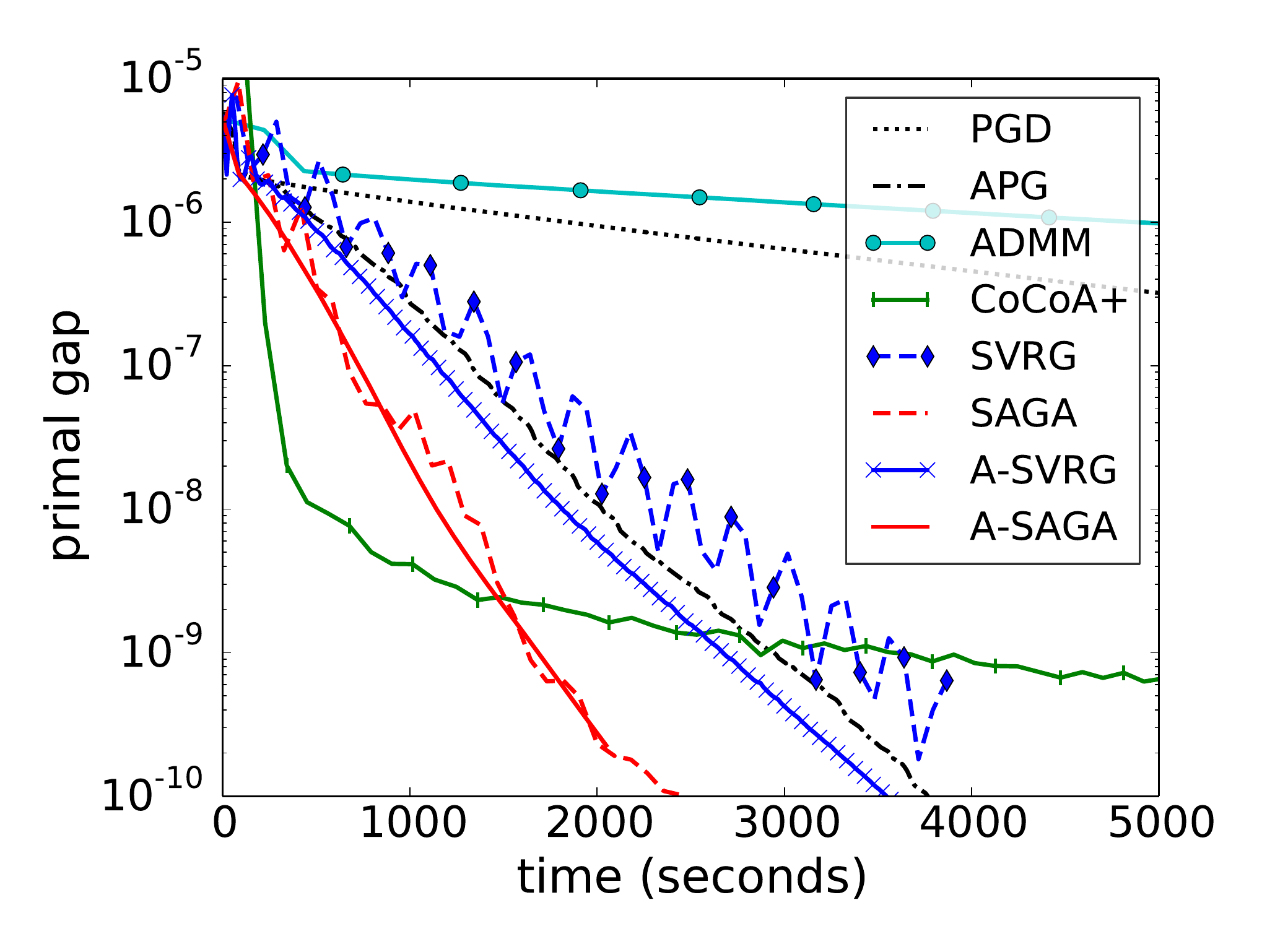}
  \caption{\texttt{splice-site}: logistic loss, $\lambda=10^{-6}$.
            randomly shuffled, $m=100$, $n=150$, $h=20$.}
  \label{fig:splice6ini}
\end{figure}

% Parameters for DSCOVR-SVRG:
% $\eta_p = 20, \eta_d=100$
% Parameters for DSCOVR-SVRG:
% $\eta_p = 100, \eta_d=100$
% Parameters for accelerated DSCOVR-SVRG:
% $\eta_p = 3, \eta_d=30$, PPA period 0.5
% Parameters for accelerated DSCOVR-SAGA:
% $\eta_p = 1, \eta_d=100$, PPA period 1.0 

\section{Conclusions and Discussions}
\label{sec:conclusions}

We proposed a class of DSCOVR algorithms for asynchronous distributed
optimization of large linear models with convex loss functions.
They avoid dealing with delays and stale updates in an asynchronous, 
event-driven environment by exploiting simultaneous data and model parallelism.
Compared with other first-order distributed algorithms, DSCOVR may require
less amount of overall communication and computation, 
and especially much less or no synchronization.
These conclusions are well supported by our computational experiments on 
a distributed computing cluster.
% in a real data center setting.

We note that there is still some gap between theory and practice.
In our theoretical analysis, 
we assume that the primal and dual block indices in 
different iterations of DSCOVR are i.i.d.\ random variables, 
sampled sequentially with replacement.
But the parallel implementations described in Section~\ref{sec:implementation}
impose some constraints on how they are generated. 
In particular, the parameter block to be updated next is randomly chosen from 
the set of blocks that are not being updated by any worker simultaneously,
and the next worker available is event-driven, depending on the loads and 
processing power of different workers as well as random communication latency.
These constraints violate the i.i.d.\ assumption, but our experiments 
show that they still lead to very competitive performance.
Intuitively some of them can be potentially beneficial, 
reminiscent of the practical advantage of sampling without
replacement over sampling with replacement in randomized coordinate descent
methods \citep[e.g.,][]{SSZhang13SDCA}.
This is an interesting topic worth future study.

In our experiments, the parallel implementation of Nesterov's accelerated 
gradient method (APG) is very competitive on all the datasets we tried
and for different regularization parameters used. 
In addition to the theoretical justifications 
in \citet{ArjevaniShamir2015NIPS} and \citet{ScamanBach2017}, 
the adaptive line-search scheme turns out to be critical for its good 
performance in practice.
The accelerated DSCOVR algorithms demonstrated comparable or better performance
than APG in the experiments, but need careful tuning of the constants
in their step size formula.
On one hand, it supports our theoretical results that DSCOVR is capable of
outperforming 
other first-order algorithms including APG, in terms of both communication
and computation complexity (see Section~\ref{sec:main-results}).
On the other hand, there are more to be done in order to realize the full 
potential of DSCOVR in practice.
In particular, we plan to follow the ideas in \cite{WangXiao2017ICML}
to develop adaptive schemes that can automatically
tune the step size parameters, as well as exploit strong convexity from data.

\section*{Acknowledgment}
\label{sec:ack}

Adams Wei Yu is currently supported by NVIDIA PhD Fellowship. The authors would like to thank Chiyuan Zhang for helpful discussion of the system implementation.

\appendix

\section{Proof of Theorem~\ref{thm:dscovr-svrg}}
\label{sec:proof-svrg}

We first prove two lemmas concerning the primal and dual proximal updates
in Algorithm~\ref{alg:dscovr-svrg}.
Throughout this appendix, $\E_t[\cdot]$ denotes the conditional expectation
taken with respect to the random indices~$j$ and~$l$ generated during the
$t$th inner iteration in Algorithm~\ref{alg:dscovr-svrg}, conditioned on all 
quantities available at the beginning of the $t$th iteration,
including $\wt$ and $\at$.
Whenever necessary, we also use the notation $j^{(t)}$ and $l^{(t)}$ to denote
the random indices generated in the $t$th iteration.

\begin{lemma}\label{lem:dual-expect-t}
For each $i=1,\ldots,m$, let $\uitp\in\R^{N_i}$ be a random variable and define
\begin{equation}\label{eqn:aitld-prox}
\aitld = \prox_{\sigma_i f_i^*}\bigl(\ait+\sigma_i\uitp\bigr).
\end{equation}
We choose an index~$j$ randomly from $\{1,\ldots,m\}$ with probability
distribution $\bigl\{p_j\bigr\}_{j=1}^m$ and let
\[
\aitp = \left\{ \begin{array}{ll} \aitld & \textrm{if}~i=j, \\
	\ait & \textrm{otherwise}. \end{array} \right.
\]
If each $\uitp$ is independent of~$j$ and satisfies 
$\E_t\bigl[\uitp\bigr]=\Xri\wt$ for $i=1,\ldots,m$, then we have
\begin{eqnarray}
&&\sum_{i=1}^{m}\left( \frac{1}{p_i} \left(\frac{1}{2\sigma_i}+\gamma_i\right)
  - \gamma_i\right) \|\ait-\aiopt\|^2 \nonumber\\
&\geq&  \sum_{i=1}^m \frac{1}{p_i}\left(\frac{1}{2\sigma_i}+\gamma_i\right)
  \E_t[\|\aitp-\aiopt\|^2] 
  + \sum_{i=1}^m \frac{1}{p_i}\left(\frac{1}{2\sigma_i}-\frac{1}{a_i}
  \right) \E_t[\|\aitp-\ait\|^2] \nonumber \\
&& - \sum_{i=1}^m \frac{a_i}{4} \E_t\bigl[\|\uitp-\Xri \wt\|^2\bigr]
  + \left\langle \wt - \wopt,\, X^T(\aopt-\at) \right\rangle \nonumber \\
&& - \sum_{i=1}^m \frac{1}{p_i} \E_t \!\left[\left\langle\aitp-\ait,\,
  \Xri(\wt-\wopt) \right\rangle \right] ,
\label{eqn:dual-one-step-sum}
\end{eqnarray}
where $(\wopt\! ,\aopt)$ is the saddle point of $L(w,\alpha)$ defined
in~\eqref{eqn:Lagrangian}, and the $a_i$'s are arbitrary positive numbers.
\end{lemma}

\begin{proof}
First, consider a fixed index $i\in\{1,\ldots,m\}$. 
The definition of $\aitld$ in~\eqref{eqn:aitld-prox} is equivalent to
\begin{equation}\label{eqn:alpha-tilde}
\aitld = \argmin_{\beta\in\R^{N_i}} \left\{
  f_i^*(\beta) - \bigl\langle \beta, \uitp \bigr\rangle 
  + \frac{\|\beta-\ait\|^2}{2\sigma_i} \right\}.
\end{equation}
By assumption, $f_i^*(\beta)$ and $\frac{1}{2\sigma_i}\|\beta-\ait\|^2$ 
are strongly convex with convexity parameters $\gamma_i$ and 
$\frac{1}{\sigma_i}$ respectively. Therefore, the objective function 
in~\eqref{eqn:alpha-tilde} is $(\frac{1}{\sigma_i}+\gamma_i)$-strongly convex,
which implies
\begin{eqnarray}
&& \frac{\|\aiopt-\ait\|^2}{2\sigma_i} 
	- \bigl\langle\aiopt,\,\uitp\bigr\rangle +f_i^*(\aiopt) \nonumber \\
&\geq& \frac{\|\aitld-\ait\|^2}{2\sigma_i} 
	- \bigl\langle\aitld,\,\uitp\bigr\rangle +f_i^*(\aitld) 
+ \left(\frac{1}{\sigma_i}+\gamma_i\right)\frac{\|\aitld-\aiopt\|^2}{2}.
\label{eqn:dual-update-sc}
\end{eqnarray}
In addition, since $(\wopt,\aopt)$ is the saddle-point of $L(w,\alpha)$,
the function $f_i^*(\ai)-\langle\ai,\Xri\wopt\rangle$
is $\gamma_i$-strongly convex in~$\ai$ and attains its minimum at $\aiopt$.
Thus we have
\[
  f_i^*\bigl(\aitld\bigr)-\bigl\langle\aitld,\, \Xri\wopt\bigr\rangle
  \geq
  f_i^*(\aiopt)-\bigl\langle\aiopt,\, \Xri\wopt\bigr\rangle
  +\frac{\gamma_i}{2}\|\aitld-\aiopt\|^2.
\]
Summing up the above two inequalities gives
\begin{eqnarray*}
\frac{\|\aiopt-\ait\|^2}{2\sigma_i}
&\geq &
\frac{\|\aitld-\ait\|^2}{2\sigma_i}
+ \Bigl(\frac{1}{2\sigma_i}+\gamma_i\Bigr) \|\aitld-\aiopt\|^2
+ \left\langle\aiopt-\aitld,\, \uitp-\Xri\wopt\right\rangle \\
&=& \frac{\|\aitld-\ait\|^2}{2\sigma_i}
+ \Bigl(\frac{1}{2\sigma_i}+\gamma_i\Bigr) \|\aitld-\aiopt\|^2
+ \left\langle\aiopt-\aitld,\, \Xri(\wt-\wopt)\right\rangle \\
&& + \left\langle\aiopt-\ait,\,\uitp-\Xri\wt \right\rangle
+ \left\langle\ait-\aitld,\,\uitp-\Xri\wt \right\rangle \\
&\geq& \frac{\|\aitld-\ait\|^2}{2\sigma_i}
+ \Bigl(\frac{1}{2\sigma_i}+\gamma_i\Bigr) \|\aitld-\aiopt\|^2
+ \left\langle\aiopt-\aitld,\, \Xri(\wt-\wopt)\right\rangle \\
&& + \left\langle\aiopt-\ait,\,\uitp-\Xri\wt \right\rangle
 - \frac{\|\ait-\aitld\|^2}{a_i} - \frac{a_i\|\uitp-\Xri\wt\|^2}{4},
\end{eqnarray*}
where in the last step we used Young's inequality 
with $a_i$ being an arbitrary positive number.
Taking conditional expectation $\E_t$ on both sides of the above inequality,
and using the assumption $\E_t[\uitp]=\Xri\wt$, we have
\begin{eqnarray}
\frac{\|\aiopt\!\!-\!\ait\|^2}{2\sigma_i}
\!\!\!\!&\geq&\!\!\!\!
\frac{\E_t\!\bigl[\|\aitld\!\!-\!\ait\|^2\bigr]}{2\sigma_i}
+ \Bigl(\frac{1}{2\sigma_i}\!+\!\gamma_i\Bigr) \E_t[\|\aitld\!\!-\!\aiopt\|^2]
+ \E_t\!\left[\left\langle\aiopt\!\!-\!\aitld\!, \Xri(\wt\!\!-\!\wopt)
  \right\rangle\right] \nonumber \\
&& - \frac{\E_t[\|\ait-\aitld\|^2]}{a_i} 
- \frac{a_i\E_t[\|\uitp-\Xri\wt\|^2]}{4}. 
\label{eqn:dual-one-step-lb}
\end{eqnarray}

Notice that each $\aitld$ depends on the random variable $\uitp$ and is
independent of the random index~$j$.
But $\aitp$ depends on both $\uitp$ and~$j$. 
Using the law of total expectation, 
\[
\E_t[\,\cdot\,] ~=~ \Prob(j=i)\E_t[\,\cdot\,|j=i] 
 + \Prob(j\neq i) \E_t[\,\cdot\,|j\neq i] ,
\]
we obtain
\begin{eqnarray}
\E_t[\aitp] &=& p_i \E_t[\aitld] + (1-p_i) \ait, 
\label{eqn:alpha-cond-mean} \\
\E_t[\|\aitp-\ait\|^2] &=& p_i\E_t[\|\aitld-\ait\|^2], 
\label{eqn:alpha-change-norm2} \\
\E_t[\|\aitp-\aiopt\|^2] &=& p_i\E_t[\|\aitld-\aiopt\|^2]
  + (1-p_i) \E_t[\|\ait-\aiopt\|^2] .
\label{eqn:alpha-change-opt}
\end{eqnarray}
Next, using the equalities~\eqref{eqn:alpha-cond-mean},
\eqref{eqn:alpha-change-norm2} and~\eqref{eqn:alpha-change-opt},
we can replace each term in~\eqref{eqn:dual-one-step-lb} containing
$\aitld$ with terms that contain only $\ait$ and $\aitp$.
By doing so and rearranging terms afterwards, we obtain
\begin{eqnarray*}
&&\left(\frac{1}{p_i}\left(\frac{1}{2\sigma_i}+\gamma_i\right)-\gamma_i\right)
\|\ait-\aiopt\|^2 \\
&\geq& \frac{1}{p_i}\left(\frac{1}{2\sigma_i}+\gamma_i\right)
\E_t[\|\aitp-\aiopt\|^2
+ \frac{1}{p_i}\left(\frac{1}{2\sigma_i}-\frac{1}{a_i}\right)
\E_t[\|\aitp-\ait\|^2 \\
&&-\, \frac{a_i\E_t[\|\uitp-\Xri\wt\|^2]}{4} 
+ \left\langle\aiopt-\ait,\,\Xri(\wt-\wopt)\right\rangle \\
&& -\,\E_t\Bigl[\Bigl\langle\frac{1}{p_i}(\aitp-\ait),\, 
\Xri(\wt-\wopt)\Bigr\rangle\Bigr].
\end{eqnarray*}
Summing up the above inequality for $i=1,\ldots,m$ gives the desired
result in~\eqref{eqn:dual-one-step-sum}.
\end{proof}

\begin{lemma}\label{lem:primal-expect-t}
For each $k=1,\ldots,n$, let $\vktp\in\R^{d_i}$ be a random variable and define
\[
\wktld = \prox_{\tau_k g_k}\bigl(\wkt-\tau_k\vktp\bigr).
\]
We choose an index~$l$ randomly from $\{1,\ldots,n\}$ with probability
distribution $\bigl\{q_l\bigr\}_{l=1}^n$ and let
\[
\wktp = \left\{ \begin{array}{ll} \wktld & \textrm{if}~k=l, \\
	\wkt & \textrm{otherwise}. \end{array} \right.
\]
If each $\vktp$ is independent of~$l$ and satisfies 
$\E_t\bigl[\vktp\bigr]=\frac{1}{m}(\Xck)^T\at$, then we have
\begin{eqnarray}
&&\sum_{k=1}^{n}\left( \frac{1}{q_k} \left(\frac{1}{2\tau_k}+\lambda\right)
  - \lambda\right) \|\wkt-\wkopt\|^2 \nonumber\\
&\geq&  \sum_{k=1}^n \frac{1}{q_k}\left(\frac{1}{2\tau_k}+\lambda\right)
  \E_t[\|\wktp-\wkopt\|^2] 
  + \sum_{k=1}^n \frac{1}{q_k}\left(\frac{1}{2\tau_k}-\frac{1}{b_k}
  \right) \E_t[\|\wktp-\wkt\|^2] \nonumber \\
&& -\sum_{k=1}^n\frac{b_k}{4}
  \E_t\!\left[\Bigl\|\vktp-\frac{1}{m}(\Xck)^T\at\Bigr\|^2\right]
  + \frac{1}{m}\left\langle  X(\wt-\wopt),\,
  \at - \aopt  \right\rangle \nonumber \\
&& + \sum_{k=1}^n \frac{1}{q_k} \E_t \!\left[\left\langle\wktp-\wkt,\,
\frac{1}{m}(\Xck)^T(\at-\aopt) \right\rangle \right] ,
\label{eqn:primal-one-step-sum}
\end{eqnarray}
where $(\wopt,\aopt)$ is the saddle point of $L(w,\alpha)$ defined
in~\eqref{eqn:Lagrangian},
and the $b_i$'s are arbitrary positive numbers.
\end{lemma}

Lemma~\ref{lem:primal-expect-t} is similar to Lemma~\ref{lem:dual-expect-t} 
and can be proved using the same techniques. 
Based on these two lemmas, we can prove the following proposition.
\begin{proposition}
\label{prop:svrg-primal-dual}
The $t$-th iteration within the $s$-th stage of
Algorithm~\ref{alg:dscovr-svrg} guarantees
\begin{eqnarray}
&& \sum_{i=1}^m \frac{1}{m}\left[\frac{1}{p_i}\left(\frac{1}{2\sigma_i}
+\gamma_i\right)-\gamma_i+\sum_{k=1}^n\frac{3\tau_k\|\Xik\|^2}{m p_i} \right]
\|\ait-\aiopt\|^2 
+ \sum_{i=1}^m\sum_{k=1}^n\frac{2\tau_k\|\Xik\|^2}{m^2 p_i}
\|\abis-\aiopt\|^2  \nonumber \\
&&\!\!\!\!\! +\sum_{k=1}^n \left[\frac{1}{q_k}\left(\frac{1}{2\tau_k}
+\lambda\right)-\lambda+\sum_{i=1}^m\frac{3\sigma_i\|\Xik\|^2}{m q_k} \right]
\|\wkt-\wkopt\|^2 
+ \sum_{i=1}^m\sum_{k=1}^n\frac{2\sigma_i\|\Xik\|^2}{m q_k}
\|\wbks-\wkopt\|^2  \nonumber \\
&\geq& \sum_{i=1}^m \frac{1}{m p_i}\left(\frac{1}{2\sigma_i}+\gamma_i\right)
\E_t\bigl[\|\aitp-\aiopt\|^2\bigr]
+\sum_{k=1}^n \frac{1}{q_k}\left(\frac{1}{2\tau_k}+\lambda\right)
\E_t\bigl[\|\wktp-\wkopt\|^2\bigr].
\label{eqn:svrg-pd-one-step}
\end{eqnarray}
\end{proposition}

\begin{proof}
Multiplying both sides of the inequality~\eqref{eqn:dual-one-step-sum}
by~$\frac{1}{m}$ and adding to the inequality~\eqref{eqn:primal-one-step-sum}
gives
\begin{eqnarray}
&&\sum_{i=1}^{m}\frac{1}{m}\left( \frac{1}{p_i} 
  \left(\frac{1}{2\sigma_i}+\gamma_i\right)
  - \gamma_i\right) \|\ait-\aiopt\|^2 
+ \sum_{k=1}^{n}\left( \frac{1}{q_k} \left(\frac{1}{2\tau_k}+\lambda\right)
  - \lambda\right) \|\wkt-\wkopt\|^2 \nonumber\\
&\geq&  \sum_{i=1}^m \frac{1}{m p_i}\left(\frac{1}{2\sigma_i}+\gamma_i\right)
  \E_t[\|\aitp-\aiopt\|^2] 
  + \sum_{k=1}^n \frac{1}{q_k}\left(\frac{1}{2\tau_k}+\lambda\right)
  \E_t[\|\wktp-\wkopt\|^2] \nonumber \\
&&+ \sum_{i=1}^m \frac{1}{m p_i}\left(\frac{1}{2\sigma_i}-\frac{1}{a_i}
  \right) \E_t[\|\aitp-\ait\|^2] 
  + \sum_{k=1}^n \frac{1}{q_k}\left(\frac{1}{2\tau_k}-\frac{1}{b_k}
  \right) \E_t[\|\wktp-\wkt\|^2] \nonumber \\
&& -\sum_{k=1}^n\frac{b_k}{4}
  \E_t\!\left[\Bigl\|\vktp-\frac{1}{m}(\Xck)^T\at\Bigr\|^2\right] 
 + \sum_{k=1}^n \frac{1}{q_k} \E_t \!\left[\left\langle\wktp-\wkt,\,
  \frac{1}{m}(\Xck)^T(\at-\aopt) \right\rangle \right] \nonumber \\
&& - \sum_{i=1}^m \frac{a_i}{4m} \E_t[\|\uitp-\Xri \wt\|^2]
 - \sum_{i=1}^m \frac{1}{m p_i} \E_t \!\left[\left\langle\aitp-\ait,\,
  \Xri(\wt-\wopt) \right\rangle \right]. 
\label{eqn:primal-dual-sum} 
\end{eqnarray}
We notice that the terms containing
$ \frac{1}{m}\left\langle  X(\wt-\wopt),\, \at - \aopt  \right\rangle$
from~\eqref{eqn:dual-one-step-sum} and~\eqref{eqn:primal-one-step-sum} 
canceled each other. 
Next we bound the last four terms on the right-hand side 
of~\eqref{eqn:primal-dual-sum}.

As in Algorithm~\ref{alg:dscovr-svrg}, for each $i=1,\ldots,m$, we define
a random variable
\[
\uitp=\ubis- \frac{1}{q_l}\Xil\wbls + \frac{1}{q_l}\Xil\wlt ,
\]
which depends on the random index~$l\in\{1,\ldots,n\}$.
Taking expectation with respect to~$l$ yields 
\[
\E_t[\uitp]=\sum_{k=1}^{n}q_k\left(\ubis- \frac{1}{q_k}\Xik\wbks
+ \frac{1}{q_k}\Xik\wkt\right) = \Xri\wt, 
\qquad i=1,2,\dots,m.
\]
Therefore $\uitp$ satisfies the assumption in Lemma~\ref{lem:dual-expect-t}. 
In order to bound its variance, we notice that
\[
 \sum_{k=1}^n q_k\left(\frac{1}{q_k}\Xik\wbks-\frac{1}{q_k}\Xik\wkt\right)
=\Xri\wbs - \Xri\wt = \ubis- \Xri\wt. 
\]
Using the relation between variance and the second moment, we have
\begin{eqnarray}
  \E_t\bigl[\|\uitp-\Xri\wt\|^2\bigr]
  &=& \sum_{k=1}^n q_k\Bigl\|\ubis-\frac{1}{q_k}\Xik\wbks
  +\frac{1}{q_k}\Xik\wkt - \Xri\wt \Bigr\|^2 \nonumber \\
  &=& \sum_{k=1}^n \frac{1}{q_k}\Bigl\|\Xik\wbks-\Xik\wkt\Bigr\|^2
  -\|\ubis-\Xri\wt\|^2 \nonumber \\
  &\leq& \sum_{k=1}^n \frac{1}{q_k}\Bigl\|\Xik(\wbks-\wkt)\Bigr\|^2
  \nonumber \\
  &\leq& \sum_{k=1}^n \frac{2\|\Xik\|^2}{q_k} \left(\|\wbks-\wkopt\|^2
  +\|\wkt-\wkopt\|^2\right) .
\label{eqn:uitp-variance}
\end{eqnarray}
Similarly, for $k=1,\ldots,n$, we have
\[
\E_t[\vktp] = \sum_{i=1}^m p_i\left( \vbks
-\frac{1}{p_i}\frac{1}{m}(\Xik)^T\abis
+\frac{1}{p_i}\frac{1}{m}(\Xik)^T\ait \right)
=\frac{1}{m}(\Xck)^T\at.
\]
Therefore $\vktp$ satisfies the assumption in Lemma~\ref{lem:primal-expect-t}. 
Furthermore, we have
\begin{eqnarray}
  \E_t\!\left[\Bigl\|\vktp-\frac{1}{m}(\Xck)^T\at\Bigl\|^2\right]
  &=& \sum_{i=1}^m p_i\Bigl\|\vbks-\frac{1}{p_i}\frac{1}{m}(\Xik)^T\abis
  +\frac{1}{p_i}\frac{1}{m}(\Xik)^T\ait
  - \frac{1}{m}(\Xck)^T\at\Bigr\|^2 \nonumber \\
  &=& \sum_{i=1}^m \frac{1}{p_i}\Bigl\|\frac{1}{m}(\Xik)^T\abis
  -\frac{1}{m}(\Xik)^T\ait \Bigr\|^2
  -\Bigl\|\vbks-\frac{1}{m}(\Xck)^T\at\Bigr\|^2 \nonumber \\
  &\leq& \sum_{i=1}^m \frac{1}{p_i}\Bigl\|\frac{1}{m}(\Xik)^T
  \left(\abis-\ait\right) \Bigr\|^2 \nonumber \\
  &\leq& \sum_{i=1}^m \frac{2\|\Xik\|^2}{p_i m^2}
  \left(\|\abis-\aiopt\|^2+\|\ait-\aiopt\|^2\right).
\label{eqn:vktp-variance}
\end{eqnarray}

Now we consider the two terms containing inner products 
in~\eqref{eqn:primal-dual-sum}.
Using the conditional expectation relation~\eqref{eqn:alpha-cond-mean},
we have
\begin{align}
\E_t\left[-\left\langle \aitp-\ait,\, \Xri(\wt-\wopt)\right\rangle\right]
&= p_i \E_t\left[ - \left\langle \aitld-\ait,\, 
  \Xri(\wt-\wopt)\right\rangle\right] \nonumber \\
&\geq p_i\E_t\left[-\frac{1}{c_i}\|\aitld-\ait\|^2
  -\frac{c_i}{4}\|\Xri(\wt-\wopt)\|^2 \right] \nonumber \\
&= -\frac{p_i}{c_i}\E_t\bigl[\|\aitld-\ait\|^2\bigr]
  -\frac{c_i p_i}{4}\|\Xri(\wt-\wopt)\|^2 \nonumber \\
&= -\frac{1}{c_i} \E_t\bigl[\|\aitp-\ait\|^2\bigr]
  -\frac{c_i p_i}{4}\|\Xri(\wt-\wopt)\|^2,
\label{eqn:alpha-inner-product}
\end{align}
where we used Young's inequality with $c_i$ being an arbitrary positive number,
and the last equality used~\eqref{eqn:alpha-change-norm2}.
We note that for any $n$ vectors $z_1,\ldots,z_n\in\R^{N_i}$, it holds that
\[
\Bigl\|\sum_{k=1}^n z_k\Bigr\|^2 \leq \sum_{k=1}^n \frac{1}{q_k}\|z_k\|^2.
\]
To see this, we let $z_{k,j}$ denote the $j$th component of~$z_k$ and use the
Cauchy-Schwarz inequality:
\begin{eqnarray*}
\biggl\|\sum_{k=1}^n z_k\biggr\|^2 
&=& \sum_{j=1}^{N_i} \left(\sum_{k=1}^n z_{k,j}\right)^2 
   =\sum_{j=1}^{N_i}\left(\sum_{k=1}^n\frac{z_{k,j}}{\sqrt{q_k}}\sqrt{q_k}
      \right)^2 \\
&\leq& \sum_{j=1}^{N_i}\left(\sum_{k=1}^n\biggl(\frac{z_{k,j}}{\sqrt{q_k}}
   \biggr)^2\right)\left(\sum_{k=1}^n\bigl(\sqrt{q_k}\bigr)^2\right) \\
&=& \sum_{j=1}^{N_i}\left(\sum_{k=1}^n\frac{z_{k,j}^2}{q_k}\right)
=\sum_{k=1}^n\frac{1}{q_k}\sum_{j=1}^{N_i}z_{k,j}^2
=\sum_{k=1}^n \frac{1}{q_k}\|z_k\|^2.
\end{eqnarray*}
Applying this inequality to the vector
$\Xri(\wt-\wopt)=\sum_{k=1}^n \Xik(\wkt-\wkopt)$, we get
\[
\|\Xri(\wt-\wopt)\|^2 \leq
\sum_{k=1}^n \frac{1}{q_k} \|\Xik(\wkt-\wkopt)\|^2.
\]
Therefore we can continue the inequality~\eqref{eqn:alpha-inner-product},
for each $i=1,\ldots,m$, as
\begin{eqnarray}
&&\E_t\!\left[-\left\langle \aitp-\ait,\, \Xri(\wt-\wopt)\right\rangle\right]
\nonumber \\
&\geq & -\frac{1}{c_i} \E_t\bigl[\|\aitp-\ait\|^2\bigr]
-\frac{c_i p_i}{4}\sum_{k=1}^n\frac{1}{q_k}\|\Xik(\wkt-\wkopt)\|^2 \nonumber \\
&\geq& -\frac{1}{c_i} \E_t\bigl[\|\aitp-\ait\|^2\bigr]
-\frac{c_i p_i}{4}\sum_{k=1}^n\frac{1}{q_k}\|\Xik\|^2\|\wkt-\wkopt\|^2 .
\label{eqn:dual-inner-prod-bound}
\end{eqnarray}
Using similarly arguments, we can obtain, for each $k=1,\ldots,n$ and 
arbitrary $h_k>0$,
\begin{eqnarray}
&&  \E_t\Bigl[\Bigl\langle \wktp-\wkt,\, \frac{1}{m}(\Xck)^T(\at-\aopt)
  \Bigr\rangle\Bigr] \nonumber \\
&\geq& -\frac{1}{h_k}\E_t\bigl[\|\wktp-\wkt\|^2\bigr]
-\frac{h_k q_k}{4 m^2}\sum_{i=1}^m\frac{1}{p_i}\|\Xik\|^2\|\ait-\aiopt\|^2.
\label{eqn:primal-inner-prod-bound}
\end{eqnarray}

Applying the bounds in~\eqref{eqn:uitp-variance},
\eqref{eqn:vktp-variance}, \eqref{eqn:dual-inner-prod-bound}
and~\eqref{eqn:primal-inner-prod-bound} 
to~\eqref{eqn:primal-dual-sum} and rearranging terms, we have
\begin{eqnarray*}
&&\sum_{i=1}^{m}\frac{1}{m}\left[ 
  \frac{1}{p_i} \left(\frac{1}{2\sigma_i}+\gamma_i\right) - \gamma_i
  + \sum_{k=1}^n\frac{b_k\|\Xik\|^2}{2mp_i} 
  + \sum_{k=1}^n\frac{h_k\|\Xik\|^2}{4mp_i}
  \right] \|\ait-\aiopt\|^2 \\ 
&& + \sum_{k=1}^{n}\left[ 
  \frac{1}{q_k} \left(\frac{1}{2\tau_k}+\lambda\right) - \lambda
  + \sum_{i=1}^m\frac{a_i\|\Xik\|^2}{2mq_k}
  + \sum_{i=1}^m\frac{c_i\|\Xik\|^2}{4mq_k}
  \right] \|\wkt-\wkopt\|^2 \\
&& +\sum_{i=1}^m\sum_{k=1}^n \frac{b_k\|\Xik\|^2}{2m^2 p_i}\|\abis-\aiopt\|^2
 + \sum_{i=1}^m\sum_{k=1}^n \frac{a_i\|\Xik\|^2}{2m q_k}\|\wbks-\wkopt\|^2 \\
&\geq&  \sum_{i=1}^m \frac{1}{m p_i}\left(\frac{1}{2\sigma_i}+\gamma_i\right)
  \E_t[\|\aitp-\aiopt\|^2] 
  + \sum_{k=1}^n \frac{1}{q_k}\left(\frac{1}{2\tau_k}+\lambda\right)
  \E_t[\|\wktp-\wkopt\|^2] \\
&& + \sum_{i=1}^m \frac{1}{m p_i}\left(\frac{1}{2\sigma_i}-\frac{1}{a_i}
  - \frac{1}{c_i} \right) \E_t[\|\aitp-\ait\|^2] \\ 
&& + \sum_{k=1}^n \frac{1}{q_k}\left(\frac{1}{2\tau_k}-\frac{1}{b_k}
  - \frac{1}{h_k} \right) \E_t[\|\wktp-\wkt\|^2]. 
\end{eqnarray*}
The desired result~\eqref{eqn:svrg-pd-one-step}
is obtained by choosing $a_i=c_i=4\sigma_i$ and $b_k=h_k=4\tau_k$.
\end{proof}

Finally, we are ready to prove Theorem~\ref{thm:dscovr-svrg}.
Let $\theta\in(0,1)$ be a parameter to be determined later, 
and let~$\Gamma$ and~$\eta$ be two constants such that
\begin{eqnarray}
\Gamma &\geq& \max_{i,k}\left\{ 
\frac{1}{p_i}\left(1+\frac{3\|\Xik\|^2}{2\theta q_k \lambda\gamma_i}\right),\;
\frac{1}{q_k}\left(1+\frac{3n\|\Xik\|^2}{2\theta m p_i \lambda\gamma_i}\right)
\right\},  \label{eqn:Gamma-def-theta} \\
\eta &=& 1-\frac{1-\theta}{\Gamma}. \label{eqn:eta-def-theta}
\end{eqnarray}
It is easy to check that $\Gamma>1$ and $\eta\in(0,1)$.
By the choices of $\sigma_i$ and $\tau_k$ in~\eqref{eqn:sigma-i-svrg}
and~\eqref{eqn:tau-k-svrg} respectively, we have 
\begin{equation}\label{eqn:sigma-tau-Gamma}
\frac{1}{p_i}\left(1 + \frac{1}{2\sigma_i \gamma_i}\right)
=\frac{1}{q_k}\left(1 + \frac{1}{2\tau_k \lambda}\right)
=\Gamma,
\end{equation}
for all $i=1,\ldots,m$ and $k=1,\ldots,n$.
Comparing the above equality with the definition of $\Gamma$ 
in~\eqref{eqn:Gamma-def-theta}, we have
\[
\frac{3\|\Xik\|^2}{2\theta q_k \lambda\gamma_i}
\leq\frac{1}{2\sigma_i \gamma_i}
\qquad\textrm{and}\qquad
\frac{3n\|\Xik\|^2}{2\theta m p_i \lambda\gamma_i}
\leq\frac{1}{2\tau_k \lambda},
\qquad
\]
which implies
\[
  \frac{3\sigma_i\|\Xik\|^2}{q_k} \leq \theta\lambda 
  \qquad\textrm{and}\qquad
  \frac{3n\tau_k\|\Xik\|^2}{m p_i} \leq \theta\gamma_i,
\]
for all $i=1,\ldots,m$ and $k=1,\ldots,n$. Therefore, we have
\begin{eqnarray}
&&  \sum_{k=1}^n\frac{3\tau_k\|\Xik\|^2}{m p_i}
  = \frac{1}{n}\sum_{k=1}^n\frac{3n\tau_k\|\Xik\|^2}{mp_i}
  \leq \frac{1}{n}\sum_{k=1}^n \theta\gamma_i
  = \theta\gamma_i, 
  \qquad i=1,\ldots,m, \label{eqn:sum-tau-k}\\
&&  \sum_{i=1}^m\frac{3\sigma_i\|\Xik\|^2}{m q_k}
  = \frac{1}{m}\sum_{i=1}^m\frac{3\sigma_i\|\Xik\|^2}{q_k}
  \leq \frac{1}{m}\sum_{i=1}^m \theta\lambda 
  = \theta\lambda, 
  ~\qquad k=1,\ldots,n. \label{eqn:sum-sigma-i}
\end{eqnarray}
Now we consider the inequality~\eqref{eqn:svrg-pd-one-step}, and examine
the ratio between the coefficients of $\|\ait-\aiopt\|^2$ and
$\E_t[\|\aitp-\aiopt\|^2]$.
Using~\eqref{eqn:sum-tau-k} and~\eqref{eqn:sigma-tau-Gamma}, we have
\begin{equation}\label{eqn:dual-coeff-ratio}
\frac{\frac{1}{p_i}\left(\frac{1}{2\sigma_i}+\gamma_i
\right)-\gamma_i+\sum_{k=1}^n\frac{3\tau_k\|\Xik\|^2}{m p_i}}{
  \frac{1}{p_i}\left(\frac{1}{2\sigma_i}+\gamma_i\right)}
\leq 1 - 
\frac{(1-\theta)\gamma_i}{\frac{1}{p_i}\left(\frac{1}{2\sigma_i}+\gamma_i\right)}
=1-\frac{1-\theta}{\Gamma} = \eta.
\end{equation}
Similarly, the ratio between the coefficients of $\|\wkt-\wkopt\|^2$ and
$\E_t[\|\wktp-\wkopt\|^2]$ can be bounded 
using~\eqref{eqn:sum-sigma-i} and~\eqref{eqn:sigma-tau-Gamma}:
\begin{equation}\label{eqn:primal-coeff-ratio}
\frac{\frac{1}{q_k}\left(\frac{1}{2\tau_k}+\lambda
\right)-\lambda+\sum_{i=1}^m\frac{3\sigma_i\|\Xik\|^2}{m q_k}}{
  \frac{1}{q_k}\left(\frac{1}{2\tau_k}+\lambda\right)}
\leq 1 - 
\frac{(1-\theta)\lambda}{\frac{1}{q_k}\left(\frac{1}{2\tau_k}+\lambda\right)}
=1-\frac{1-\theta}{\Gamma} = \eta.
\end{equation}
In addition, the ratio between the coefficients of
$\|\abis-\aiopt\|^2$ and $\E_t[\|\aitp-\aiopt\|^2]$ and that of
$\|\wbks-\wkopt\|^2$ and $\E_t[\|\wktp-\wkopt\|^2]$ can be bounded as
\begin{eqnarray}
\frac{\sum_{k=1}^k\frac{2\tau_k\|\Xik\|^2}{m p_i}}{
  \frac{1}{p_i}\left(\frac{1}{2\sigma_i}+\gamma_i\right)}
\leq
\frac{(2/3)\theta\gamma_i}{\frac{1}{p_i}\left(\frac{1}{2\sigma_i}+\gamma_i\right)}
=\frac{(2/3)\theta}{\Gamma} = \frac{2\theta(1-\eta)}{3(1-\theta)}, 
\label{eqn:dual-stage-ratio} \\
\frac{\sum_{i=1}^m\frac{3\sigma_i\|\Xik\|^2}{m q_k}}{
  \frac{1}{q_k}\left(\frac{1}{2\tau_k}+\lambda\right)}
\leq \frac{(2/3)\theta\lambda}{
  \frac{1}{q_k}\left(\frac{1}{2\tau_k}+\lambda\right)}
=\frac{(2/3)\theta}{\Gamma} = \frac{2\theta(1-\eta)}{3(1-\theta)}. 
\label{eqn:primal-stage-ratio}
\end{eqnarray}
Using~\eqref{eqn:sigma-tau-Gamma} and the four inequalities
\eqref{eqn:dual-coeff-ratio}, \eqref{eqn:primal-coeff-ratio},
\eqref{eqn:dual-stage-ratio} and \eqref{eqn:primal-stage-ratio},
we conclude that the inequality~\eqref{eqn:svrg-pd-one-step} implies
\begin{eqnarray*}
&& \eta \sum_{i=1}^m \frac{\Gamma\gamma_i}{m}\|\ait-\aiopt\|^2
  + \frac{2\theta(1-\eta)}{3(1-\theta)} 
  \sum_{i=1}^m \frac{\Gamma\gamma_i}{m}\|\abis-\aiopt\|^2 \\
&& + \eta \sum_{k=1}^n \Gamma\lambda\|\wkt-\wkopt\|^2 
  + \frac{2\theta(1-\eta)}{3(1-\theta)} 
  \sum_{k=1}^n \Gamma\lambda\|\wbks-\wkopt\|^2 \\
&\geq&
\sum_{i=1}^m \frac{\Gamma\gamma_i}{m}\E_t[\|\aitp-\aiopt\|^2]
  + \sum_{k=1}^n \Gamma\lambda\E_t[\|\wktp-\wkopt\|^2].
\end{eqnarray*}
Using the definite of $\Omega(\cdot)$ in~\eqref{eqn:Omega}, 
the inequality above is equivalent to
\begin{eqnarray}
&&  \eta \,\Omega\bigl(\wt-\wopt,\at-\aopt\bigr)
  + \frac{2\theta(1-\eta)}{3(1-\theta)} 
    \Omega\bigl(\wbs-\wopt,\abs-\aopt\bigr)
\nonumber \\
&\geq& \E_t\!\left[\Omega\bigl(\wtp-\wopt,\atp-\aopt\bigr)\right].
\label{eqn:svrg-pd-one-step-simple}
\end{eqnarray}

To simplify further derivation, we define
\begin{eqnarray*}
\Dt &=&\E\left[\Omega\bigl(\wt-\wopt,\at-\aopt\bigr)\right],\\
\Dbs &=&\E\left[\Omega\bigl(\wbs-\wopt,\abs-\aopt\bigr)\right],
\end{eqnarray*}
where the expectation is taken with respect to all randomness in the $s$th
stage, that is, the random variables 
$\{(j^{(0)},l^{(0)}), (j^{(1)},l^{(1)}),\ldots, (j^{(M-1)},l^{(M-1)})\}$.
Then the inequality~\eqref{eqn:svrg-pd-one-step-simple} implies
\[
  \frac{2\theta(1-\eta)}{3(1-\theta)}\Dbs + \eta\Dt \geq \Dtp.
\]
Dividing both sides of the above inequality by $\eta^{t+1}$ gives
\[
  \frac{2\theta(1-\eta)}{3(1-\theta)}\frac{\Dbs}{\eta^{t+1}}
  + \frac{\Dt}{\eta^t} \geq \frac{\Dtp}{\eta^{t+1}} .
\]
Summing for $t=0,1,,\ldots,M-1$ gives 
\[
  \left(\frac{1}{\eta}+\frac{1}{\eta^2}+\cdots+\frac{1}{\eta^M}\right)
  \frac{2\theta(1-\eta)}{3(1-\theta)}\Dbs
  + \Dini \geq \frac{\DT}{\eta^M} ,
\]
which further leads to
\[
(1-\eta^M)\frac{2\theta}{3(1-\theta)}\Dbs + \eta^M\Dini \geq \Delta^{(M)}.
\]
Now choosing $\theta=1/3$ and using the relation $\Dbs=\Dini$ for each stage,
we obtain
\[
\left(\frac{1}{3}+\frac{2}{3}\eta^M\right)\Dini \geq \Delta^{(M)}.
\]
Therefore if we choose $M$ large enough such that $\eta^M\leq\frac{1}{2}$, then
\[
\Delta^{(M)}\leq\frac{2}{3}\Dini, \quad \textrm{or~equivalently,}\quad
\Dbsp\leq\frac{2}{3}\Dbs.
\]
The condition $\eta^M\leq\frac{1}{2}$ is equivalent to
$M\geq\frac{\log(2)}{\log(1/\eta)}$, which can be guaranteed by
\[
  M\geq \frac{\log(2)}{1-\eta} = \frac{\log(2)}{1-\theta}\Gamma
  =\frac{3\log(2)}{2}\Gamma = \log(\sqrt{8})\Gamma.
\]
To further simplify, it suffices to have $M\geq\log(3)\Gamma$.
Finally, we notice that $\Dbsp\leq(2/3)\Dbs$ implies
$\Dbs\leq\left(2/3\right)^s \Dbini$.
which is the desired result in Theorem~\ref{thm:dscovr-svrg}.

\subsection{Alternative bounds and step sizes}
\label{sec:proof-Fro-step-sizes}
Alternatively, we can let $\Gamma$ to satisfy
\begin{equation}\label{eqn:Gamma-Fro-theta}
\Gamma \geq \max_{i,k}\left\{ 
\frac{1}{p_i}\left(1+\frac{3\|\Xck\|_F^2}{2\theta mq_k\lambda\gamma_i}\right),\;
\frac{1}{q_k}\left(1+\frac{3\|\Xri\|_F^2}{2\theta mp_i\lambda\gamma_i}\right)
\right\},  
\end{equation}
where $\|\cdot\|_F$ denotes the Frobenius norm of a matrix.
Then by choosing $\sigma_i$ and $\tau_k$ that 
satisfy~\eqref{eqn:sigma-tau-Gamma}, we have
\[
  \frac{3\sigma_i\|\Xck\|_F^2}{m q_k} \leq \theta\lambda 
  \qquad\textrm{and}\qquad
  \frac{3\tau_k\|\Xri\|_F^2}{m p_i} \leq \theta\gamma_i,
\]
We can bound the left-hand sides in~\eqref{eqn:sum-tau-k} 
and~\eqref{eqn:sum-sigma-i} using H\"older's inequality, which results in
\begin{eqnarray}
&&  \sum_{k=1}^n\frac{3\tau_k\|\Xik\|^2}{m p_i}
  \leq \sum_{k=1}^n\frac{3\tau_k\|\Xik\|_F^2}{m p_i}
  \leq \frac{3\max_k\{\tau_k\}\|\Xri\|_F^2}{m p_i} \leq \theta\gamma_i, 
  \qquad i=1,\ldots,m, \label{eqn:sum-tau-k-Fro}\\
&&  \sum_{i=1}^m\frac{3\sigma_i\|\Xik\|^2}{m q_k}
  \leq \sum_{i=1}^m\frac{3\sigma_i\|\Xik\|_F^2}{m q_k}
  \leq \frac{3\max_i\{\sigma_i\}\|\Xck\|_F^2}{m q_k} \leq \theta\lambda, 
  \qquad k=1,\ldots,n. \label{eqn:sum-sigma-i-Fro}
\end{eqnarray}
The rest of the proof hold without any change.
Setting $\theta=1/3$ gives the condition on~$\Gamma$ 
in~\eqref{eqn:Gamma-Fro-bound}.

In Theorem~\ref{thm:dscovr-svrg} and the proof above, 
we choose $\Gamma$ as a uniform bound over all combinations of $(i,k)$ 
in order to obtain a uniform convergence rates on all blocks of the
primal and dual variables $\wkt$ and $\ait$, so we have a simple conclusion
as in~\eqref{eqn:svrg-stage-converge}.
In practice, we can use different bounds on different blocks and 
choose step sizes to allow them to converge at different rates. 

For example, we can choose the step sizes $\sigma_i$ and $\tau_k$ such that
\begin{eqnarray*}
\frac{1}{p_i}\left(1+\frac{1}{2\sigma_i\gamma_i}\right)
&=& \max_k \left\{
\frac{1}{p_i}\left(1+\frac{3\|\Xck\|_F^2}{2\theta m q_k \lambda\gamma_i}\right)
\right\}, \qquad i=1,\ldots,m, \\
\frac{1}{q_k}\left(1+\frac{1}{2\tau_k\lambda}\right)
&=& \max_i \left\{
\frac{1}{q_k}\left(1+\frac{3\|\Xri\|_F^2}{2\theta m p_i \lambda\gamma_i}\right)
\right\}, \qquad k=1,\ldots,n.
\end{eqnarray*}
Then the inequalities~\eqref{eqn:sum-tau-k-Fro} and~\eqref{eqn:sum-sigma-i-Fro} 
still hold, and we can still show linear convergence with a similar rate.
In this case, the step sizes are chosen as
\begin{eqnarray*}
\sigma_i &=& \min_k \left\{\frac{\theta m q_k \lambda}{3\|\Xck\|_F^2}\right\},
  \qquad i=1,\ldots,m, \\
\tau_k &=& \min_i \left\{\frac{\theta m p_i \gamma_i}{3\|\Xri\|_F^2}\right\},
  \qquad k=1,\ldots,n.
\end{eqnarray*}
If we choose the probabilities to be proportional to the norms of the data
blocks, i.e., 
\begin{eqnarray*}
  p_i = \frac{\|\Xri\|_F^2}{\|X\|_F^2}, \qquad
  q_k = \frac{\|\Xck\|_F^2}{\|X\|_F^2}, 
\end{eqnarray*}
then we have 
\begin{eqnarray*}
\sigma_i = \frac{\theta m\lambda}{3\|X\|_F^2}, \qquad
\tau_k = \frac{\theta m\gamma_i}{3\|X\|_F^2}.
\end{eqnarray*}
If we further normalize the rows of $X$, and let $R$ be the norm of each row,
then (with $\theta=1/3$)
\begin{eqnarray*}
\sigma_i=\frac{\theta\lambda}{3R^2}\frac{m}{N}=\frac{\lambda}{9R^2}\frac{m}{N},
\qquad
\tau_k=\frac{\theta\gamma_i}{3R^2}\frac{m}{N}=\frac{\gamma_i}{9R^2}\frac{m}{N}.
\end{eqnarray*}
For distributed ERM, we have $\gamma_i=\frac{N}{m}\gamma$, thus
$\tau_k = \frac{\gamma}{9R^2}$ as in~\eqref{eqn:simple-step-size-R}.

\section{Proof of Theorem~\ref{thm:dscovr-svrg-obj}}
\label{sec:proof-svrg-obj}

Consider the following saddle-point problem with doubly separable structure:
\begin{equation}\label{eqn:saddle-block}
\min_{w\in\R^D}~\max_{\alpha\in\R^N}~ \biggl\{
\Lagr(w,\alpha) ~\equiv~ 
\frac{1}{m}\sum_{i=1}^{m}\sum_{k=1}^n\ai^T \Xik \wk
-\frac{1}{m}\sum_{i=1}^{m}f_i^*(\ai)+  
\sum_{k=1}^n g_k(\wk) \biggr\}.
\end{equation}
Under Assumption~\ref{asmp:convexity}, $L$ has a unique saddle point
$(\wopt,\aopt)$.
We define
\begin{eqnarray}
\Ptld_k(\wk) &\equiv& \frac{1}{m}(\aopt)^T \Xck \wk + g_k(\wk)
- \frac{1}{m}(\aopt)^T \Xck \wkopt - g_k(\wkopt), \qquad k=1,\ldots,n, 
\label{eqn:Ptld-def} \\
\Dtld_i(\ai) &\equiv& \frac{1}{m}\left(\ai^T \Xri \wopt - f_i^*(\ai) 
- (\aiopt)^T \Xri\wopt + f_i^*(\aiopt) \right), \qquad i=1,\ldots,m.
\label{eqn:Dtld-def}
\end{eqnarray}
We note that $\wkopt$ is the minimizer of $\Ptld_k$ with $\Ptld_k(\wkopt)=0$
and $\aiopt$ is the maximizer of $\Dtld_i$ with $\Dtld_i(\aiopt)=0$. 
Moreover, by the assumed strong convexity,
\begin{equation}\label{eqn:Ptld-Dtld-quadratic}
\Ptld_k(\wk) ~\geq~ \frac{\lambda}{2}\|\wk - \wkopt\|^2, \qquad
\Dtld_i(\ai) ~\leq~ -\frac{\gamma_i}{2m}\|\ai - \aiopt\|^2. 
\end{equation}
Moreover, 
we have the following lower bound for the duality gap $P(w)-D(\alpha)$:
\begin{equation}\label{eqn:Ptld-Dtld-lower-bound}
\sum_{k=1}^n \Ptld_k(w_k) - \sum_{i=1}^m \Dtld_i(\alpha_i) 
~=~ \Lagr(w,\aopt) - \Lagr(\wopt,\alpha) 
~\leq~ P(w) -  D(\alpha).
\end{equation}
We can also use them to derive an upper bound for the duality gap, 
as in the following lemma.
\begin{lemma}\label{eqn:pd-gap-upper-bound}
Suppose Assumption~\ref{asmp:convexity} holds.
Let $(\wopt,\aopt)$ be the saddle-point of $L(w,\alpha)$ and define
\[
P(w) = \sup_{\alpha} L(w,\alpha), \qquad
D(\alpha) = \inf_{w} L(w,\alpha).
\]
Then we have
\[
P(w)-D(\alpha) ~\leq~ \Lagr(w,\aopt) - \Lagr(\wopt,\alpha) 
+\biggl(\frac{1}{m}\sum_{i=1}^m\frac{\|\Xri\|^2}{2\gamma_i}\biggr)\|w-\wopt\|^2 
+ \frac{\|X\|^2}{2m^2\lambda}\|\alpha-\aopt\|^2.
\]
\end{lemma}

\begin{proof}
By definition, the primal function can be written as
$P(w) = F(w) + g(w)$, where 
\[
  F(w) = \frac{1}{m}\sum_{i=1}^m f_i(\Xri w) 
  = \frac{1}{m}\max_\alpha\biggl\{\alpha^T X w 
    - \sum_{i=1}^m f_i^*(\alpha_i)\biggr\}.
\]
From the optimality conditions satisfied by the saddle point $(\wopt,\aopt)$,
we have
\[
\nabla F(\wopt) = \frac{1}{m}X^T\aopt.
\]
By assumption, $\nabla F(w)$ is Lipschitz continuous with smooth constant
$\frac{1}{m}\sum_{i=1}^m\frac{\|\Xri\|^2}{\gamma_i}$, which implies
\begin{eqnarray*}
F(w) 
&\leq& F(\wopt) + \langle \nabla F(\wopt),\,w-\wopt\rangle
+\biggl(\frac{1}{m}\sum_{i=1}^m\frac{\|\Xri\|^2}{2\gamma_i}\biggr)
\|w-\wopt\|^2\\
&=& \frac{1}{m}\biggl((\aopt)^T X\wopt-\sum_{i=1}^m f_i^*(\aiopt) \biggr)
  + \frac{1}{m} (\aopt)^T X(w-\wopt) 
+\biggl(\frac{1}{m}\sum_{i=1}^m\frac{\|\Xri\|^2}{2\gamma_i}\biggr)
\|w-\wopt\|^2\\
&=& \frac{1}{m}\biggl((\aopt)^T X w-\sum_{i=1}^m f_i^*(\aiopt) \biggr)
+\biggl(\frac{1}{m}\sum_{i=1}^m\frac{\|\Xri\|^2}{2\gamma_i}\biggr)
\|w-\wopt\|^2.
\end{eqnarray*}
Therefore,
\begin{eqnarray*}
P(w) &=& F(w) + g(w) \\
&\leq& \frac{1}{m}(\aopt)^T X w-\frac{1}{m}\sum_{i=1}^m f_i^*(\aiopt) + g(w)
+\biggl(\frac{1}{m}\sum_{i=1}^m\frac{\|\Xri\|^2}{2\gamma_i}\biggr)
\|w-\wopt\|^2\\
&=& \Lagr(w,\aopt) 
+\biggl(\frac{1}{m}\sum_{i=1}^m\frac{\|\Xri\|^2}{2\gamma_i}\biggr)
\|w-\wopt\|^2.
\end{eqnarray*}
Using similar arguments, especially that $\nabla g^*(\alpha)$ has Lipschitz
constant $\frac{\|X\|}{m^2 \lambda}$, we can show that
\[
D(\alpha) ~\geq~ L(\wopt,\alpha)-\frac{\|X\|^2}{2m^2\lambda}\|\alpha-\aopt\|^2.
\]
Combining the last two inequalities gives the desired result.
\end{proof}

The rest of the proof follow similar steps as in the proof of
Theorem~\ref{thm:dscovr-svrg}.
The next two lemmas are variants of Lemmas~\ref{lem:dual-expect-t}
and~\ref{lem:primal-expect-t}.

\begin{lemma}\label{lem:dual-expect-t-obj}
Under the same assumptions and setup in Lemma~\ref{lem:dual-expect-t}, 
we have
\begin{eqnarray}
&&\sum_{i=1}^{m}\left( \frac{1}{p_i} 
\left(\frac{1}{2\sigma_i}+\frac{\gamma_i}{2}\right)
- \frac{\gamma_i}{2}\right) \|\ait-\aiopt\|^2 
-\sum_{i=1}^m\left(\frac{1}{p_i}-1\right)m \Dtld_i(\ait) \nonumber\\
&\geq&  \sum_{i=1}^m \frac{1}{p_i}\left(\frac{1}{2\sigma_i}
+\frac{\gamma_i}{2}\right)  \E_t[\|\aitp\!\!-\aiopt\|^2] 
  + \sum_{i=1}^m \frac{1}{2p_i\sigma_i}\E_t[\|\aitp\!\!-\ait\|^2] 
  -\sum_{i=1}^m\frac{m}{p_i}\E_t\bigl[\Dtld_i(\aitp)\bigr] 
  \nonumber \\
&& + \left\langle \wt - \wopt,\, X^T(\aopt-\at) \right\rangle 
 - \sum_{i=1}^m \frac{1}{p_i} \E_t \!\left[\left\langle\aitp-\ait,\,
  \uitp - \Xri\wopt \right\rangle \right] .
\label{eqn:dual-one-step-sum-obj}
\end{eqnarray}
\end{lemma}

\begin{proof}
We start by taking conditional expectation $\E_t$ on both sides of
the inequality~\eqref{eqn:dual-update-sc}, and would like to
replace every term containing~$\aitld$ with terms that contain only
$\ait$ and $\aitp$.
In addition to the relations in~\eqref{eqn:alpha-cond-mean},
\eqref{eqn:alpha-change-norm2} and~\eqref{eqn:alpha-change-opt},
we also need
\[
  \E_t\bigl[f_i^*(\aitp)\bigr] ~=~ p_i f_i^*(\aitld)
  +(1-p_i)f_i^*(\ait).
\]
After the substitutions and rearranging terms, we have
\begin{eqnarray}
&&\left( \frac{1}{p_i}\left(\frac{1}{2\sigma_i}+\frac{\gamma_i}{2}\right)
- \frac{\gamma_i}{2}\right) \|\ait-\aiopt\|^2 
+ \left(\frac{1}{p_i}-1\right)\bigl(f_i^*(\ait)-f_i^*(\aiopt)\bigr)
\nonumber\\
&\geq&  \frac{1}{p_i}\left(\frac{1}{2\sigma_i}
+\frac{\gamma_i}{2}\right)  \E_t[\|\aitp\!\!-\aiopt\|^2] 
  + \frac{1}{2p_i\sigma_i}\E_t[\|\aitp\!\!-\ait\|^2] 
  + \frac{1}{p_i}\E_t\!
  \bigl[\bigl(f_i^*(\aitp)-f_i^*(\aiopt)\bigr)\bigr]
\nonumber \\
&& \E_t\bigl[\langle \aiopt-\ait,\, \uitp\rangle\bigr]
- \frac{1}{p_i} \E_t \!\left[\left\langle\aitp-\ait,\,
  \uitp \right\rangle \right].
\nonumber
\end{eqnarray}
Next, we use the assumption $\E_t\bigl[\uitp\bigr]=\Xri\wt$
and the definition of $\Dtld_i(\cdot)$ in~\eqref{eqn:Dtld-def} to obtain
\begin{eqnarray}
&&\left( \frac{1}{p_i}\left(\frac{1}{2\sigma_i}+\frac{\gamma_i}{2}\right)
- \frac{\gamma_i}{2}\right) \|\ait-\aiopt\|^2 
- \left(\frac{1}{p_i}-1\right) m \Dtld_i(\ait) 
\nonumber\\
&\geq&  \frac{1}{p_i}\left(\frac{1}{2\sigma_i}
+\frac{\gamma_i}{2}\right)  \E_t[\|\aitp-\aiopt\|^2] 
  + \frac{1}{2p_i\sigma_i}\E_t[\|\aitp-\ait\|^2] 
  - \frac{m}{p_i}\E_t\bigl[\Dtld_i(\aitp)\bigr]
\nonumber \\
&& + \left\langle \aiopt-\ait,\, \Xri\bigl(\wt - \wopt\bigr)\right\rangle 
 - \frac{1}{p_i} \E_t \!\left[\left\langle\aitp-\ait,\,
  \uitp - \Xri\wopt \right\rangle \right] .
\nonumber
\end{eqnarray}
Summing up the above inequality for $i=1,\ldots,m$ gives the desired 
result~\eqref{eqn:dual-one-step-sum-obj}.
\end{proof}

\begin{lemma}\label{lem:primal-expect-t-obj}
Under the same assumptions and setup in Lemma~\ref{lem:primal-expect-t}, 
we have
\begin{eqnarray}
&&\sum_{k=1}^{K}\left( \frac{1}{q_k} \left(\frac{1}{2\tau_k}
+\frac{\lambda}{2}\right) - \frac{\lambda}{2}\right) \|\wkt-\wkopt\|^2 
+ \sum_{k=1}^n\left(\frac{1}{q_k}-1\right)\Ptld_k(\wkt)
\nonumber\\
&\geq&  \sum_{k=1}^n \frac{1}{q_k}\left(\frac{1}{2\tau_k}
+\frac{\lambda}{2}\right) \E_t[\|\wktp\!-\wkopt\|^2] 
  + \sum_{k=1}^n \frac{1}{2q_k\tau_k} \E_t[\|\wktp\!-\wkt\|^2] 
  + \sum_{k=1}^n \frac{1}{q_k}\E_t\!\bigl[\Ptld_k(\wktp)\bigr]
\nonumber \\
&&  + \frac{1}{m}\left\langle  X(\wt-\wopt),\, \at - \aopt  \right\rangle 
 + \sum_{k=1}^n \frac{1}{q_k} \E_t \!\left[\left\langle\wktp-\wkt,\,
\vktp - \frac{1}{m}(\Xck)^T\aopt \right\rangle \right] .
\nonumber
\end{eqnarray}
\end{lemma}

Based on Lemma~\ref{lem:dual-expect-t-obj} and 
Lemma~\ref{lem:primal-expect-t-obj}, we can prove the following proposition.
The proof is very similar to that of Proposition~\ref{prop:svrg-primal-dual},
thus we omit the details here.

\begin{proposition}
\label{prop:svrg-primal-dual-obj}
The $t$-th iteration within the $s$-th stage of
Algorithm~\ref{alg:dscovr-svrg} guarantees
\begin{eqnarray}
&& \sum_{k=1}^n \left(\frac{1}{q_k}-1\right)\Ptld_k(\wkt)
-\sum_{i=1}^m \left(\frac{1}{p_i}-1\right)\Dtld_i(\ait) \nonumber\\
&&\!\!\!\!\! +\sum_{i=1}^m \frac{1}{m}\left[\frac{1}{p_i}\left(
  \frac{1}{2\sigma_i} \!+\!\frac{\gamma_i}{2}\right)-\frac{\gamma_i}{2}
+\sum_{k=1}^n\frac{3\tau_k\|\Xik\|^2}{m p_i} \right] \|\ait\!\!-\aiopt\|^2 
+ \sum_{i=1}^m\sum_{k=1}^n\frac{2\tau_k\|\Xik\|^2}{m^2 p_i}
\|\abis\!\!-\aiopt\|^2  \nonumber \\
&&\!\!\!\!\! +\sum_{k=1}^n \left[\frac{1}{q_k}\left(\frac{1}{2\tau_k}
+\frac{\lambda}{2}\right)-\frac{\lambda}{2} 
+\sum_{i=1}^m\frac{3\sigma_i\|\Xik\|^2}{m q_k} \right]\|\wkt-\wkopt\|^2 
+ \sum_{i=1}^m\sum_{k=1}^n\frac{2\sigma_i\|\Xik\|^2}{m q_k}
\|\wbks-\wkopt\|^2  \nonumber \\
&\geq& \sum_{k=1}^n \frac{1}{q_k}\E_t\bigl[\Ptld_k(\wktp)\bigr]
-\sum_{i=1}^m \frac{1}{p_i}\E_t\bigl[\Dtld_i(\aitp)\bigr] \nonumber\\
&&\!\!\!\!\! +\sum_{i=1}^m \frac{1}{m p_i}\left(\frac{1}{2\sigma_i}
+\frac{\gamma_i}{2}\right)\E_t\bigl[\|\aitp-\aiopt\|^2\bigr]
+\sum_{k=1}^n \frac{1}{q_k}\left(\frac{1}{2\tau_k}
+\frac{\lambda}{2}\right)\E_t\bigl[\|\wktp-\wkopt\|^2\bigr].
\label{eqn:svrg-pd-one-step-obj}
\end{eqnarray}
\end{proposition}

Now we proceed to prove Theorem~\ref{thm:dscovr-svrg-obj}.
Let $\theta\in(0,1)$ be a parameter to be determined later, 
and let~$\Gamma$ and~$\eta$ be two constants such that
\begin{eqnarray}
\Gamma &\geq& \max_{i,k}\left\{ 
\frac{1}{p_i}\left(1+\frac{6\Lambda}{\theta q_k \lambda\gamma_i}\right),\;
\frac{1}{q_k}\left(1+\frac{6n\Lambda}{\theta p_i m\lambda\gamma_i}\right)
\right\},  \label{eqn:Gamma-def-theta-obj} \\
\eta &=& 1-\frac{1-\theta}{\Gamma}. \label{eqn:eta-def-theta-obj}
\end{eqnarray}
It is easy to check that $\Gamma>1$ and $\eta\in(0,1)$.
The choices of $\sigma_i$ and $\tau_k$ in~\eqref{eqn:sigma-i-obj} 
and~\eqref{eqn:tau-k-obj} satisfy
\begin{eqnarray}
\frac{1}{p_i}\left(\frac{1}{2}+\frac{1}{2\sigma_i \gamma_i}\right) 
&=& \frac{\Gamma}{2}, \qquad i=1,\dots,m,  
\label{eqn:sigma-Gamma-obj} \\
\frac{1}{q_k}\left(\frac{1}{2}+\frac{1}{2\tau_k \lambda}\right) 
&=& \frac{\Gamma}{2}, \qquad k=1,\dots,n.
\label{eqn:tau-Gamma-obj} 
\end{eqnarray}
Comparing them with the definition of $\Gamma$ 
in~\eqref{eqn:Gamma-def-theta-obj}, 
and using the assumption $\Lambda\geq\|\Xik\|_F^2\geq\|\Xik\|^2$, we get
\[
\frac{6\|\Xik\|^2}{\theta q_k \lambda\gamma_i} \leq
\frac{6\Lambda}{\theta q_k \lambda\gamma_i} \leq\frac{1}{\sigma_i \gamma_i}
\qquad\textrm{and}\qquad
\frac{6n\|\Xik\|^2}{\theta p_i m\lambda\gamma_i} \leq 
\frac{6n\Lambda}{\theta p_i m\lambda\gamma_i} \leq\frac{1}{\tau_k \lambda},
\]
which implies
\[
  \frac{3\sigma_i\|\Xik\|^2}{q_k} \leq \theta\frac{\lambda}{2} 
  \qquad\textrm{and}\qquad
  \frac{3n\tau_k\|\Xik\|^2}{m p_i} \leq \theta\frac{\gamma_i}{2},
\]
for all $i=1,\ldots,m$ and $k=1,\ldots,n$. Therefore, we have
\begin{eqnarray}
&&  \sum_{k=1}^n\frac{3\tau_k\|\Xik\|^2}{m p_i}
= \frac{1}{n} \sum_{k=1}^n\frac{3n\tau_k\|\Xik\|^2}{m p_i}
\leq \theta\frac{\gamma_i}{2}, 
  \qquad i=1,\ldots,m, \label{eqn:sum-tau-k-obj}\\
&&  \sum_{i=1}^m\frac{3\sigma_i\|\Xik\|^2}{m q_k}
= \frac{1}{m} \sum_{i=1}^m\frac{3\sigma_i\|\Xik\|^2}{q_k}
\leq \theta\frac{\lambda}{2}, 
  \qquad k=1,\ldots,n. \label{eqn:sum-sigma-i-obj}
\end{eqnarray}

Now we consider the inequality~\eqref{eqn:svrg-pd-one-step-obj}, and examine
the ratio between the coefficients of $\|\ait-\aiopt\|^2$ and
$\E_t[\|\aitp-\aiopt\|^2]$.
Using~\eqref{eqn:sum-tau-k-obj} and~\eqref{eqn:sigma-Gamma-obj}, we have
\begin{equation}\label{eqn:dual-coeff-ratio-obj}
  \frac{\frac{1}{p_i}\left(\frac{1}{2\sigma_i}+\frac{\gamma_i}{2}
  \right)-\frac{\gamma_i}{2}+\sum_{k=1}^n\frac{3\tau_k\|\Xik\|^2}{m p_i}}{
    \frac{1}{p_i}\left(\frac{1}{2\sigma_i}+\frac{\gamma_i}{2}\right)}
\leq 1 - 
\frac{(1-\theta)\frac{\gamma_i}{2}}{
  \frac{1}{p_i}\left(\frac{1}{2\sigma_i}+\frac{\gamma_i}{2}\right)}
=1-\frac{1-\theta}{\Gamma} = \eta.
\end{equation}
Similarly, the ratio between the coefficients of $\|\wkt-\wkopt\|^2$ and
$\E_t[\|\wktp-\wkopt\|^2]$ can be bounded 
using~\eqref{eqn:sum-sigma-i-obj} and~\eqref{eqn:tau-Gamma-obj}:
\begin{equation}\label{eqn:primal-coeff-ratio-obj}
  \frac{\frac{1}{q_k}\left(\frac{1}{2\tau_k}+\frac{\lambda}{2}
  \right)-\frac{\lambda}{2}+\sum_{i=1}^m\frac{3\sigma_i\|\Xik\|^2}{m q_k}}{
    \frac{1}{q_k}\left(\frac{1}{2\tau_k}+\frac{\lambda}{2}\right)}
\leq 1 - 
\frac{(1-\theta)\frac{\lambda}{2}}{
  \frac{1}{q_k}\left(\frac{1}{2\tau_k}+\frac{\lambda}{2}\right)}
=1-\frac{1-\theta}{\Gamma} = \eta.
\end{equation}
In addition, the ratio between the coefficients of
$\|\abis-\aiopt\|^2$ and $\E_t[\|\aitp-\aiopt\|^2]$ and that of
$\|\wbks-\wkopt\|^2$ and $\E_t[\|\wktp-\wkopt\|^2]$ can be bounded as
\begin{eqnarray}
\frac{\sum_{k=1}^n\frac{2\tau_k\|\Xik\|^2}{m p_i}}{
  \frac{1}{p_i}\left(\frac{1}{2\sigma_i}+\frac{\gamma_i}{2}\right)}
\leq
\frac{(2/3)\theta\frac{\gamma_i}{2}}{\frac{1}{p_i}
\left(\frac{1}{2\sigma_i}+\frac{\gamma_i}{2}\right)}
=\frac{(2/3)\theta}{\Gamma} = \frac{2\theta(1-\eta)}{3(1-\theta)}, 
\label{eqn:dual-stage-ratio-obj} \\
\frac{\sum_{i=1}^m\frac{2\sigma_i\|\Xik\|^2}{m q_k}}{
  \frac{1}{q_k}\left(\frac{1}{2\tau_k}+\frac{\lambda}{2}\right)}
  \leq \frac{(2/3)\theta\frac{\lambda}{2}}{
    \frac{1}{q_k}\left(\frac{1}{2\tau_k}+\frac{\lambda}{2}\right)}
=\frac{(2/3)\theta}{\Gamma} = \frac{2\theta(1-\eta)}{3(1-\theta)}. 
\label{eqn:primal-stage-ratio-obj}
\end{eqnarray}
Also, the ratios between the coefficients 
of $\Ptld_k(\wkt)$ and $\E_t\bigl[\Ptld_k(\wktp)\bigr]$ is $1-q_k$,
and that of $\Dtld_k(\ait)$ and $\E_t\bigl[\Dtld_i(\aitp)\bigr]$ is $1-p_i$.
From the definition of~$\Gamma$ and~$\eta$ in~\eqref{eqn:Gamma-def-theta-obj}
and~\eqref{eqn:eta-def-theta-obj}, we have
\begin{equation}\label{eqn:p-q-eta-obj}
  1-p_i \leq \eta \quad\textrm{for}\quad i=1,\ldots,m,
  \quad\textrm{and}\quad
  1-q_k \leq \eta \quad\textrm{for}\quad k=1,\ldots,n.
\end{equation}
Using the relations in~\eqref{eqn:sigma-Gamma-obj} 
and~\eqref{eqn:tau-Gamma-obj} and the inequalities
\eqref{eqn:dual-coeff-ratio-obj}, \eqref{eqn:primal-coeff-ratio-obj},
\eqref{eqn:dual-stage-ratio-obj}, \eqref{eqn:primal-stage-ratio-obj}
and~\eqref{eqn:p-q-eta-obj},
we conclude that the inequality~\eqref{eqn:svrg-pd-one-step-obj} implies
\begin{eqnarray*}
&& \eta \left(\sum_{k=1}^n \frac{1}{q_k}\Ptld_k(\wkt)
-\sum_{i=1}^m \frac{1}{p_i}\Dtld_i(\ait) \right) 
 + \eta \left(\sum_{i=1}^m \frac{\Gamma\gamma_i}{2m}\|\ait-\aiopt\|^2
   + \sum_{k=1}^n \frac{\Gamma\lambda}{2}\|\wkt-\wkopt\|^2 \right) \\
&&  + \frac{2\theta(1-\eta)}{3(1-\theta)} \left(
  \sum_{i=1}^m \frac{\Gamma\gamma_i}{2m}\|\abis-\aiopt\|^2 
  \sum_{k=1}^n \frac{\Gamma\lambda}{2}\|\wbks-\wkopt\|^2 \right)\\
&\geq& \sum_{k=1}^n \frac{1}{q_k}\E_t\bigl[\Ptld_k(\wktp)\bigr]
-\sum_{i=1}^m \frac{1}{p_i}\E_t\bigl[\Dtld_i(\aitp)\bigr] \nonumber\\
&&+\sum_{i=1}^m \frac{\Gamma\gamma_i}{2m}\E_t[\|\aitp-\aiopt\|^2]
+ \sum_{k=1}^n \frac{\Gamma\lambda}{2}\E_t[\|\wktp-\wkopt\|^2],
\end{eqnarray*}
which is equivalent to
\begin{eqnarray}
&& \eta \left(\sum_{k=1}^n \frac{1}{q_k}\Ptld_k(\wkt)
-\sum_{i=1}^m \frac{1}{p_i}\Dtld_i(\ait) 
+ \frac{\Gamma\lambda}{2}\|\wt-\wopt\|^2 
+ \frac{1}{m}\sum_{i=1}^m\frac{\Gamma\gamma_i}{2}\|\ait-\aiopt\|^2\right)
\nonumber\\
&&  + \frac{2\theta(1-\eta)}{3(1-\theta)} 
\left(\frac{\Gamma\lambda}{2}\|\wbs-\wopt\|^2
  +\frac{1}{m}\sum_{i=1}^m\frac{\Gamma\gamma_i}{2}\|\abis-\aiopt\|^2\right) 
\label{eqn:svrg-pd-one-step-simple-obj} \\
&\geq& 
  \E_t\Biggl[\sum_{k=1}^n \frac{1}{q_k}\Ptld_k(\wktp)
-\sum_{i=1}^m \frac{1}{p_i}\Dtld_i(\aitp) 
+\frac{\Gamma\lambda}{2}\|\wtp\!-\wopt\|^2
+\frac{1}{m}\sum_{i=1}^m\frac{\Gamma\gamma_i}{2}\|\aitp\!-\aiopt\|^2\Biggr].
\nonumber 
\end{eqnarray}

To simplify further derivation, we define
\begin{eqnarray*}
\Dt &=& \sum_{k=1}^n \frac{1}{q_k}\Ptld_k(\wkt)
-\sum_{i=1}^m \frac{1}{p_i}\Dtld_i(\ait) 
+ \frac{\Gamma\lambda}{2}\|\wt-\wopt\|^2 
+ \frac{1}{m}\sum_{i=1}^m\frac{\Gamma\gamma_i}{2}\|\ait-\aiopt\|^2, \\
\Dbs &=& \sum_{k=1}^n \frac{1}{q_k}\Ptld_k(\wbks)
-\sum_{i=1}^m \frac{1}{p_i}\Dtld_i(\abis) 
+ \frac{\Gamma\lambda}{2}\|\wbs-\wopt\|^2 
+ \frac{1}{m}\sum_{i=1}^m\frac{\Gamma\gamma_i}{2}\|\abis-\aiopt\|^2.
\end{eqnarray*}
Using the facts that $\Ptld_k(\wbks)\geq 0$ and $-\Dtld_i(\abis)\geq 0$,
the inequality~\eqref{eqn:svrg-pd-one-step-simple-obj} implies
\[
\frac{2\theta(1-\eta)}{3(1-\theta)}\Dbs + \eta\E\bigl[\Dt\bigr]
~\geq~ \E\bigl[\Dtp\bigr],
\]
where the expectation is taken with respect to all randomness in the $s$-th
stage, that is, the random variables 
$\{(j^{(0)},l^{(0)}), (j^{(1)},l^{(1)}),\ldots, (j^{(M-1)},l^{(M-1)})\}$.
Next we choose $\theta=1/3$ and follow the same arguments as in the proof 
for Theorem~\ref{thm:dscovr-svrg} to obtain
$\E\bigl[\DM\bigr]~\leq~\frac{2}{3}\Dini$,
provided $M\geq\log(3)\Gamma$. This further implies
\begin{equation}\label{eqn:stage-converge-obj}
\E\bigl[\Dbs\bigr]~\leq~\left(\frac{2}{3}\right)^s \Dbini.
\end{equation}

From the definition of $\Gamma$ in~\eqref{eqn:Gamma-def-theta-obj}, we have
$\frac{1}{q_k}<\Gamma$ for $k=1,\ldots,n$ and
$\frac{1}{p_i}<\Gamma$ for $i=1,\ldots,m$. Therefore,
\begin{eqnarray}
\Dbini
&\leq& \Gamma\left(\sum_{k=1}^n \Ptld_k(\wbkini) -\sum_{i=1}^m \Dtld_i(\abiini) 
+ \frac{\lambda}{2}\|\wbini-\wopt\|^2 
+ \frac{1}{m}\sum_{i=1}^m\frac{\gamma_i}{2}\|\abiini-\aiopt\|^2\right) 
\nonumber \\
&\leq& 2\Gamma \left(\sum_{k=1}^n\Ptld_k(\wbkini)
-\sum_{i=1}^m\Dtld_i(\abiini)\right) \nonumber \\
&\leq& 2\Gamma \left(P(\wbini)-D(\abini)\right) ,
\label{eqn:Delta-0-leq-gap}
\end{eqnarray}
where the second inequality used~\eqref{eqn:Ptld-Dtld-quadratic}
and the last inequality used~\eqref{eqn:Ptld-Dtld-lower-bound}.
On the other hand, we can also lower bound $\Dbs$ using $P(\wbs)-D(\abs)$.
To this end, we notice that with $\theta=1/3$,
\begin{eqnarray*}
\Gamma &\geq& \max_{i,k}\left\{ 
\frac{1}{p_i}\left(1+\frac{18\Lambda}{q_k \lambda\gamma_i}\right),\;
\frac{1}{q_k}\left(1+\frac{18n\Lambda}{p_i m\lambda\gamma_i}\right)\right\}
~\geq~ \max_{i,k}\left\{ 
\frac{18\Lambda}{p_i q_k \lambda\gamma_i},\;
\frac{18n\Lambda}{p_i q_k m\lambda\gamma_i}\right\}.
\end{eqnarray*}
Noticing that $\max_k\{1/q_k\}\geq n$ and $n\Lambda\geq\|\Xri\|_F^2$ for all
$i=1,\ldots,m$, we have 
\[
\Gamma \geq \max_{i,k}\left\{ \frac{18\Lambda}{q_k\lambda\gamma_i} \right\}
\geq \max_i\left\{ \frac{18n\Lambda}{\lambda\gamma_i} \right\}
\geq\frac{18}{m\lambda}\sum_{i=1}^m \frac{n\Lambda}{\gamma_i}
\geq\frac{18}{m\lambda}\sum_{i=1}^m \frac{\|\Xri\|_F^2}{\gamma_i}
\geq\frac{18}{m\lambda}\sum_{i=1}^m \frac{\|\Xri\|^2}{\gamma_i}.
\]
Moreover, since 
$\Gamma\geq\max_k\left\{\frac{18\Lambda}{p_i q_k \lambda\gamma_i}\right\}
\geq\frac{18n\Lambda}{p_i \lambda\gamma_i}$ for all~$i$
and $mn\Lambda\geq\|X\|_F^2$, we have
\begin{eqnarray*}
\frac{1}{m}\sum_{i=1}^m \Gamma\gamma_i\|\abis-\aiopt\|^2 
&\geq&
\frac{1}{m}\sum_{i=1}^m \frac{18n\Lambda}{p_i\lambda\gamma_i}
\gamma_i\|\abis-\aiopt\|^2
=\frac{18mn\Lambda}{m^2\lambda} \sum_{i=1}^m \frac{\|\abis-\aiopt\|^2}{p_i}\\
&\geq&
\frac{18\|X\|_F^2}{m^2\lambda}\biggl(\sum_{i=1}^m\|\abis-\aiopt\|\biggr)^2\\
&\geq&
\frac{18\|X\|^2}{m^2\lambda}\sum_{i=1}^m\|\abis-\aiopt\|^2
=\frac{18\|X\|^2}{m^2\lambda}\|\abs-\aopt\|^2.
\end{eqnarray*}
Therefore, from the definition of $\Dbs$,
\begin{eqnarray}
\Dbs &=& \sum_{k=1}^n \frac{1}{q_k}\Ptld_k(\wbks)
-\sum_{i=1}^m \frac{1}{p_i}\Dtld_i(\abis) 
+ \frac{\Gamma\lambda}{2}\|\wbs-\wopt\|^2 
+ \frac{1}{m}\sum_{i=1}^m\frac{\Gamma\gamma_i}{2}\|\abis-\aiopt\|^2 \nonumber\\
&\geq& \sum_{k=1}^n \Ptld_k(\wbks)-\sum_{i=1}^m \Dtld_i(\abis)
+\biggl(\frac{18}{m}\sum_{i=1}^m\frac{\|\Xri\|^2}{\gamma_i}\biggr)
\|\wbs-\wopt\|^2
+\frac{18\|X\|^2}{m^2\lambda}\|\abs-\aopt\|^2  \nonumber \\ 
&=& \Lagr(\wbs,\aopt) - \Lagr(\wopt,\abs)
+\biggl(\frac{18}{m}\sum_{i=1}^m\frac{\|\Xri\|^2}{\gamma_i}\biggr)
\|\wbs-\wopt\|^2
+\frac{18\|X\|^2}{m^2\lambda}\|\abs-\aopt\|^2  \nonumber \\
&\geq& P(\wbs)-D(\abs),
\label{eqn:Delta-s-geq-gap}
\end{eqnarray}
where the last inequality is due to Lemma~\ref{eqn:pd-gap-upper-bound}.
Combining~\eqref{eqn:stage-converge-obj}, \eqref{eqn:Delta-0-leq-gap}
and~\eqref{eqn:Delta-s-geq-gap} gives the desired result:
\[
  \E\left[ P(\wbs)-D(\abs) \right] ~\leq~ 
  \left(\frac{2}{3}\right)^s 
  2\Gamma \left(P(\wbini)-D(\abini)\right) .
\]

\section{Proof of Theorem~\ref{thm:dscovr-saga}}
\label{sec:proof-saga}

To facilitate the analysis of DSCOVR-SAGA in Algorithm~\ref{alg:dscovr-saga},
we define two sequences of matrices recursively.
The first is $\{\Wt\}_{t\geq 0}$, where each $\Wt\in\R^{m\times d}$.
They are partitioned into $m\times n$ blocks, and we denote each block as
$\Wikt\in\R^{1\times d_k}$.
The recursive updates for $\Wt$ are as follows:
\begin{eqnarray}
\Wini &=& \ones_m \otimes \bigl(\wini\bigr)^T, \nonumber \\
\Wiktp &=& \left\{ \begin{array}{ll}
	\bigl(\wlt\bigr)^T & \textrm{if}~i=j~\textrm{and}~k=l,\\[0.3em]
\Wikt & \textrm{otherwise}, \end{array} \right.
\qquad t=0,1,2,\ldots,
\label{eqn:Wt-recursion}
\end{eqnarray}
where $\ones_m$ denotes the vector of all ones in $\R^m$.
and $\otimes$ denotes the Kronecker product of two matrices.
The second sequence is $\{\At\}_{t\geq 0}$, where each $\At\in\R^{N\times n}$.
They are partitioned into $m\times n$ blocks, and we denote each block as
$\Aikt\in\R^{N_i\times 1}$.
The recursive updates for $\At$ are as follows:
\begin{eqnarray}
\Aini &=& \aini \otimes \ones_n^T, \nonumber \\
\Aiktp &=& \left\{ \begin{array}{ll}
	\ajt & \textrm{if}~i=j~\textrm{and}~k=l,\\[0.3em]
\Aikt & \textrm{otherwise}, \end{array} \right.
\qquad t=0,1,2,\ldots.
\label{eqn:At-recursion}
\end{eqnarray}
The matrices $\Wt$ and $\At$ consist of most recent values of the primal and
dual block coordinates, updated at different times, up to time~$t$.

Notice that in Algorithm~\ref{alg:dscovr-saga}, 
the matrices $\Ut\in\R^{N\times n}$ follow the same partitioning as the
matrices $\At$, and the matrices $\Vt\in\R^{m\times d}$ follow the same
partitioning as the matrices $\Wt$.
According to the updates of $\Ut$, $\Vt$, $\ubt$ and $\vbt$
in Algorithm~\ref{alg:dscovr-saga}, we have for each $t\geq 0$,
\begin{eqnarray}
\Uikt &=& \Xik\bigl(\Wikt\bigr)^T, 
	\quad i=1,\ldots,m,\quad k=1,\ldots,n, \nonumber \\
\Vikt &=& \frac{1}{m}\bigl(\Aikt\bigr)^T \Xik, 
	\quad i=1,\ldots,m,\quad k=1,\ldots,n. \nonumber 
\end{eqnarray}

\begin{proposition}
\label{prop:saga-primal-dual}
Suppose Assumption~\ref{asmp:convexity} holds. 
The $t$-th iteration of~Algorithm~\ref{alg:dscovr-saga} guarantees
\begin{eqnarray}
&& \sum_{i=1}^m \frac{1}{m}\left[\frac{1}{p_i}\left(\frac{1}{2\sigma_i}
+\gamma_i\right)-\gamma_i+\sum_{k=1}^n\frac{3\tau_k\|\Xik\|^2}{m p_i} \right]
\|\ait-\aiopt\|^2  
+ \sum_{i=1}^m\sum_{k=1}^n\frac{2\tau_k\|\Xik\|^2}{m^2 p_i}
\|\Aikt-\aiopt\|^2  \nonumber\\
&&\!\!\!\!\!\! +\sum_{k=1}^n \left[\frac{1}{q_k}\left(\frac{1}{2\tau_k}
+\lambda\right)-\lambda+\sum_{i=1}^m\frac{3\sigma_i\|\Xik\|^2}{m q_k} 
  \right] \|\wkt-\wkopt\|^2 
+ \sum_{i=1}^m\sum_{k=1}^n\frac{2\sigma_i\|\Xik\|^2}{m q_k}
\|\bigl(\Wikt\bigr)^T-\wkopt\|^2  \nonumber \\
&\geq& \sum_{i=1}^m \frac{1}{m p_i}\left(\frac{1}{2\sigma_i}+\gamma_i\right)
\E_t\bigl[\|\aitp-\aiopt\|^2\bigr]
+\sum_{k=1}^n \frac{1}{q_k}\left(\frac{1}{2\tau_k}+\lambda\right)
\E_t\bigl[\|\wktp-\wkopt\|^2\bigr]
\label{eqn:saga-one-step-prop}
\end{eqnarray}
\end{proposition}

\begin{proof}
The main differences between Algorithm~\ref{alg:dscovr-svrg} and
Algorithm~\ref{alg:dscovr-saga} are the definitions of $\ujtp$ and $\vltp$.
We start with the inequality~\eqref{eqn:primal-dual-sum} and revisit the 
bounds for the following two quantities:
\[
  \E_t\bigl[\|\uitp-\Xri \wt\|^2\bigr] 
  \qquad\textrm{and}\qquad
  \E_t\!\left[\Bigl\|\vktp-\frac{1}{m}(\Xck)^T\at\Bigr\|^2\right].
\]
For Algorithm~\ref{alg:dscovr-saga}, we have
\begin{eqnarray*}
\uitp &=& \ubit-\frac{1}{q_l}\Uilt+\frac{1}{q_l}\Xil\wlt, 
\qquad i=1,\ldots,m,\\
\vktp &=& \vbkt-\frac{1}{p_j}(\Vjkt)^T +\frac{1}{p_j}\frac{1}{m}(\Xjk)^T\ajt,
\qquad k=1,\ldots,n.
\end{eqnarray*}
We can apply the reasoning in~\eqref{eqn:ujtp-expect} 
and~\eqref{eqn:vltp-expect} to every block coordinate and obtain
\begin{eqnarray*}
  \E_t\bigl[\uitp\bigr] &=& \Xri\wt, \qquad i=1,\ldots,m, \\
  \E_t\bigl[\vktp\bigr] &=& \frac{1}{m}(\Xck)^T\at , \qquad k=1,\ldots,n.
\end{eqnarray*}
Therefore they satisfy the assumptions in
Lemma~\ref{lem:dual-expect-t} and Lemma~\ref{lem:primal-expect-t}, respectively.
Moreover, following similar arguments as in~\eqref{eqn:uitp-variance} 
and~\eqref{eqn:vktp-variance}, we have
\begin{eqnarray*}
\E_t\bigl[\|\uitp-\Xri \wt\|^2\bigr] 
&\leq& \sum_{k=1}^n\frac{2\|\Xik\|^2}{q_k}\left(
\Bigl\|\bigl(\Wikt\bigr)^T-\wkopt\Bigr\|^2 + \|\wkt-\wkopt\|^2\right), \\
\E_t\!\left[\Bigl\|\vktp-\frac{1}{m}(\Xck)^T\at\Bigr\|^2\right]
&\leq& \sum_{i=1}^m\frac{2\|\Xik\|^2}{m^2 p_i}\left(
\Bigl\|\bigl(\Aikt\bigr)^T-\aiopt\Bigr\|^2 + \|\ait-\aiopt\|^2\right). 
\end{eqnarray*}
The rest of the proof are the same as in the proof of
Proposition~\ref{prop:svrg-primal-dual}.
\end{proof}

Now we are ready to prove Theorem~\ref{thm:dscovr-saga}.
By the definition of $\Wt$ in~\eqref{eqn:Wt-recursion} and $\At$ 
in~\eqref{eqn:At-recursion}, we have
\begin{eqnarray}
\E_t\left[\bigl\|\bigl(\Wiktp)^T-\wkopt\bigr\|^2\right]
&=& p_i q_k \bigl\|\wkt-\wkopt\bigr\|^2 
	+ (1-p_i q_k)\bigl\|\bigl(\Wikt\bigr)^T-\wkopt\bigr\|^2, 
\label{eqn:expect-Wiktp}\\
\E_t\left[\bigl\|\Aiktp-\aiopt\bigr\|^2\right]
&=& p_i q_k \bigl\|\ait-\aiopt\bigr\|^2 
	+ (1-p_i q_k)\bigl\|\Aikt-\aiopt\bigr\|^2 .
\label{eqn:expect-Aiktp}
\end{eqnarray}
For all $i=1,\ldots,m$ and $k=1,\ldots,n$, let
\begin{equation}\label{eqn:xi-zeta}
\xi_{ik} = \frac{3\sigma_i\|\Xik\|^2}{m p_i q_k^2}
\qquad\textrm{and}\qquad
\zeta_{ik} = \frac{3\tau_k\|\Xik\|^2}{m^2 p_i^2 q_k}.
\end{equation}
We multiply~\eqref{eqn:expect-Wiktp} by $\xi_{ik}$ and~\eqref{eqn:expect-Aiktp}
by $\zeta_{ik}$ and add them to~\eqref{eqn:saga-one-step-prop} to obtain
\begin{eqnarray}
&& \sum_{i=1}^m \frac{1}{m}\left[\frac{1}{p_i}\left(\frac{1}{2\sigma_i}
+\gamma_i\right)-\gamma_i+\sum_{k=1}^n\frac{6\tau_k\|\Xik\|^2}{m p_i} \right]
\|\ait-\aiopt\|^2 \nonumber\\  
&&\!\!\! +\sum_{k=1}^n \left[\frac{1}{q_k}\left(\frac{1}{2\tau_k}+\lambda
\right)-\lambda+\sum_{i=1}^m\frac{6\sigma_i\|\Xik\|^2}{m q_k} \right]
\|\wkt-\wkopt\|^2 \nonumber \\
&& + \sum_{i=1}^m\sum_{k=1}^n \left(1-\frac{1}{3}p_i q_k\right) 
	\zeta_{ik} \bigl\|\Aikt-\aiopt\bigr\|^2  
   + \sum_{i=1}^m\sum_{k=1}^n\left(1-\frac{1}{3}p_i q_k\right)
	\xi_{ik}\bigl\|\bigl(\Wikt\bigr)^T-\wkopt\bigr\|^2   \nonumber\\
&\geq& \sum_{i=1}^m \frac{1}{m p_i}\left(\frac{1}{2\sigma_i}+\gamma_i\right)
\E_t\bigl[\|\aitp-\aiopt\|^2\bigr]
+\sum_{k=1}^n \frac{1}{q_k}\left(\frac{1}{2\tau_k}+\lambda\right)
\E_t\bigl[\|\wktp-\wkopt\|^2\bigr] \nonumber \\
&& + \sum_{i=1}^m\sum_{k=1}^n 
\zeta_{ik}\E_t\bigl[\bigl\|\Aiktp-\aiopt\bigr\|^2\bigr]
+ \sum_{i=1}^m\sum_{k=1}^n 
\xi_{ik}\E_t\bigl[\bigl\|\bigl(\Wiktp\bigr)^T-\wkopt\bigr\|^2\bigr].
\label{eqn:saga-pd-one-step}
\end{eqnarray}
Let $\theta\in(0,1)$ be a parameter to be determined later, 
and $\Gamma$ be a constant such that
\begin{eqnarray}
\Gamma &\geq& \max_{i,k}\left\{ 
\frac{1}{p_i}\left(1+\frac{3\|\Xik\|^2}{2\theta q_k \lambda\gamma_i}\right),\;
\frac{1}{q_k}\left(1+\frac{3n\|\Xik\|^2}{2\theta p_i m\lambda\gamma_i}\right),\;
\frac{1}{p_i q_k} \right\}.  \label{eqn:Gamma-def-theta-saga} 
\end{eqnarray}
The choices of $\sigma_i$ in~\eqref{eqn:sigma-i-saga} and
$\tau_k$ in~\eqref{eqn:tau-k-saga} satisfy
\begin{equation}\label{eqn:sigma-tau-Gamma-saga}
\frac{1}{p_i}\left(1 + \frac{1}{2\sigma_i \gamma_i}\right)
=\frac{1}{q_k}\left(1 + \frac{1}{2\tau_k \lambda}\right)
=\Gamma.
\end{equation}
Comparing the above equality with the definition of $\Gamma$ 
in~\eqref{eqn:Gamma-def-theta-saga}, we have
\[
\frac{3\|\Xik\|^2}{2\theta q_k \lambda\gamma_i}
\leq\frac{1}{2\sigma_i \gamma_i}
\qquad\textrm{and}\qquad
\frac{3n\|\Xik\|^2}{2\theta p_i m\lambda\gamma_i}
\leq\frac{1}{2\tau_k \lambda},
\qquad
\]
which implies that
\begin{equation}\label{eqn:theta-lambda-theta-gamma}
  \frac{6\sigma_i\|\Xik\|^2}{q_k} \leq 2\theta\lambda 
  \qquad\textrm{and}\qquad
  \frac{6n\tau_k\|\Xik\|^2}{m p_i} \leq 2\theta\gamma_i
\end{equation}
hold for all $i=1,\ldots,m$ and $k=1,\ldots,n$. Therefore, we have
\begin{eqnarray}
&&  \sum_{k=1}^n\frac{6\tau_k\|\Xik\|^2}{m p_i}
 =\frac{1}{n}\sum_{k=1}^n\frac{6n\tau_k\|\Xik\|^2}{m p_i}
 \leq 2\theta\gamma_i, 
 \qquad i=1,\ldots,m, \label{eqn:sum-tau-k-saga}\\
&&  \sum_{i=1}^m\frac{6\sigma_i\|\Xik\|^2}{m q_k}
 = \frac{1}{m}\sum_{i=1}^m \frac{6\sigma_i\|\Xik\|^2}{q_k} 
 \leq 2\theta\lambda, 
 \qquad k=1,\ldots,n. \label{eqn:sum-sigma-i-saga}
\end{eqnarray}
Now we consider the inequality~\eqref{eqn:saga-pd-one-step}, and examine
the ratio between the coefficients of $\|\ait-\aiopt\|^2$ and
$\E_t[\|\aitp-\aiopt\|^2]$.
Using~\eqref{eqn:sum-tau-k-saga} and~\eqref{eqn:sigma-tau-Gamma-saga}, we have
\begin{equation}\label{eqn:dual-coeff-ratio-saga}
\frac{\frac{1}{p_i}\left(\frac{1}{2\sigma_i}+\gamma_i
\right)-\gamma_i+\sum_{k=1}^n\frac{6\tau_k\|\Xik\|^2}{m p_i}}{
  \frac{1}{p_i}\left(\frac{1}{2\sigma_i}+\gamma_i\right)}
\leq 1 - 
\frac{(1-2\theta)\gamma_i}{\frac{1}{p_i}\left(\frac{1}{2\sigma_i}
	+\gamma_i\right)}
=1-\frac{1-2\theta}{\Gamma}.
\end{equation}
Similarly, the ratio between the coefficients of $\|\wkt-\wkopt\|^2$ and
$\E_t[\|\wktp-\wkopt\|^2]$ can be bounded 
using~\eqref{eqn:sum-sigma-i-saga} and~\eqref{eqn:sigma-tau-Gamma-saga}:
\begin{equation}\label{eqn:primal-coeff-ratio-saga}
\frac{\frac{1}{q_k}\left(\frac{1}{2\tau_k}+\lambda
\right)-\lambda+\sum_{i=1}^m\frac{6\sigma_i\|\Xik\|^2}{m q_k}}{
  \frac{1}{q_k}\left(\frac{1}{2\tau_k}+\lambda\right)}
\leq 1 - 
\frac{(1-2\theta)\lambda}{\frac{1}{q_k}\left(\frac{1}{2\tau_k}+\lambda\right)}
=1-\frac{1-2\theta}{\Gamma}.
\end{equation}
We notice that in~\eqref{eqn:saga-pd-one-step},
the ratios between the coefficients of
$\zeta_{ik} \bigl\|\Aikt-\aiopt\bigr\|^2$ and
$\zeta_{ik}\E_t\!\left[\bigl\|\Aiktp-\aiopt\bigr\|^2\right]$, as well as
and that of
$\xi_{ik}\bigl\|\bigl(\Wikt\bigr)^T-\wkopt\bigr\|^2$ and
$\xi_{ik}\E_t\bigl[\bigl\|\bigl(\Wiktp\bigr)^T-\wkopt\bigr\|^2\bigr]$, 
are all $1-\frac{1}{3}p_i q_k$.
By definition of $\Gamma$ in~\eqref{eqn:Gamma-def-theta-saga}, we have
\begin{equation}\label{eqn:1-third-piqk}
  1-\frac{1}{3}p_i q_k ~\leq~ 1-\frac{1}{3\Gamma},
  \qquad i=1,\ldots,m, \quad k=1,\ldots,n.
\end{equation}
We choose $\theta=1/3$ so that the ratios
in~\eqref{eqn:dual-coeff-ratio-saga} and \eqref{eqn:primal-coeff-ratio-saga}
have the same bound $1-\frac{1}{3\Gamma}$.
Therefore, it follows from inequality~\eqref{eqn:saga-pd-one-step} that
\begin{eqnarray}
&& \sum_{i=1}^m \frac{\Gamma\gamma_i}{m}\E_t\bigl[\|\aitp-\aiopt\|^2\bigr]
+\sum_{k=1}^n \Gamma\lambda \E_t\bigl[\|\wktp-\wkopt\|^2\bigr] \nonumber \\
&& + \sum_{i=1}^m\sum_{k=1}^n 
\zeta_{ik}\E_t\bigl[\bigl\|\Aiktp-\aiopt\bigr\|^2\bigr]
+ \sum_{i=1}^m\sum_{k=1}^n 
\xi_{ik}\E_t\bigl[\bigl\|\bigl(\Wiktp\bigr)^T-\wkopt\bigr\|^2\bigr].
\nonumber \\ 
&\leq & \left(1-\frac{1}{3\Gamma}\right) \Biggl(
  \sum_{i=1}^m \frac{\Gamma\gamma_i}{m}\|\ait-\aiopt\|^2 
+\sum_{k=1}^n \Gamma\lambda\|\wkt-\wkopt\|^2 \nonumber \\
&& \qquad\qquad
+ \sum_{i=1}^m\sum_{k=1}^n \zeta_{ik} \bigl\|\Aikt-\aiopt\bigr\|^2  
+ \sum_{i=1}^m\sum_{k=1}^n \xi_{ik}\bigl\|\bigl(\Wikt\bigr)^T-\wkopt\bigr\|^2
\Biggr). 
\label{eqn:saga-one-step-reduction}
\end{eqnarray}
Let's define
\[
\Dt=\lambda\|\wt\!-\wopt\|^2+\frac{1}{m}\sum_{i=1}^m\gamma_i\|\ait\!-\aiopt\|^2
+ \sum_{i=1}^m\sum_{k=1}^n \frac{\zeta_{ik}}{\Gamma}
\bigl\|\Aikt\!-\aiopt\bigr\|^2
+ \sum_{i=1}^m\sum_{k=1}^n \frac{\xi_{ik}}{\Gamma}
\bigl\|\bigl(\Wikt\bigr)^T\!\!-\wkopt\bigr\|^2.
\]
Then~\eqref{eqn:saga-one-step-reduction} implies
\begin{equation}\label{eqn:saga-Delta-converge}
\E\left[\Dt\right] \leq \left(1-\frac{1}{3\Gamma}\right)^{\! t} \Dini,
\end{equation}
where the expectation is taken with respect to all random variables 
generated by Algorithm~\ref{alg:dscovr-saga} up to iteration~$t$.

By the definition of $\xi_{ik}$ in~\eqref{eqn:xi-zeta}, we have
\[
  \frac{\xi_{ik}}{\Gamma}
  = \frac{3\sigma_i\|\Xik\|^2}{m p_i q_k^2}\frac{1}{\Gamma} 
  \leq \frac{\theta\lambda}{mp_i q_k} \frac{1}{\Gamma}
  \leq \frac{\theta\lambda}{m}
  = \frac{\lambda}{3m},
\]
where the first inequality is due to~\eqref{eqn:theta-lambda-theta-gamma}
and the second inequality is due to the relation $\Gamma\geq\frac{1}{p_i q_k}$
from the definition of~$\Gamma$ in~\eqref{eqn:Gamma-def-theta-saga}.
Similarly, we have
\[
  \frac{\zeta_{ik}}{\Gamma}
  = \frac{3\tau_k\|\Xik\|^2}{m^2 p_i^2 q_k}\frac{1}{\Gamma} 
  \leq \frac{\theta \gamma_i}{m n p_i q_k} \frac{1}{\Gamma}
  \leq \frac{\theta\gamma_i}{3mn}
  = \frac{\gamma_i}{3mn}.
\]
Moreover, by the construction in~\eqref{eqn:Wt-recursion}
and~\eqref{eqn:At-recursion}, we have for $t=0$,
\begin{eqnarray*}
\Aikini &=& \aiini, \quad\textrm{for}\quad k=1,\ldots,n 
  \quad\textrm{and}\quad i=1,\ldots,m, \\
  \bigl(\Wikini\bigr)^T &=& \wkini, \quad\textrm{for}\quad i=1,\ldots,m
  \quad\textrm{and}\quad k=1,\ldots,n. 
\end{eqnarray*}
Therefore, the last two terms in the definition of $\Dini$ can be bounded as
\begin{eqnarray*}
&& \sum_{i=1}^m\sum_{k=1}^n \frac{\zeta_{ik}}{\Gamma}
\bigl\|\Aikini-\aiopt\bigr\|^2
+ \sum_{i=1}^m\sum_{k=1}^n \frac{\xi_{ik}}{\Gamma}
\bigl\|\bigl(\Wikini\bigr)^T-\wkopt\bigr\|^2 \\
&\leq& \sum_{i=1}^m\sum_{k=1}^n \frac{\gamma_i}{3mn}
\bigl\|\aiini-\aiopt\bigr\|^2
+ \sum_{i=1}^m\sum_{k=1}^n \frac{\lambda}{3m}
\bigl\|\wkini-\wkopt\bigr\|^2 \\
&=& \frac{1}{3m}\sum_{i=1}^m \gamma_i \bigl\|\aiini-\aiopt\bigr\|^2
+ \frac{\lambda}{3} \bigl\|\wini-\wopt\bigr\|^2 ,
\end{eqnarray*}
which implies
\[
\Dini \leq \frac{4}{3}\left( \lambda\bigl\|\wini-\wopt\bigr\|^2
+\frac{1}{m}\sum_{i=1}^m \gamma_i \bigl\|\aiini-\aiopt\bigr\|^2 \right).
\]
Finally, combining with~\eqref{eqn:saga-Delta-converge}, we have
\[
\E\left[\Dt\right] \leq \left(1-\frac{1}{3\Gamma}\right)^{t} 
\frac{4}{3} \Biggl( \lambda\bigl\|\wini-\wopt\bigr\|^2
+\frac{1}{m}\sum_{i=1}^m \gamma_i\bigl\|\aiini-\aiopt\bigr\|^2 \Biggr),
\]
which further implies the desired result.

\section{Proof of Theorem~\ref{thm:dscovr-accl}}
\label{sec:proof-accl}

To simplify the presentation, we present the proof for the case
$\gamma_i=\gamma$ for all $i=1,\ldots,m$.
It is straightforward to generalize to the case
where the $\gamma_i$'s are different.

\begin{lemma}\label{lem:ppa-contraction}
Let $g:\R^D\to\R$ be $\lambda$-strongly convex, 
and $f_i^*:\R^{N_i}\to\R\cup\{\infty\}$ 
be $\gamma$-strongly convex over its domain.
Given any $\wtld\in\R^{d}$ and $\atld\in\R^N$,
we define the following two functions:
\begin{eqnarray}
\Lagr(w,\alpha) &=& g(w) + \frac{1}{m}\alpha^T X w 
-\frac{1}{m}\sum_{i=1}^m f_i*(\alpha_i),
\label{eqn:simple-Lagrangian} \\
\Lagr_\delta(w,\alpha) 
&=& \Lagr(w,\alpha) + \frac{\delta\lambda}{2}\|w-\wtld\|^2 
- \frac{\delta\gamma}{2m}\|\alpha-\atld\|^2 .
\label{eqn:proximal-Lagrangian}
\end{eqnarray}
Let $(\wopt,\aopt)$ and $(\wtldopt,\atldopt)$ be the (unique) saddle points
of $\Lagr(w,\alpha)$ and $\Lagr_\delta(w,\alpha)$, respectively.
Then we have
\begin{eqnarray}
\lambda\|\wtld-\wtldopt\|^2 +\frac{\gamma}{m}\|\atld-\atldopt\|^2
&\leq&\lambda\|\wtld-\wopt\|^2 +\frac{\gamma}{m}\|\atld-\aopt\|^2, 
\label{eqn:ppa-tldopt-close}\\
\left( \lambda\|\wtldopt-\wopt\|^2
+\frac{\gamma}{m}\|\atldopt-\aopt\|^2\right)^{1/2}
&\leq& \frac{\delta}{1+\delta}
\left(\lambda\|\wtld-\wopt\|^2+\frac{\gamma}{m}\|\atld-\aopt\|^2\right)^{1/2}.
\label{eqn:ppa-tldopt-contract}
\end{eqnarray}
\end{lemma}

\begin{proof}
This lemma can be proved using the theory of monotone operators
\citep[e.g.,][]{Rockafellar76PPA,RyuBoyd2016},
as done by \citet{BalamuruganBach2016}.
Here we give an elementary proof based on first-order optimality conditions.

By assumption, we have
\begin{eqnarray*}
(\wopt,\aopt) &=& \arg\,\min_{w}\,\max_{\alpha} \Lagr(w,\alpha), \\
(\wtldopt,\atldopt) &=& \arg\,\min_{w}\,\max_{\alpha} \Lagr_\delta(w,\alpha).
\end{eqnarray*}
Optimality conditions for $(\wtldopt,\atldopt)$ as a saddle point of
$\Lagr_\delta$:
\begin{eqnarray}
-\frac{1}{m}X^T\atldopt - \delta\lambda\left(\wtldopt-\wtld\right)
&\in& \partial g(\wtldopt),  \label{eqn:g-wtldopt-subg}\\
X\wtldopt - \delta\gamma\left(\atldopt-\atld\right)
&\in& \partial \sum_{i=1}^m f_i^*(\atldopt). \label{eqn:Psi-atldopt-subg}
\end{eqnarray}
For any $\xi\in\partial g(\wtldopt)$, it holds that
$\xi + \frac{1}{m} X^T \aopt \in \partial_w \Lagr(\wtldopt,\aopt)$. 
Therefore using~\eqref{eqn:g-wtldopt-subg} we have
\[
\frac{1}{m}X^T(\aopt-\atldopt) - \delta\lambda\left(\wtldopt-\wtld\right)
~\in~ \partial_w \Lagr(\wtldopt, \aopt).
\]
Since $\Lagr(w,\aopt)$ is strongly convex in~$w$ with convexity 
parameter~$\lambda$, we have
\begin{equation}\label{eqn:ppa-w-sc}
\Lagr(\wtldopt,\aopt) +\left(
\frac{1}{m}X^T(\aopt-\atldopt) - \delta\lambda\left(\wtldopt-\wtld\right)
\right)^T (\wopt-\wtldopt) + \frac{\lambda}{2}\|\wtldopt-\wopt\|^2
\leq \Lagr(\wopt,\aopt).
\end{equation}
Similarly, we have 
\[
\frac{1}{m}X^T(\wtldopt-\wopt)-\frac{\delta\gamma}{m}\left(\atldopt-\atld\right)
~\in~ \partial_\alpha \left(-\Lagr(\wopt, \atldopt)\right),
\]
and since $-\Lagr(\wopt,\alpha)$ is strongly convex in~$\alpha$ with convexity 
parameter~$\frac{\gamma}{m}$, we have
\begin{equation}\label{eqn:ppa-a-sc}
-\Lagr(\wopt,\atldopt) +\left(
\frac{1}{m}X^T(\wtldopt-\wopt)-\frac{\delta\gamma}{m}\left(\atldopt-\atld\right)
\right)^T (\aopt-\atldopt) + \frac{\gamma}{2m}\|\atldopt-\aopt\|^2
\leq - \Lagr(\wopt,\aopt).
\end{equation}
Adding inequalities~\eqref{eqn:ppa-w-sc} and~\eqref{eqn:ppa-a-sc} together gives
\begin{eqnarray*}
&&\!\!\!  \Lagr(\wtldopt,\aopt)-\Lagr(\wopt,\atldopt)  \\
&+&\!\!\! \delta\lambda(\wtldopt-\wtld)^T(\wtldopt-\wopt)
+ \frac{\delta\gamma}{m}(\atldopt-\atld)^T(\atldopt-\aopt)
+ \frac{\lambda}{2}\|\wtldopt-\wopt\|^2
+ \frac{\gamma}{2m}\|\atldopt-\aopt\|^2
\leq 0.
\end{eqnarray*}
Combining with the inequality 
\[
\Lagr(\wtldopt,\aopt)-\Lagr(\wopt,\atldopt)
\geq\frac{\lambda}{2}\|\wtldopt-\wopt\|^2 
+\frac{\gamma}{2m}\|\atldopt-\aopt\|^2,
\]
we obtain
\begin{equation}\label{eqn:ppa-pd-norms}
\lambda\|\wtldopt-\wopt\|^2
+ \frac{\gamma}{m}\|\atldopt-\aopt\|^2
+ \delta\lambda(\wtldopt-\wtld)^T(\wtldopt-\wopt)
+ \frac{\delta\gamma}{m}(\atldopt-\atld)^T(\atldopt-\aopt)
\leq 0.
\end{equation}

\emph{Proof of the first claim.}
We can drop the nonnegative terms on the left-hand side 
of~\eqref{eqn:ppa-pd-norms} to obtain
\[
\lambda(\wtldopt-\wtld)^T(\wtldopt-\wopt)
+ \frac{\gamma}{m}(\atldopt-\atld)^T(\atldopt-\aopt)
\leq 0.
\]
The two inner product terms on the left-hand side of the inequality above
can be expanded as follows:
\begin{eqnarray*}
(\wtldopt-\wtld)^T(\wtldopt-\wopt)
&=& (\wtldopt-\wtld)^T(\wtldopt-\wtld+\wtld-\wopt)
~=~ \|\wtldopt-\wtld\|^2 + (\wtldopt-\wtld)^T(\wtld-\wopt)\,, \\
(\atldopt-\atld)^T(\atldopt-\aopt)
&=& (\atldopt-\atld)^T(\atldopt-\atld+\atld-\aopt)
~=~ \|\atldopt-\atld\|^2 + (\atldopt-\atld)^T(\atld-\aopt).
\end{eqnarray*}
Combining them with the last inequality, we have
\begin{eqnarray*}
\lambda\|\wtldopt-\wtld\|^2 +\frac{\gamma}{m}\|\atldopt-\atld\|^2
&\leq& - \lambda(\wtldopt-\wtld)^T(\wtld-\wopt) 
- \frac{\gamma}{m}(\atldopt-\atld)^T(\atld-\aopt) \\
&\leq&\frac{\lambda}{2}\left(\|\wtldopt-\wtld\|^2+\|\wtld-\wopt\|^2\right)
+\frac{\gamma}{2m}\left(\|\atldopt-\atld\|^2+\|\atld-\aopt\|^2\right),
\end{eqnarray*}
which implies
\[
\frac{\lambda}{2}\|\wtldopt-\wtld\|^2 +\frac{\gamma}{2m}\|\atldopt-\atld\|^2
~\leq~\frac{\lambda}{2}\|\wtld-\wopt\|^2 +\frac{\gamma}{2m}\|\atld-\aopt\|^2.
\]

\emph{Proof of the second claim.}
We expand the two inner product terms in~\eqref{eqn:ppa-pd-norms} as follows:
\begin{align*}
(\wtldopt-\wtld)^T(\wtldopt-\wopt)
&= (\wtldopt-\wopt+\wopt-\wtld)^T(\wtldopt-\wopt)
= \|\wtldopt-\wopt\|^2 + (\wopt-\wtld)^T(\wtldopt-\wopt), \\
(\atldopt-\atld)^T(\atldopt-\aopt)
&= (\atldopt-\aopt+\aopt-\atld)^T(\atldopt-\aopt)
= \|\atldopt-\aopt\|^2 + (\aopt-\atld)^T(\atldopt-\aopt).
\end{align*}
Then~\eqref{eqn:ppa-pd-norms} becomes
\begin{eqnarray*}
&& (1+\delta)\lambda\|\wtldopt-\wopt\|^2  
+(1+\delta)\frac{\Gamma}{m}\|\wtldopt-\wopt\|^2 \\
&\leq& \delta\lambda(\wtld-\wopt)^T(\wtldopt-\wopt)
+ \frac{\delta\gamma}{m}(\atld-\aopt)^T(\atldopt-\aopt) \\
&\leq& \delta 
\left( \lambda\|\wtld-\wopt\|^2+\frac{\gamma}{m}\|\atld-\aopt\|^2\right)^{1/2}
\left( \lambda\|\wtldopt-\wopt\|^2
	+\frac{\gamma}{m}\|\atldopt-\aopt\|^2\right)^{1/2},
\end{eqnarray*}
where in the second inequality we used the Cauchy-Schwarz inequality.
Therefore we have
\[
\left( \lambda\|\wtldopt-\wopt\|^2
+\frac{\gamma}{m}\|\atldopt-\aopt\|^2\right)^{1/2}
~\leq~ \frac{\delta}{1+\delta}
\left(\lambda\|\wtld-\wopt\|^2+\frac{\gamma}{m}\|\atld-\aopt\|^2\right)^{1/2},
\]
which is the desired result.
\end{proof}

To simplify notations in the rest of the proof, we let $z=(w,\alpha)$ 
and define
\[
\|z\| = \left(\lambda\|w\|^2+\frac{\gamma}{m}\|\alpha\|^2\right)^{1/2}.
\]
The results of Lemma~\ref{lem:ppa-contraction} can be written as
\begin{eqnarray}
\|\ztld-\ztldopt\| &\leq&\|\ztld-\zopt\|,
\label{eqn:ppa-z-close} \\
\|\ztldopt-\zopt\| &\leq&\frac{\delta}{1+\delta}\|\ztld-\zopt\|.
\label{eqn:ppa-z-contract}
\end{eqnarray}
Next consider the convergence of Algorithm~\ref{alg:accl-dscovr}, 
and follow the proof ideas in \citet[Section D.3]{BalamuruganBach2016}.

If we use DSCOVR-SVRG (option~1) in each round of 
Algorithm~\ref{alg:accl-dscovr},
then Algorithm~\ref{alg:dscovr-svrg} is called with initial point
$\ztldr=(\wtldr,\atldr)$ and after~$S$ stages, it outputs $\ztldrp$ 
as an approximate saddle point of $\Lagr^{(r)}_\delta(w,\alpha)$,
which is defined in~\eqref{eqn:Lagrangian-delta}.
Then Theorem~\ref{thm:dscovr-svrg} implies
\begin{equation}\label{eqn:ppa-svrg-rate}
\E\bigl[\|\ztldrp-\ztldoptr\|^2\bigr] \leq \left(\frac{2}{3}\right)^S
\E\bigl[\|\ztldr-\ztldoptr\|^2\bigr],
\end{equation}
where $\ztldoptr$ denotes the unique saddle point of
$\Lagr^{(r)}_\delta(w,\alpha)$.
By Minkowski's inequality, we have
\[
  \left(\E\bigl[\|\ztldrp-\zopt\|^2\bigr]\right)^{1/2} 
  \leq \left(\E\bigl[\|\ztldrp-\ztldoptr\|^2\bigr]\right)^{1/2} 
  + \left(\E\bigl[\|\ztldoptr-\zopt\|^2\bigr]\right)^{1/2},
\]
where $\zopt$ is the unique saddle point of $\Lagr(w,\alpha)$.
Using~\eqref{eqn:ppa-svrg-rate}, \eqref{eqn:ppa-z-close}
and~\eqref{eqn:ppa-z-contract}, we obtain
\begin{eqnarray}
\left(\E\bigl[\|\ztldrp-\zopt\|^2\bigr]\right)^{1/2}
&\leq& \left(\frac{2}{3}\right)^{S/2}
 \left(\E\bigl[\|\ztldr-\ztldoptr\|^2\bigr]\right)^{1/2}
 + \left(\E\bigl[\|\ztldoptr-\zopt\|^2\bigr]\right)^{1/2} \nonumber \\
&\leq& \left(\frac{2}{3}\right)^{S/2} 
 \left(\E\bigl[\|\ztldr-\zopt\|^2\bigr]\right)^{1/2}
 + \frac{\delta}{1+\delta} \left(\E\bigr[\|\ztldr-\zopt\|^2\bigr]\right)^{1/2}
\nonumber \\
&=& \left[\left(\frac{2}{3}\right)^{S/2} +\frac{\delta}{1+\delta}\right] 
\left(\E\bigl[\|\ztldr-\zopt\|^2\bigr]\right)^{1/2},
\label{eqn:ppa-round-svrg}
\end{eqnarray}
Therefore, if $S\geq\frac{2\log(2(1+\delta))}{\log(3/2)}$, we have
\[
\left(\frac{2}{3}\right)^{S/2} +\frac{\delta}{1+\delta}
\leq \frac{1}{2(1+\delta)} +\frac{\delta}{1+\delta}
= \frac{1+2\delta}{2(1+\delta)} = 1 - \frac{1}{2(1+\delta)},
\]
which implies
\begin{equation}\label{eqn:ppa-round-converge}
\E\bigl[\|\ztldrp-\zopt\|^2\bigr]
~\leq~ \left(1 - \frac{1}{2(1+\delta)}\right)^2
\E\bigl[\|\ztldr-\zopt\|^2\bigr].
\end{equation}

If we use DISCOVR-SAGA (option~2) in Algorithm~\ref{alg:accl-dscovr}, 
then Algorithm~\ref{alg:dscovr-saga} is called with initial point
$\ztldr=(\wtldr,\atldr)$ and after~$M$ steps, it outputs $\ztldrp$ 
as an approximate saddle point of $\Lagr^{(r)}_\delta(w,\alpha)$.
Then Theorem~\ref{thm:dscovr-saga} implies
\[
\E\bigl[\|\ztldrp-\ztldoptr\|^2\bigr] 
\leq  \frac{4}{3} \left(1-\frac{1}{3\Gamma_{\delta}}\right)^{M}
\E\bigl[\|\ztldr-\ztldoptr\|^2\bigr].
\]
Using similar arguments as in~\eqref{eqn:ppa-round-svrg}, we have
\[
\left(\E\bigl[\|\ztldrp - \zopt\|\bigr]^2\right)^{1/2} 
~\leq~ \left[\frac{4}{3}\left(1-\frac{1}{3\Gamma_{\delta}}\right)^{M/2}
+\frac{\delta}{1+\delta}\right] 
\left(\E\bigl[\|\ztldr-\zopt\|^2\bigr]\right)^{1/2}.
\]
Therefore, 
if $M\geq 6\log\left(\frac{8(1+\delta)}{3}\right) \Gamma_{\delta}$, we have
\[
\frac{4}{3}\left(1-\frac{1}{3\Gamma_{\delta}}\right)^{M/2}
+\frac{\delta}{1+\delta}
\leq \frac{1}{2(1+\delta)} +\frac{\delta}{1+\delta}
= \frac{1+2\delta}{2(1+\delta)} = 1 - \frac{1}{2(1+\delta)},
\]
which implies the same inequality in~\eqref{eqn:ppa-round-converge}.

In summary, using either option~1 or option~2 
in Algorithm~\ref{alg:accl-dscovr}, we have
\[
\E\bigl[\|\ztldr-\zopt\|^2\bigr]
~\leq~ \left(1 - \frac{1}{2(1+\delta)}\right)^{2r}
\|\ztldini-\zopt\|^2.
\]
In order to have $\E\bigl[\|\ztldr-\zopt\|^2\bigr]\leq\epsilon$, 
it suffices to have 
$r\geq (1+\delta)\log\left(\|\ztldini-\zopt\|^2/\epsilon\right)$.

\bibliography{dscovr}

\end{document}